\documentclass[11pt]{amsart}
\usepackage{amsmath}
\usepackage{amssymb}
\usepackage{amsthm}
\usepackage{latexsym}
\usepackage{hyperref}
\usepackage{enumerate}
\usepackage[all]{xy}

\newtheorem{thm}{Theorem}[section]
\newtheorem{cor}[thm]{Corollary}
\newtheorem{lem}[thm]{Lemma}
\newtheorem{prop}[thm]{Proposition}

\newtheorem{thmintro}{Theorem}

\theoremstyle{definition}
\newtheorem{defn}[thm]{Definition}

\newtheorem{ex}[thm]{Example}

\providecommand{\norm}[1]{\left\| #1 \right\|}
\newcommand{\enuma}[1]{\begin{enumerate}[\textup{(}a\textup{)}] {#1} \end{enumerate}}

\newcommand{\mh}{\mathbb}
\newcommand{\mr}{\mathrm}
\newcommand{\mc}{\mathcal}
\newcommand{\mf}{\mathfrak}

\newcommand{\N}{\mathbb N}
\newcommand{\Z}{\mathbb Z}

\newcommand{\R}{\mathbb R}
\newcommand{\C}{\mathbb C}

\newcommand{\es}{\emptyset}

\newcommand{\red}{\mr{red}}
\newcommand{\cpt}{\mr{cpt}}

\newcommand{\inp}[2]{\langle #1 \,,\, #2 \rangle}

\def\Hom{{\rm Hom}}

\def\Rep{{\rm Rep}}

\begin{document}

\title[Quasi-reductive groups over local fields]{
Pseudo-reductive and quasi-reductive groups \\ over non-archimedean local fields}

\date{\today}
\maketitle

\begin{center}
{\Large Maarten Solleveld} \\
IMAPP, Radboud Universiteit \\
Heyendaalseweg 135, 6525AJ Nijmegen, the Netherlands \\
email: m.solleveld@science.ru.nl 
\end{center}

\begin{abstract}
Among connected linear algebraic groups, quasi-reductive groups generalize 
pseudo-reductive groups, which in turn form a useful relaxation of the notion of 
reductivity. We study quasi-reductive groups over non-archimedean local fields, 
focusing on aspects involving their locally compact topology.

For such groups we construct valuated root data (in the sense of Bruhat--Tits) and we make 
them act nicely on affine buildings. We prove that they admit Iwasawa and Cartan 
decompositions, and we construct small compact open subgroups with an Iwahori decomposition.

We also initiate the smooth representation theory of quasi-reductive groups. Among others,
we show that their irreducible smooth representations are uniformly admissible, and that 
all these groups are of type I.

Finally we discuss how much of these results remains valid if we omit the connectedness
assumption on our linear algebraic groups.
\end{abstract}

\tableofcontents

\newpage

\section*{Introduction}
\noindent
Pseudo-reductive and quasi-reductive groups generalize connected reductive groups. These 
classes of linear algebraic groups were already known to Tits \cite{BoTi} and Springer 
\cite{Spr}, who developed some basic theory. The recent work of Conrad, Gabber and Prasad
\cite{CGP,CP,CP2} has revived the interest in these groups. 

These sources are mainly
concerned with the structure of pseudo-reductive groups over arbitrary or separably closed
fields. Properties which involve the locally compact topology of pseudo-reductive or 
quasi-reductive groups over local fields have been investigated far less (apart from
reductive groups of course). With this paper we try to narrow that gap. \\

Let $\mc G$ be a connected linear algebraic group defined over a field $F$. The $F$-unipotent 
radical $\mc R_{u,F}(\mc G)$ of $\mc G$ is the largest connected normal unipotent $F$-subgroup 
of $\mc G$ (so it is contained in the usual unipotent radical of $\mc G$). By definition,
$\mc G$ is pseudo-reductive as $F$-group if $\mc R_{u,F}(\mc G) = 1$.

Over a field of characteristic zero, every connected linear algebraic group $\mc G$ admits 
a Levi decomposition, that is, it can be written as the semidirect product of its unipotent 
radical and a reductive subgroup. But this is not always the case over fields $F$ of positive 
characteristic $p$. Firstly, suitable Levi factors need not exist, even if $F$ is algebraically
closed \cite[\S A.6]{CGP}. Secondly, the unipotent radical need not be defined over $F$. 
For example, suppose that $F' / F$ is an inseparable field extension 
of degree $p$ and that $\mc G'$ is a nontrivial connected reductive $F'$-group. Then 
restriction of scalars yields a pseudo-reductive $F$-group $R_{F'/F}(\mc G')$ which is not 
reductive \cite[Proposition 1.1.10]{CGP}.

Nevertheless, $\mc G$ is always an extension of a pseudo-reductive $F$-group by a
unipotent $F$-group. Namely, in the short exact sequence
\begin{equation}\label{eq:1}
1 \to \mc R_{u,F}(\mc G) \to \mc G \to \mc G / \mc R_{u,F}(\mc G) \to 1 
\end{equation}
the quotient group $\mc G / \mc R_{u,F}(\mc G)$ is easily seen to be pseudo-reductive over $F$.
Thus one can try to understand connected linear algebraic $F$-groups in terms of pseudo-reductive
$F$-groups and unipotent $F$-groups. This is considerably easier over perfect fields,
for then every pseudo-reductive $F$-group is in fact reductive and every connected unipotent 
group is $F$-split.

Over a general field $F$, every connected unipotent group is an extension of a $F$-wound unipotent 
group by a $F$-split unipotent group. Such $F$-wound unipotent groups were studied deeply
by Tits (see \cite{Oes} and \cite[Appendix B]{CGP}). Although their structure is described
quite well, it is safe to say that they are much more complicated than split unipotent groups.\\

Every connected linear algebraic $F$-group $\mc G$ has a maximal normal $F$-split unipotent 
subgroup, its $F$-split unipotent radical $\mc R_{us,F}(\mc G)$. As in \cite[\S 1.1.12]{BrTi2}, we 
say that $\mc G$ is quasi-reductive over $F$ if $\mc R_{us,F}(\mc G) = 1$. Since $\mc R_{u,F}(\mc G) 
\supset \mc R_{us,F}(\mc G)$, every pseudo-reductive $F$-group is also quasi-reductive. If $F$ is 
perfect, then every connected unipotent $F$-group is $F$-split, so every quasi-reductive 
$F$-group is also pseudo-reductive and in fact reductive. When $F$ is not perfect, there do exist
quasi-reductive $F$-groups that are not pseudo-reductive, for instance $F$-wound unipotent groups.

In the short exact sequence 
\begin{equation}\label{eq:2}
1 \to \mc R_{us,F}(\mc G) \to \mc G \to \mc G / \mc R_{us,F}(\mc G) \to 1 
\end{equation}
the quotient $\mc G / \mc R_{us,F}(\mc G)$ is always quasi-reductive \cite[Corollary B.3.5]{CGP}. 
Thus, as an alternative to \eqref{eq:1}, one can try analyse connected linear algebraic groups in 
terms of quasi-reductive groups and split unipotent groups. An advantage of doing this with 
quasi-reductive groups (as compared to pseudo-reductive groups) is that split unipotent groups are 
much easier than non-split unipotent groups, and that \eqref{eq:2} always splits as a sequence of 
$F$-varieties \cite[Theorem 14.2.6]{Spr}. For every quasi-reductive $F$-group $\mc G$, the 
$F$-unipotent radical $\mc R_{u,F}(\mc G)$ is $F$-wound. Thus \eqref{eq:1} shows that every 
quasi-reductive $F$-group is an extension of a pseudo-reductive $F$-group by a $F$-wound 
unipotent group.\\

Now we consider a local field $F$ and a quasi-reductive $F$-group $\mc G$. If $F$ is archimedean,
then $\mc G (F)$ is a real reductive group -- an object which we do not investigate in this
paper. When $F$ is non-archimedean and has characteristic zero, $\mc G (F)$ is reductive
$p$-adic group. The structure of such groups, and more generally of reductive groups over 
discretely valued fields, was studied deeply by Bruhat and Tits \cite{BrTi1,BrTi2,Tit}. 

The most interesting case for us arises when $F$ is a local field of positive characteristic
and $\mc G$ is quasi-reductive but not reductive. From the work of Borel and Tits (see \cite{BoTi}
and \cite[Appendix C.2]{CGP}) it is known that $\mc G$ shares many properties with reductive groups. 
Summarizing: there exists a maximal $F$-split torus in $\mc G$, an associated root system, a Weyl 
group and root subgroups, which are $F$-split unipotent. There are pseudo-parabolic subgroups with 
similar properties as parabolic subgroups of reductive groups, and the Bruhat decomposition holds
for $\mc G (F)$.\\

Since $F$ is a local field, we can also consider $G = \mc G (F)$ with the topology coming from 
the metric on $F$. Then it becomes a locally compact, totally disconnected, unimodular group (Lemma 
\ref{lem:1.4}). We investigate several properties of $G$ involving this locally compact topology.

\begin{thmintro}\label{thm:1} (see Theorem \ref{thm:3.3} and Propostion \ref{prop:3.17}) \\
$G$ has a generating root datum with a prolonged valuation in the sense of \cite{BrTi1}. 
\end{thmintro}
In fact it was already shown in \cite[\S C.2.28]{CGP} that $G$ has a generating root datum,
we provide a prolonged valuation thereof.

\begin{thmintro}\label{thm:2} (see Theorem \ref{thm:3.4}) \\
To $\mc G (F)$ we can associate two affine buildings: a thick building $\mc{BT}(\mc G,F)$
coming from a double Tits system in $\mc G (F)$ and an extended building $\mc B (\mc G,F)$,
which is the direct product of $\mc{BT}(\mc G,F)$ and a vector space. 
Moreover $\mc B (\mc G,F)$ is a universal space for proper $\mc G (F)$-actions.
\end{thmintro}

\begin{thmintro}\label{thm:3} (see Theorem \ref{thm:3.10}) \\
$G$ has arbitrarily small compact open subgroups $K$ that admit an Iwahori decomposition
with respect to a given maximal $F$-split torus $S$ in $G$.
\end{thmintro}
More explicitly, let $P = MU, \bar P = M \bar U$ be a pair of opposite pseudo-parabolic 
$F$-subgroups of $G$ containing $Z_G (S)$. Then the multiplication map
\[
(K \cap U) \times (K \cap M) \times (K \cap \bar U) \to K 
\]
is bijective.

\begin{thmintro}\label{thm:4} (see Theorem \ref{thm:3.6}) \\
$G$ admits Iwasawa decompositions. More precisely, let $x$ be a special vertex of 
$\mc B (\mc G,F)$. Then the isotropy group $G_x$ is a maximal compact subgroup of $G$ and 
$G = P G_x$ for every pseudo-parabolic $F$-subgroup $P$ of $G$.
\end{thmintro}

\begin{thmintro}\label{thm:5} (see Theorem \ref{thm:3.7}) \\
$G$ admits Cartan decompositions. More precisely, let $S$ be a maximal $F$-split torus in 
$G$ and let $x$ be a special vertex of the apartment of $\mc B (\mc G,F)$ associated to $S$. 
Then there exists a finitely generated semigroup $A \subset Z_G (S)$ such that
$G = \bigsqcup_{a \in A} G_x a G_x$.
\end{thmintro}

In the course of proving the Cartan decomposition, we also establish an Iwahori--Bruhat
decomposition. That is, we show that the double cosets with respect to a kind of Iwahori
subgroup of $G$ are in bijection with a suitable (extended) affine Weyl group.\\

The upshot of all the above geometric results is that quasi-reductive $F$-groups are actually not 
so different from reductive groups. The main difference lies in the structure of Cartan subgroups 
and ``$F$-Cartan subgroups", i.e. the centralizers of maximal $F$-split tori in $\mc G$. 
Upon dividing out its center, such a $F$-Cartan subgroup becomes $F$-anisotropic (see Lemma 
\ref{lem:3.16}). By \cite[Proposition A.5.7]{Con} the $F$-rational points of any $F$-anisotropic 
group form a compact group, and conversely. Thus the $F$-rational points of any $F$-Cartan 
subgroup of $\mc G$ form a group which is compact modulo its center. But, whereas the structure 
of $F$-tori and of anisotropic reductive $F$-groups is understood very well, this is much less 
the case for commutative or anisotropic pseudo- or quasi-reductive $F$-groups.

With this in mind we start to investigate smooth complex representation of quasi-reductive 
groups. In the representation theory of reductive $p$-adic groups $F$-Cartan subgroups are
often treated as black boxes, in the sense that one only uses that they are compact modulo
a central torus. Therefore most of the elementary representation theory of reductive $F$-groups
should remain valid for quasi-reductive groups. In particular, almost everything in the
influential preprint \cite{Cas} should hold for quasi-reductive groups. Using Renard's treatment 
\cite{Ren} of Bernstein's work as main reference, we check that several well-known results
involving parabolic induction generalize to quasi-reductive groups. Our main result about
$G$-representations is uniform admissibility:

\begin{thmintro} \label{thm:6} (see Theorem \ref{thm:1.1}) \\
Let $K$ be a compact open subgroup of $G$. There exists a bound $N(G,K) \in \N$ such that
$\dim_\C (V^K) \leq N(G,K)$ for every irreducible smooth $G$-representation $V$. The same
holds when $V$ is a topologically irreducible unitary $G$-representation.
\end{thmintro}
As a consequence, every quasi-reductive group over a non-archimedean local field has type I
(Corollary \ref{cor:1.3}). 

Finally, we check in Section \ref{sec:disc} which of the above results remain valid for
disconnected linear algebraic $F$-groups with a quasi-reductive connected component. This 
turns out to be a nontrivial issue, because such groups do not always possess good maximal
compact subgroups. But under mild conditions almost everything we did generalizes. \\

Several of our proofs are much easier for pseudo-reductive groups than for 
quasi-reductive groups. In the pseudo-reductive case one can show Theorems 
\ref{thm:2} -- \ref{thm:5} without using Theorem \ref{thm:1}. The main technique, relying 
heavily on \cite{CGP}, is reduction to the cases of semisimple groups and of commutative 
pseudo-reductive groups. We put our results and arguments for pseudo-reductive groups 
in Section \ref{sec:pseudo}. Let us note here that our results about parabolic induction
and restriction for pseudo-reductive groups (see Paragraph \ref{par:ind}) rely only
on Section \ref{sec:pseudo}.

Only in Section \ref{sec:quasi} we turn to quasi-reductive groups.
A difficulty in the generalization from pseudo-reductive to quasi-reductive arises from
taking $F$-rational points in \eqref{eq:1}. The resulting sequence
\[
1 \to \mc R_{u,F}(\mc G)(F) \to \mc G (F) \to (\mc G / \mc R_{u,F}(\mc G))(F)
\]
need not be exact at the right hand side, even when $\mc G$ is quasi-reductive. We analyse
the failure of surjectivity in Theorem \ref{thm:3.2}, showing that the image of the 
quasi-reductive group $\mc G (F)$ in the pseudo-reductive group $(\mc G / \mc R_{u,F}(\mc G))(F)$
has finite index and contains several important subgroups. But still, the non-surjectivity 
makes it hard to derive certain
results for quasi-reductive groups from those for pseudo-reductive groups. To overcome this,
we feel forced to appeal to Bruhat--Tits theory. Obviously this approach is rather technical, 
but once we have established the setup with valuated root data, we get many beautiful results
quite easily. Moreover, even for pseudo-reductive groups we can produce more precise results
than before.\\

One motivation for this paper was our desire to prove the Baum--Connes conjecture \cite{BCH}
for pseudo-reductive and quasi-reductive groups over local function fields. We prepare 
specifically for that in Paragraph \ref{par:Xi}. In joint work with K. Li (M\"unster)
we intend to use results from \cite{Laf} and from this paper to verify that conjecture 
for all linear algebraic groups over non-archimedean local fields, thus complementing
the results of \cite{CEN}.

\vspace{3mm} \textbf{Acknowledgements.} \\
The author is supported by a NWO Vidi grant ``A Hecke algebra approach to the local
Langlands correspondence" (nr. 639.032.528). 
We thank Brian Conrad and Gopal Prasad for lot of useful advice, which lead to substantial 
improvement of the paper. The author enjoyed some enlightening discussions with Kang Li 
about this paper and related topics.

\newpage

\section{Notations and preliminaries}

Our default field is called $F$.  Our algebraic groups are assumed to be smooth and affine, 
unless explicitly stated otherwise. When $\mc G$ is an algebraic group defined over $F$, 
we will often denote its group of $F$-rational points $\mc G (F)$ by $G$ (and similarly 
$\mc T (F) = T$ etcetera). When $F'/F$ is a finite field extension, we denote the Weil 
restriction functor by $R_{F' / F}$. Thus $R_{F'/F}(\mc G)$ is a linear algebraic $F$-group
with $\mc R_{F'/F}(\mc G)(F) = \mc G (F')$.

We denote the derived group of $\mc G$ by $\mc D (\mc G)$.
The algebraic groups $\mc R_{u,F}(\mc G)$ and $\mc R_{us,F}(\mc G)$ are by definition the 
$F$-unipotent radical and the $F$-split unipotent radical of $\mc G$. We say that $\mc G$ is
$F$-anisotropic if it contains neither the additive group $\mc G_a$
nor the multiplicative group $\mc G_m$ as an $F$-subgroup.

Starting from Paragraph \ref{par:local}, $F$ will be local and non-archimedean.
Then the group $\mc G (F)$ is endowed with the topology coming from the metric on $F$. 
A phrase like ``$S$ is a maximal $F$-split torus of $G$" means that $S = \mc S (F)$, where 
$\mc S$ is a maximal $F$-split torus in $\mc G$.
By $Z (\mc G)$ we mean the scheme-theoretic center of $\mc G$, in contrast to $Z(G)$, 
which is the center of $G$ as an abstract group. Similarly $Z_{\mc G}(X)$ will be the
scheme-theoretic centralizer of $X$ in $\mc G$, whereas $Z_G (X)$ denotes the ordinary
centralizer of $X$ in $G$.

By an affine building we mean those buildings that appear in Bruhat--Tits theory for
reductive groups over discretely valued fields \cite{Tit}. In other words, every affine 
building is a direct product of Euclidean/affine buildings that are simplicial complexes.
Typical instances of the latter are the real line and the building coming from an 
irreducible affine Tits system \cite[\S 2]{BrTi1}. All actions of topological groups on 
topological spaces (e.g. buildings) are assumed to be continuous.\\

Let us recall here a few results which do not have a natural place elsewhere 
in the paper, as they are valid in far greater generality. We will use them 
repeatedly in Paragraph \ref{par:local}.

\begin{prop}\label{prop:1.8} \textup{\cite[Proposition A.5.7]{Con}} \ \\
Let $\mc G$ be a connected linear algebraic group defined over a local field $F$.
Then $\mc G$ is $F$-anisotropic if and only if $\mc G (F)$ is compact.
\end{prop}

\begin{prop}\label{prop:1.7}
Let $\mc G$ be a connected linear algebraic group defined over a local field $F$, 
and let $\mc N$ be a normal $F$-subgroup.
\enuma{
\item The map $\mc G (F) \to (\mc G / \mc N)(F)$ is a submersion of $F$-analytic manifolds.
In particular this map is open and smooth.
\item The map $\mc G \to \mc G / \mc N$ sends maximal $F$-split tori onto maximal
$F$-split tori.
\item Suppose moreover that $\mc G / \mc N$ is pseudo-reductive or quasi-reductive over $F$.
Then the image of $\mc G (F)$ in $(\mc G / \mc N)(F)$ is an open subgroup of finite index.
}
\end{prop}
\begin{proof}
(a) According to \cite[Proposition 5.5.10 and Corollary 12.2.2]{Spr}, $\mc G / \mc N$ is 
again a connected linear algebraic $F$-group. By \cite[Corollary 5.5.4]{Spr}
the map $\mc G \to \mc G / \mc N$ is a separable morphism of $F$-varieties, and by
\cite[Theorem 4.3.7]{Spr} its differential at any point of $\mc G$ is surjective. 
Hence it is an $F$-analytic submersion.\\
(b) Let $\mc S$ be a maximal $F$-split torus in $\mc G$. Its image $\mc S_N$ in $\mc G / 
\mc N$ is a maximal $F$-split torus over there \cite[Theorem 22.6.ii]{Bor}. We note that 
by part (a) this result works when $F$ is not perfect -- see also \cite[Lemma C.2.31]{CGP}.\\
(c) By part (a) the image of $\mc G (F)$ in $(\mc G / \mc N)(F)$ is open.
Let $\mc S$ and $\mc S_N$ be as above. By \cite[Lemma 4.1.2.i]{Con} the image of the map 
$\mc S (F) \to \mc S_N (F)$ is open and has finite index. 
These facts allow us to conclude with \cite[Proposition 4.1.9]{Con}.
\end{proof}

For any smooth linear algebraic group $\mc G$, the modulus character 
$\delta_{\mc G} : \mc G \to GL_1$ is defined as the determinant of the adjoint representation 
of $\mc G$ on its Lie algebra. Equivalently, $\delta_{\mc G}$ describes the action of $\mc G$ 
on the one-dimensional space of top-degree left-invariant differential forms on $\mc G$.

The following proof was kindly communicated to the author by Gopal Prasad.

\begin{lem}\label{lem:1.4}
The modulus character of any pseudo-reductive or quasi-reductive $F$-group is trivial.
\end{lem}
\begin{proof}
Let $F_s$ be a separable closure of $F$. By \cite[Proposition 1.1.9]{CGP}
\[
\mc R_{u,F_s}(F_s \otimes_F \mc G) = F_s \otimes_F \mc R_{u,F}(\mc G) ,
\]
and by \cite[Proposition B.3.2]{CGP} right hand side is $F_s$-wound. Hence
$F_s \otimes_F \mc G$ is quasi-reductive, and we may assume that $F$ is separably closed.

Let $\mc T$ be any maximal $F$-torus in $\mc G$, and let $\mc C = Z_{\mc G}(\mc T)$ be the
associated Cartan $F$-subgroup of $\mc G$. By \cite[Proposition 1.2.6]{CGP}
$\mc G = \mc C \mc D (\mc G)$. Clearly the algebraic character $\delta_{\mc G}$ is
trivial on the derived group $\mc D (G)$, so it is enough to prove that $\delta_{\mc G}$ 
is trivial on $\mc C$. 

Since $\mc G$ is quasi-reductive, the decomposition of its Lie algebra as 
$\mc T$-representation yields a root system $\Phi$ \cite[Theorem C.2.15]{CGP}.
For any root $\alpha \in \Phi$, it is shown in \cite[pp. 601--602]{CGP} that 
$N_{\mc G}(\mc T) (F)$ contains an element $n_\alpha$ which acts on $\Phi (\mc G,\mc T)$
as the reflection associated to $\alpha$. In particular conjugation by $n_\alpha$ 
interchanges the root subgroups $\mc U_\alpha$ and $\mc U_{-\alpha}$, which therefore have
the same dimension. For $t \in \mc T (F)$ one computes
\begin{align*}
\delta_{\mc G}(t) & = \det \big( \mr{Ad}(t) : \mr{Lie}(\mc G) \to \mr{Lie}(\mc G) \big) \\
 & = \prod\nolimits_{\alpha \in \Phi} \det \big( \mr{Ad}(t) : 
 \mr{Lie}(\mc U_\alpha) \to \mr{Lie}(\mc U_\alpha) \big) \; = \; 
 \prod\nolimits_{\alpha \in \Phi} \alpha (t)^{\dim U_\alpha} \\
 & = \prod\nolimits_{\alpha \in \Phi / \{\pm 1\}} \alpha (t)^{\dim U_\alpha} 
 (-\alpha)(t)^{\dim U_{-\alpha}} \\
 & = \prod\nolimits_{\alpha \in \Phi / \{\pm 1\}} \alpha (t)^{\dim U_\alpha} 
 \alpha(t)^{- \dim U_\alpha} \hspace{10mm} = \; 1
\end{align*}
Thus $\delta_{\mc G}$ restricted to $\mc T$ is trivial and $\delta_{\mc G}$ restricted
to $\mc C$ factors through the group $\mc C / \mc T$. By the maximality of $\mc T$, 
$\mc C / \mc T$ has no nontrivial tori, which means that it is unipotent. But a unipotent
group does not admit nontrivial algebraic characters, so $\delta_{\mc G}$ is trivial
on $\mc C$.
\end{proof}

\newpage

\section{Pseudo-reductive groups}
\label{sec:pseudo}

Based on the work of Tits, Conrad, Gabber and Prasad, we give a rough description of the
structure of pseudo-reductive groups. In \cite{CGP} and \cite{CP} pseudo-reductive groups are 
classified in terms of reductive groups and commutative pseudo-reductive groups. 
Among them is a class of ``standard pseudo-reductive" groups, which turns out to exhaust all 
possibilities if char$(F) \notin \{2,3\}$. We now recall the standard construction.

Let $F_i$ be a finite extension of a field $F$. (The interesting case here is when $F_i / F$ 
is inseparable.) Let $\mc G_i$ be an absolutely simple and simply connected $F_i$-group, 
and let $\mc T_i$ be a maximal $F_i$-torus of $\mc G_i$. 
Then $R_{F_i / F}(\mc G_i)$ is a pseudo-reductive $F$-group \cite[Proposition 1.1.10]{CGP} 
and by \cite[Proposition A.5.15]{CGP}
\begin{equation}\label{eq:1.17}
R_{F_i/F}(\mc T_i) \text{ is the centralizer of a maximal }F\text{-torus in } R_{F_i / F}(\mc G_i) .
\end{equation}
Furthermore $R_{F_i / F}(\mc G_i)$ is reductive if and only if $F_i / F$ is separable.
Let $i$ run through a finite index set $I_s$ and put
\begin{equation}\label{eq:1.14}
\mc G' = \prod\nolimits_{i \in I_s} R_{F_i/F}(\mc G_i) ,\quad 
\mc T' = \prod\nolimits_{i \in I_s} R_{F_i/F}(\mc T_i) .
\end{equation}
Let $\mc C$ be a commutative pseudo-reductive $F$-group, endowed with $F$-homomorphisms
\begin{equation}\label{eq:1.1}
\mc T' \xrightarrow{\phi_{\mc C}} \mc C \xrightarrow{\psi_{\mc C}}
\prod\nolimits_{i \in I_s} R_{F_i / F} \big( \mc T_i  / Z (\mc G_i) \big)
\end{equation}
such that $\psi_{\mc C} \circ \phi_{\mc C} : \mc T' \to \prod\nolimits_{i \in I_s} R_{F_i / F} 
\big( \mc T_i  / Z (\mc G_i) \big)$ is the canonical map. (We warn that $\psi_{\mc C} \circ 
\phi_{\mc C}$ need not be surjective.) The group $\prod\nolimits_{i \in I_s} R_{F_i / F} 
\big( \mc T_i  / Z (\mc G_i) \big)$ acts on $\mc G'$, by $R_{F_i / F}$ applied to the conjugation
action of $\mc T_i / Z(\mc G_i)$ on $\mc G_i$. Hence we can use $\psi_{\mc C}$ 
to build a semi-direct product $\mc G' \rtimes \mc C$. The map
\[
\begin{array}{llll}
\alpha: & \mc T' & \to & \mc G' \rtimes \mc C , \\
 & t & \mapsto & (t^{-1}, \phi_{\mc C}(t))
\end{array}
\]
provides an embedding of $\mc T'$ as a central $F$-subgroup $\mc G' \rtimes \mc C$. Then
\begin{equation}\label{eq:1.27}
\mc G := ( \mc G' \rtimes \mc C ) / \alpha (\mc T)
\end{equation}
is a pseudo-reductive $F$-group \cite[Proposition 1.4.3]{CGP}. Every standard
pseudo-reductive group arises in this way from data $(\mc G', \prod_{i \in I_s} F_i / F,
\mc T', \mc C)$, which are essentially unique \cite[\S 4.2]{CGP}. 

From now on, $\mc G$ is any pseudo-reductive $F$-group and we assume that $[F : F^2] \leq 2$
if char$(F) = 2$. According to \cite[Theorem 10.2.1]{CGP}, $\mc G$ factors as
\begin{equation}\label{eq:2.1}
\mc G = \mc G_{nr} \times \mc G_r ,
\end{equation}
where $\mc G_{nr}$ is ``totally non-reduced" and $\mc G_r$ has a reduced root system (when
computed over a separable closure of $F$). 

By \cite[Theorem 10.2.1.(2)]{CGP} $\mc G_r$ admits a generalized standard presentation, similar
to \eqref{eq:1.27}. First let $\mc G_i, F_i$ and $\mc T_i$ be as above, so with $\mc G_i$ 
an absolutely simple and simply connected $F_i$-group and
$\mc T_i$ a maximal $F_i$-torus in $\mc G_i$. Next we allow more possible $\mc G_i$, indexed
by a new set $I_e$. Namely, they may also be ``basic exotic pseudo-reductive", as in 
\cite[Definition 7.2.6]{CGP}. The group $R_{F_i / F}(\mc G_i)$ is then called ``exotic 
pseudo-reductive". Such groups exist only if the characteristic $p$ of $F$ is 2 or 3. For every 
such index $i \in I_e$, let $\mc T_i$ be the centralizer of a maximal $F_i$-torus in $\mc G_i$. 
The construction in \cite[\S 7.2]{CGP} entails that $\mc T_i$ is commutative. Moreover, by
\cite[Propositions 7.1.5 and 7.3.3]{CGP} there exist a semisimple simply connected $F_i$-group 
$\overline{\mc G_i}$ and a canonical homomorphism of $F$-groups
\begin{equation}\label{eq:2.11}
(\mc R_{F_i/F} \, \mc G_i) (F) = \mc G_i (F_i) \; \longrightarrow \; \overline{\mc G_i}(F_i) =
(\mc R_{F_i/F} \, \overline{\mc G_i}) (F) ,
\end{equation}
which is a homeomorphism if $F_i$ is a local field. In that case \eqref{eq:2.11} provides 
a bijection between the collections of maximal $F$-split tori on both sides 
\cite[Corollary 7.3.4]{CGP}. Since every pseudo-parabolic $F$-subgroup is determined by a
cocharacter with values in a maximal $F$-split torus, \eqref{eq:2.11} also induces a bijection
between the pseudo-parabolic $F$-subgroups on both sides.

Using all these objects we put
\begin{equation}\label{eq:2.33}
\mc T' = \prod\nolimits_{i \in I_s \cup I_e} R_{F_i / F}(\mc T_i) ,\quad
\mc G'_r = \prod\nolimits_{i \in I_s \cup I_e} R_{F_i / F} (\mc G_i) .
\end{equation}
Here $\mc T'$ is a Cartan subgroup of $\mc G'_r$. Although this looks the same as in 
\eqref{eq:1.14}, there are more possibilities when char$(F) \in \{2,3\}$. The remainder of 
the construction of reduced pseudo-reductive group is the same as before, we write it down 
for reference. We need a commutative pseudo-reductive $F$-group $\mc C$ and $F$-homomorphisms
\begin{equation}\label{eq:2.4}
\mc T' \xrightarrow{\phi_{\mc C}} \mc C \xrightarrow{\psi_{\mc C}} 
\prod\nolimits_{i \in I_s \cup I_e} R_{F_i / F} (Z_{\mc G_i,\mc T_i}),
\end{equation}
where $Z_{\mc G_i,\mc T_i}$ is an algebraic group of automorphisms of $\mc G_i$ which restrict 
to the identity on $\mc T_i$ \cite[p. 427]{CGP}. With these we build the semidirect product
$\mc G'_r \rtimes \mc C$. Define 
\[
\alpha : \mc T' \to \mc T' \times \mc C , \quad t \mapsto (t^{-1},\phi_{\mc C}(t)).
\] 
Then $\mc T' \times \mc C = \alpha (\mc T') \times \mc C$ is the centralizer of $\mc T'$ in 
$\mc G'_r \rtimes \mc C$. Now we define
\begin{equation}\label{eq:2.5}
\mc G_r = (\mc G'_r \rtimes \mc C) / \alpha (\mc T') .
\end{equation}
Here the image of $\mc G'_r$ in $\mc G_r$ is the derived group $\mc D(\mc G_r)$ 
\cite[Proposition 4.1.4.(1) and Remark 10.1.11]{CGP}. Furthermore $\mc C$ embeds naturally in 
$\mc G_r$ as a Cartan subgroup, namely the centralizer of a maximal $F$-torus in $\mc C$. 

Totally non-reduced pseudo-reductive groups exist only when char$(F) = 2$. Their classification 
is much easier if $[F:F^2] = 2$, so we assume that now. By \cite[Proposition 10.1.4]{CGP}
\begin{equation}\label{eq:2.2}
\mc G_{nr} = \prod\nolimits_{i \in I_{nr}} R_{F'_i / F}(\mc G'_{nr,i}) ,
\end{equation}
where each $\mc G'_{nr,i}$ is a ``basic non-reduced pseudo-simple" group over a finite
extension $F'_i$ of $F$. By \cite[10.1.2 and Theorem 9.9.3]{CGP} there exist $n_i \in \N$ and
quadratic inseparable extensions $F''_i / F'_i$ such that the (solvable) radical 
$\mc R (\mc G'_{nr,i})$ is defined over $F''_i$ and
\[
\mc G'_{nr,i} / \mc R (\mc G'_{nr,i}) \cong Sp_{2 n_i} \text{ as } F''_i\text{-groups}.
\]
This gives rise to a canonical homomorphism of $F'_i$-groups
\begin{equation}\label{eq:2.3}
\mc G'_{nr,i} \to R_{F''_i / F'_i}(Sp_{2 n_i}) .
\end{equation}
If $[F'_i : (F'_i)^2] = 2$ then by \cite[Proposition 9.9.2]{CGP} the $F'_i$-group homomorphism 
\begin{equation}\label{eq:2.34}
\mc G'_{nr,i} (F'_i) \to Sp_{2 n_i} (F''_i) \text{ is bijective.}  
\end{equation}
Proposition \ref{prop:1.7} implies that \eqref{eq:2.34} induces a bijection between
the maximal $F$-split tori on both sides. Furthermore  \cite[Proposition 11.4.4]{CGP} says that 
\eqref{eq:2.3} provides a bijection between the pseudo-parabolic $F'_i$-subgroups of 
$\mc G'_{nr,i}$ and the pseudo-parabolic $F''_i$-subgroups of $Sp_{2 n_i}$.

The shape of a general pseudo-reductive $F$-group becomes \cite[Theorem 10.2.1]{CGP}
\begin{equation}\label{eq:2.6}
\mc G = \mc G_{nr} \times (\mc G'_r \rtimes \mc C) / \alpha (\mc T') = 
(\mc G_{nr} \times \mc G'_r \rtimes \mc C) / \alpha (\mc T') .
\end{equation}
Since $\mc G_{nr}$ is perfect \cite[Definition 10.1.1]{CGP}, the image of 
$\mc G_{nr} \times \mc G'_r$ in $\mc G$ equals the derived group $\mc D (\mc G)$.

\subsection{Structure over non-archimedean local fields} \
\label{par:local}

The above holds over any field, from now on we specialize to a non-archimedean local field $F$.
As a $F$-variety, $\mc G (F)$ is locally compact and totally disconnected. Hence it has a modular 
function $\delta_{\mc G (F)}$. We recall from \cite[p. 167--168]{PlRa} that the modular function 
of any linear algebraic $F$-group $\mc H (F)$ can be computed as the norm of its modulus character:
\begin{equation}\label{eq:1.35}
\delta_{\mc H (F)}(h) = \norm{ \det \big( \mr{Ad}(h) : 
\mr{Lie}(\mc H) \to \mr{Lie}(\mc H) \big) }_F . 
\end{equation}
In view of Lemma \ref{lem:1.4}, this means that $\mc G (F)$ is unimodular.

We write $\mc G' = \mc G_{nr} \times \mc G'_r$ and
\begin{equation}\label{eq:1.18}
G' \rtimes C = \mc G' (F) \rtimes \mc C (F) =
\prod\nolimits_{i \in I_s \cup I_e} \mc G_i (F_i) \rtimes \mc C (F) .
\end{equation}
By \cite[Propositions 4.1.4.(3) and 10.2.2.(3)]{CGP} all maximal $F$-tori of $\mc G'_r$ are equally 
good for the standard presentation, which means that we may choose any maximal $F_i$-torus 
$\mc T_i$ in $\mc G_i$. Thus we may and will assume that every $\mc T_i$ contains a maximal 
$F_i$-split torus of $\mc G_i$. 

Let $\mc S$ be the unique maximal $F$-split torus in $\mc C$. The image of $\mc S (F)$ in 
$\mc G_r (F)$ contains the image of $\mc T_i (F)$ in $\mc G_r (F)$, for each $i \in I_s \cup I_s$. 
By the assumption on $\mc T_i$ and by \cite[Proposition A.5.15]{CGP}, the image of $\mc S$ in 
$\mc G_r$ is a maximal $F$-split torus in there. Let $\mc S_\es$ be
a maximal $F$-split torus in $\mc G$ containing $\mc S$. By \eqref{eq:2.6} it factorizes as
\begin{equation}\label{eq:1.45}
\mc S_\es = (\mc S_\es \cap \mc G_{nr}) \times \mc S. 
\end{equation}
By \cite[Proposition 2.2.12]{CGP} the preimage of any maximal $F$-torus of $\mc G_r$ under
$\mc G'_r \rtimes \mc C \to \mc G_r$ contains a unique maximal $F$-torus of $\mc G'_r \rtimes 
\mc C$. Hence the preimage of $\mc S \subset \mc C \subset \mc G_r$ in $\mc G'_r \rtimes \mc C$ 
contains a unique maximal $F$-split torus, say $\mc S' \times \mc S$. Then
\begin{equation}\label{eq:1.46}
\mc S' \subset \mc G'_r \quad \text{and} \quad (\mc S_\es \cap \mc G_{nr}) \times 
(\mc S' \times \mc S) \; \subset \; \mc G_{nr} \times \mc G'_r \rtimes \mc C
\end{equation}
are maximal $F$-split tori. 

Recall that $[F:F^p ] = p$ for every non-archimedean local field of residual characteristic $p$. 
For all our purposes, i.e. for groups over local fields, \eqref{eq:2.11} shows that we may replace
every exotic pseudo-reductive group $\mc G_i$ by $\overline{\mc G_i}$ and $\mc T_i$ by its 
image $\overline{\mc T_i}$ in $\overline{\mc G_i}$. Then \cite[Theorem 1.3.9]{CGP} allows us 
to replace $Z_{\mc G_i, \mc T_i}$ in \eqref{eq:2.4} by $\overline{\mc T_i} / 
Z (\overline{\mc G_i})$, like in \eqref{eq:1.1}.
Similarly \eqref{eq:2.34} shows that we may replace every non-reduced pseudo-simple group 
$\mc G'_{nr,i}$ by $R_{F''_i / F'_i}(Sp_{2 n_i})$. In particular we obtain
\begin{equation}\label{eq:1.37}
G' = \mc G_{nr} (F) \times \mc G'_r (F) = \prod\nolimits_{i \in I_{nr}} Sp_{2 n_i} (F''_i) 
\times \prod\nolimits_{i \in I_s \cup I_e} \mc G_i (F_i) ,
\end{equation}
where we may read $\overline{\mc G_i}(F_i)$ for an exotic pseudo-reductive group $\mc G_i (F_i)$.
By the remarks after \eqref{eq:2.11} and \eqref{eq:2.34}, these replacements preserve the notions
of maximal $F$-split torus and of pseudo-parabolic $F$-subgroup.

\begin{lem}\label{lem:1.6}
\enuma{
\item $\mc C / \mc S$ is $F$-anisotropic and $\mc C (F) / \mc S (F)$ is compact.
\item $Z_{\mc G}(\mc S_\es) / \mc S_\es$ is $F$-anisotropic and 
$Z_{\mc G}(\mc S_\es) (F) / \mc S_\es(F)$ is compact.
} 
\end{lem}
\begin{proof}
(a) Let $\mc T_{\mc C}$ be the maximal $F$-torus in $\mc C$. Then $\mc C / \mc T_{\mc C}$ 
is a commutative $F$-group without nontrivial $F$-tori, so by \cite[Theorem 13.3.6]{Spr} 
it has no nontrivial tori, which means that it is unipotent. If $\mc C / \mc T_{\mc C}$ 
would contain any $F$-split unipotent subgroup, then by \cite[Lemma 4.1.4]{Con} so would $\mc C$.
That would contradict the pseudo-reductivity of $\mc C$, so $\mc C / \mc T_{\mc C}$ is a
$F$-wound unipotent group. Since $\mc S$ is the maximal $F$-split subtorus of $\mc T_{\mc C}$,
$\mc T_{\mc C} / \mc S$ and $\mc C / \mc S$ are $F$-anisotropic. Now Proposition \ref{prop:1.8}
tells us that $(\mc C / \mc S)(F)$ is compact. Since $\mc S$ is $F$-split, its first Galois
cohomology group is trivial and $\mc C (F) / \mc S (F)$
is isomorphic to $(\mc C / \mc S)(F)$ as topological groups.\\
(b) In view of \eqref{eq:2.6}, we first ignore $\mc G_{nr}$ and focus on $\mc G_r$. The preimage of 
$Z_{\mc G_r}(\mc S)$ is $Z_{\mc G'_r} (\mc S') \rtimes \mc C$. By \cite[Proposition A.5.15.ii]{CGP} 
there are (unique) maximal $F_i$-split tori $\mc S_i$ in $\mc G_i$ that 
$\mc S'$ is contained in $\prod_i R_{F_i/F}(\mc S_i)$. Since $F^\times$ is Zariski-dense and
cocompact in $F_i^\times$, $\mc S' (F)$ is Zariski-dense and cocompact in $\mc S_i (F_i)$, and
$R_{F_i/F}(\mc S_i) / \mc S'$ is $F$-anisotropic. Hence
\begin{equation}\label{eq:1.60}
Z_{\mc G'_r}(\mc S') (F) = 
\prod\nolimits_{i \in I_s \cup I_e} R_{F_i / F} (Z_{\mc G_i}(\mc S_i)) (F).
\end{equation}
For the standard factors of $\mc G'_r$, the reductivity of $\mc G_i$ and the maximality of 
$\mc S_i$ imply that each $Z_{\mc G_i}(\mc S_i)/\mc S_i$ is $F_i$-anisotropic. For the exotic 
factors $\mc G_i$ one first observes that $\mc S_i$ is mapped to a maximal $F_i$-split torus 
$\overline{\mc S_i}$ of $\overline{\mc G_i}$ by \eqref{eq:2.11}. Then 
$Z_{\overline{\mc G_i}}(\overline{\mc S_i}) / \overline{\mc S_i}$ is $F_i$-anisotropic because 
$\overline{\mc G_i}$ is reductive. Proposition \ref{prop:1.8} says that 
$\big( Z_{\overline{\mc G_i}} (\overline{\mc S_i}) / \overline{\mc S_i}\big)(F_i)$ is compact. 
Since $\overline{\mc S_i}$ is $F_i$-split this is homeomorphic to 
$Z_{\overline{\mc G_i}}(\overline{\mc S_i})(F_i) / \overline{\mc S_i} (F_i)$, which by
\eqref{eq:2.11} is homeomorphic to $Z_{\mc G_i}(\mc S_i) (F_i) / \mc S_i (F_i)$ as topological 
groups. In particular these groups are compact. By Proposition \ref{prop:1.7}.c also 
$(Z_{\mc G_i}(\mc S_i) / \mc S_i)(F_i)$ is compact, so by Proposition \ref{prop:1.8} 
$Z_{\mc G_i}(\mc S_i) / \mc S_i$ is $F_i$-anisotropic -- just as for the standard factors.

From the short exact sequence
\begin{multline*}
1 \to \prod\nolimits_{i \in I_s \cup I_e} R_{F_i/F}(\mc S_i) / \mc S' \to 
\prod\nolimits_{i \in I_s \cup I_e} R_{F_i / F} (Z_{\mc G_i}(\mc S_i)) / \mc S' \\
\to \prod\nolimits_{i \in I_s \cup I_e} R_{F_i / F} (Z_{\mc G_i}(\mc S_i)) / R_{F_i / F} 
(\mc S_i) \to 1 
\end{multline*}
we see that the middle term is $F$-anisotropic. By \eqref{eq:1.60} so is
$Z_{\mc G'_r} (\mc S') / \mc S'$. Then 
\[
1 \to Z_{\mc G'_r} (\mc S') / \mc S' \to \big( Z_{\mc G'_r} (\mc S') \rtimes \mc C \big) \big/
\big( \mc S' \times \mc S \big) \to \mc C / \mc S \to 1
\]
shows that also $( Z_{\mc G'_r} (\mc S') \rtimes \mc C ) / ( \mc S' \times \mc S )$ is 
$F$-anisotropic. Dividing out the central subgroup $\alpha (\mc T')$, we find that 
$Z_{\mc G_r}(\mc S) / \mc S$ is  also $F$-anisotropic. Now Proposition \ref{prop:1.8} tells us
that $(Z_{\mc G_r}(\mc S) / \mc S)(F)$ is compact. Since $\mc S$ is $F$-split the group 
$Z_{\mc G_r}(\mc S) (F) / \mc S (F)$ is homeomorphic to that, and also compact.

The totally non-reduced factors of $\mc G'$ (i.e. those of $\mc G_{nr}$) 
can be handled in the same way as the exotic factors, using \eqref{eq:2.34}.
\end{proof}

In general the quotient map $\mc G' \rtimes \mc C \to \mc G$ need not be surjective on 
$F$-rational points, but we can get quite close:

\begin{lem}\label{lem:2.5}
\enuma{
\item The natural map $(\mc G' \rtimes \mc C)(F) / \alpha (\mc T')(F) \to \mc G(F)$ 
is smooth, injective and open. Its image has finite index in $\mc G (F)$.
\item Let $\tilde{\mc G}(F)$ be the image of $\mc G'(F)$ in $\mc G(F)$. Then
\[
\tilde{\mc G}(F) Z_{\mc G_r}(\mc S)(F) = Z_{\mc G_r}(\mc S)(F) \tilde{\mc G}(F) = \mc G (F) .
\]
}
\end{lem}
\begin{proof}
(a) Since $\alpha$ is an embedding, $\alpha (\mc T' (F)) = \alpha (\mc T') (F)$, which 
implies the injectivity. The remaining properties were shown in Proposition \ref{prop:1.7}.\\
(b) By definition $\mc G_{nr}(F)$ is a direct factor on both sides, 
so it suffices to consider the map
\begin{equation}\label{eq:1.10}
(\mc G_r' \rtimes \mc C)(F) / \alpha (\mc T')(F) \to \mc G_r (F).
\end{equation}
Let $N_{\mc G_r}(\mc S)$ be the normalizer of the maximal $F$-split torus $\mc S$ in $\mc G_r$, 
and let $W = N_{\mc G_r}(\mc S) / Z_{\mc G_r}(\mc S)$ be the associated Weyl group. Let $\mc U$ 
be the $F$-split unipotent radical of a minimal pseudo-parabolic subgroup $\mc P$ of $\mc G_r$
containing $Z_{\mc G_r} (\mc S)$. The Bruhat decomposition of $\mc G_r (F)$ (see 
\cite[Theorem C.2.8]{CGP} or Theorem \ref{thm:2.2}) guarantees that any element 
$g \in \mc G_r (F)$ can be written as $z u \tilde{w} u'$, where 
$z \in Z_{\mc G_r} (\mc S) (F)$, $u, u' \in \mc U (F)$ and 
$\tilde w \in N_{\mc G_r}(\mc S)$ is any representative for an element $w \in W$. 

Since $\alpha (\mc T')$ is central in $\mc G'_r \rtimes \mc C$, the groups 
$\mc G_r' \rtimes \mc C$ and $\mc G_r$ have the same $F$-root system, and the pre-image $\mc P'$ of 
$\mc P$ in $\mc G'_r \rtimes \mc C$ is a minimal pseudo-parabolic subgroup over there 
\cite[Proposition 2.2.12]{CGP}. Then \eqref{eq:1.10} induces an isomorphism of 
algebraic groups $\mc U' \to \mc U$, where $\mc U'$ denotes the $F$-split unipotent radical of 
$\mc P'$. By \cite[Proposition C.2.10]{CGP} every element of $W$ can be represented by an 
element of $\mc G' (F)$, let us choose $\tilde w$ in this way. Then $u,u'$ and $\tilde{w}$ 
lie in $\tilde{\mc G} (F)$, so $g = z u w u' \in Z_{\mc G_r}(\mc S)(F) \tilde{\mc G}(F)$.
\end{proof}

Let $\lambda : GL_1 \to \mc S_\es$ be a cocharacter defined over $F$. It gives rise to a
pseudo-parabolic $F$-subgroup $\mc P_{\mc G}(\lambda)$, with $F$-split unipotent
radical $\mc R_{us,F}(\mc P_{\mc G}(\lambda)) = U_{\mc G}(\lambda)$. The group
$\mc Z_{\mc G}(\lambda) = \mc P_{\mc G}(\lambda) \cap \mc P_{\mc G}(\lambda^{-1})$
is again pseudo-reductive over $F$ \cite[Corollary 2.2.5]{CGP}. It satisfies 
$\mc P_{\mc G}(\lambda) = \mc Z_{\mc G}(\lambda) \ltimes \mc U_{\mc G}(\lambda)$ and
\begin{equation}\label{eq:2.7}
\text{Lie}(\mc G) = \text{Lie}(\mc U_{\mc G}(\lambda)) \oplus
\text{Lie}(\mc Z_{\mc G}(\lambda)) \oplus \text{Lie}(\mc U_{\mc G}(\lambda^{-1})) .
\end{equation}
Then every pseudo-parabolic $F$-subgroup containing $Z_{\mc G}(\mc S_\es)$ is of
the above form $\mc P_{\mc G}(\lambda)$, see \cite[Theorem C.2.15]{CGP}.
Since $\mc G$ is pseudo-reductive over $F$ and $\mc S_\es$ is maximal $F$-split torus,
the adjoint action of $\mc S_\es$ on the Lie algebra of $\mc G$ gives rise to a root system
$\Phi (\mc G,\mc S_\es)$, whose set of reduced roots we denote by $\Phi (\mc G,\mc S_\es)_\red$. 
For every reduced root $\beta \in \Phi (\mc G,\mc S_\es)$. there is a root subgroup $\mc U_\beta$, 
which by \cite[Proposition 2.1.10]{CGP} is a $F$-split unipotent group. Choose any ordering of 
the set of reduced roots $\Phi (\mc U_{\mc G}(\lambda),\mc S_\es)_\red$ appearing in the 
Lie algebra of $\mc U_{\mc G}(\lambda)$. Then the multiplication map
\begin{equation}\label{eq:2.40}
\prod\nolimits_{\beta \in \Phi (\mc U_{\mc G}(\lambda),\mc S_\es)_\red} \mc U_\beta 
\; \to \; \mc U_{\mc G}(\lambda)
\end{equation}
is a bijection \cite[Proposition C.2.26]{CGP}.

\begin{prop}\label{prop:2.1}
There exists a decreasing sequence $(K_n )_{n=0}^\infty$ of compact open subgroups
of $\mc G (F)$ such that:
\begin{itemize}
\item $\{ K_n : n \geq 0 \}$ is a neighborhood basis of 1 in $\mc G (F)$;
\item each $K_n$ is a normal subgroup of $K_0$;
\item each $K_n$ admits an Iwahori decomposition. That is, for every pseudo-parabolic
$F$-subgroup $\mc P_{\mc G}(\lambda)$ of $\mc G$ which contains $Z_{\mc G}(\mc S_\es)$,
the multiplication map
\[
\big( K_n \cap U_{\mc G}(\lambda)(F) \big) \times \big( K_n \cap Z_{\mc G}(\lambda)(F) \big) 
\times \big( K_n \cap U_{\mc G}(\lambda^{-1})(F) \big) \to K_n
\]
is a homeomorphism.
\end{itemize}
\end{prop}
\begin{proof}
We use an extension of the method of \cite{Del}.
We fix an ordering on $\Phi (\mc G, \mc S_\es)_\red \cup \{0\}$, such that the positive 
roots (with respect to $\mc P_\es$) come first, then $0$ and finally the negative roots. 
Writing $\mc U_0 = Z_{\mc G}(\mc S_\es)$, the root space decomposition becomes
\begin{equation}\label{eq:2.8}
\text{Lie}(\mc G) = \bigoplus\nolimits_{\alpha \in \Phi (\mc G, \mc S_\es)_\red \cup \{0\}}
\text{Lie}(\mc U_\alpha) .
\end{equation}
Moreover the properties of the open Bruhat cell \cite[Corollary 3.3.6 and p. 624]{CGP} 
imply that the multiplication map
\begin{equation}\label{eq:2.31}
\prod\nolimits_{\alpha \in \Phi (\mc G, \mc S_\es) \cup \{0\}}
\mc U_\alpha (F) \; \to \; \mc G (F)  
\end{equation}
is an open immersion.

Let $\mc O_F$ be the ring of integers of $F$ and $\omega_F \in \mc O_F$ a uniformizer.
Choose an $\mc O_F$-lattice $L$ in Lie$(\mc G)(F)$ such that in \eqref{eq:2.8} $L$ is
the direct sum of the
\[
L_\alpha := L \cap \text{Lie}(\mc U_\alpha)(F) \quad \text{with} \quad
\alpha \in \Phi (\mc G,\mc S)_\red \cup \{0\}.
\]
Choose $\mc O_F$-bases of the $L_\alpha$, and combine them to an $\mc O_F$-basis $B$ of $L$.
For all $b,b' \in B$, the Lie bracket $[b,b']$ is an element of Lie$(\mc G)(F)$. Pick
$i \in 2 \N$ such that
\[
[\omega_F^i b, \omega_F^i b'] = \omega_F^{2i} [b,b'] \in \omega_F^{i/2} (\omega_F^i L)
\quad \forall b, b' \in B.
\]
Then the $\mc O_F$-lattice $\omega_F^i L$ is a Lie algebra, it even satisfies
\begin{equation}\label{eq:2.18}
[x,y] \in \omega_F^{i/2} (\omega_F^i L) \quad \forall x,y \in \omega_F^i L .
\end{equation}
For each $\alpha \in \Phi (\mc G,\mc S)_\red \cup \{0\}$ we choose an analytic diffeomorphism 
$\phi_\alpha$ from a neighborhood of 0 in Lie$(\mc U_\alpha)(F)$ to a neighborhood of 1 in 
$\mc U_\alpha (F)$, such that $\phi_\alpha (0) = 1$ and the differential of $\phi_\alpha$ 
at 0 is the identity map of Lie$(\mc U_\alpha)(F)$. For $l_\alpha$ in the domain of
$\phi_\alpha$ we define 
\[
\phi \Big( \sum\nolimits_\alpha l_\alpha \Big) = \prod\nolimits_\alpha \phi_\alpha (l_\alpha), 
\]
where we use the ordering of $\Phi (\mc G,\mc S)_\red \cup \{0\}$. By \eqref{eq:2.31} $\phi$
is an analytic diffeomorphism from a neighborhood of 0 in Lie$(\mc G)(F)$ to a neighborhood of 
1 in $\mc G (F)$, and d$\phi$ is the identity map. The subsets $\phi (\omega_F^i L)$ of
$\mc G (F)$ (for $i \in \N$ so large that $\omega_F^i L$ is contained in the domain of $\phi$) 
form a neighborhood basis of 1 in $\mc G (F)$, because the sets
$\omega_F^i L$ form a neighborhood basis of 0 in Lie$(\mc G)(F)$.

Let us endow Lie$(\mc G)(F)$ with a non-archimedean norm $\norm{}$. Because $L$ is open in
Lie$(\mc G)(F)$, there exists a $\epsilon \in \R_{>0}$ such that $L$ contains the ball of
radius $\epsilon \max_{l \in L} \norm{L}$ around 0 in Lie$(\mc G)(F)$. Then, for each $i \in \N$:
\begin{equation}\label{eq:2.44}
\omega_F^i L \text{ contains the ball of radius } \epsilon \max_{l \in \omega_F^i L} \norm{l} . 
\end{equation}
By the continuity of inversion in $\mc G (F)$, the domain of $\phi$ contains 
$\phi (\omega_F^i L)^{-1}$ for sufficiently large $i$. For such $i$ we define the map
\[
\begin{array}{cccc}
\imath : & \omega_F^i L & \to & L\\
& l & \mapsto & \phi^{-1}(\phi(l)^{-1}) + l   
\end{array} .
\]
The assumptions on $\phi$ entail that $\imath$ is $F$-analytic and that its derivative at 
$0 \in \omega_F^i L$ is zero. In terms of the norm, this says
\[
\lim_{l \to 0} \norm{\imath (l)} \norm{l}^{-1} = 0 . 
\]
Hence, for $j \in \N$ sufficiently large, $\norm{\imath(l)} < \epsilon \norm{l}$ on $\omega_F^j L$.
With \eqref{eq:2.44} we obtain $\imath (\omega_F^j L) \subset \omega_F^j L$. Furthermore
$\omega_F^j L$ is a subgroup of Lie$(\mc G)(F)$, so 
\[
\phi^{-1}(\phi (l)^{-1}) = \imath (l) - l \in \omega_F^j L \quad \text{for all } l \in \omega_F^j L .
\]
This says that $\phi (\omega_F^j L)$ is closed under taking inverses, for $j \in \N$ sufficiently large.

By the continuity of multiplication in $\mc G (F)$, the image of $\phi$ contains 
$\phi (\omega_F^i L)^2$ for $i$ sufficiently large. For such $i$ we consider the $F$-analytic map
\[
\begin{array}{cccc}
\mu : & \omega_F^i L \times \omega_F^i L & \to & L\\
 & (l,m) & \mapsto & \phi^{-1}(\phi(l) \phi(m)) - l - m  
\end{array} .
\]
Since $\mu (0,m) = 0$ for all $m$ and $\mu$ is smooth, there exists $f : \omega_F^i L \to \R_{\geq 0}$ 
such that $\norm{\mu (l,m)} \leq f(m) \norm{l}$ for all $l,m \in \omega_F^i L$. Since $\omega_F^i L$
is compact, we can take
\begin{equation}\label{eq:2.41}
f(m) = \sum\nolimits_{b \in B} \max_{l\in \omega_F^i L} 
\norm{ \frac{\partial \mu}{\partial l_b} (l,m) } ,
\end{equation}
where $\partial \mu / \partial l_b$ denotes the partial derivative with respect to the variable
$l$ in the direction of the basis element $b$ of $L$. Since $l \mapsto \mu (l,0)$ is identically
zero, so is $l \mapsto \frac{\partial \mu}{\partial l_b} (l,0)$. Therefore the above reasoning
can also be applied to $\partial \mu / \partial l_b$. This yields $C_b \in \R_{>0}$ such that
$\norm{ \frac{\partial \mu}{\partial l_b} (l,m) } \leq C_b \norm{m}$ for all $l,m \in \omega_F^i L$, 
for example
\begin{equation}\label{eq:2.42}
C_b = \sum\nolimits_{a \in B} \max_{l,m \in \omega_F^i L} 
\norm{\frac{\partial^2 \mu}{\partial m_a \partial l_b} (l,m)} .
\end{equation}
From \eqref{eq:2.41} and \eqref{eq:2.42} we conclude that 
\begin{equation}\label{eq:2.43}
\norm{\mu (l,m)} \leq f(m) \norm{l} \leq \sum\nolimits_{b \in B} C_b \norm{m} \norm{l} 
\quad \text{for all} \quad l,m \in \omega_F^i L .
\end{equation}
Suppose that $\omega_F^j L$ contains the ball of radius $\sum_{b \in B} C_b 
\max_{l \in \omega_F^j L} \norm{l}^2$ in $L$. (By \eqref{eq:2.44} this is the case for 
$j \in \N$ sufficiently large.) 
Then \eqref{eq:2.43} says that $\mu \big( (\omega_F^j L)^2 \big) \subset \omega_F^j L$.
From the definition of $\mu$ we see that also $\phi^{-1}(\phi(l) \phi(m)) \in \omega_F^j L$,
or equivalently $\phi(l)\phi(m) \in \phi (\omega_F^j L)$ for all $l,m \in \omega_F^j L$.

We showed that, for $j \in \N$ sufficiently large, $\phi (\omega_F^j)$ is closed under 
multiplication and under inversion. Thus there exists $i_0 \in \N$ such that,
for $j \geq i_0$, $\phi (\omega_F^j L)$ is a compact open subgroup of $\mc G (F)$.

For $k \in \phi (\omega_F^{i_0} L)$ we consider the map
\begin{equation}\label{eq:2.24}
\begin{array}{cccc} 
\gamma_k : & \omega_F^{i_0} L & \to & \omega_F^{i_0} L \\
 & l & \mapsto & \phi^{-1} (k \phi (l) k^{-1} \phi (l)^{-1}) 
\end{array}.
\end{equation}
By the analyticity of $\phi$, the $\gamma_k$ converge to the zero map as $k$ goes to
$1 \in \mc G (F)$. In view of the compactness of $\omega_F^{i_0} L$, this can be formulated
more precisely. For every $D \in \N$ there exists an open neighborhood
$U_D$ of 1 in $\mc G (F)$ such that
\begin{equation}\label{eq:2.19}
\norm{\gamma_k (l)} < D^{-1} \norm{l} \quad \forall k \in U_D, l \in \omega_F^{i_0} L.
\end{equation}
For $D$ such that the ball of radius $D^{-1} \max_{l \in L} \norm{l}$ is
contained in $L$, \eqref{eq:2.19} implies that
\begin{equation}\label{eq:1.25}
\begin{array}{l}
k \phi (l) k^{-1} \phi (l)^{-1} = \phi (\gamma_k (l)) \in \phi (\omega_F^j L) \\
k \phi (\omega_F^j L) k^{-1} \subset \phi (\omega_F^j L)
\end{array}
\qquad \forall k \in U_D, j \geq i_0, l \in \omega_F^j L .
\end{equation}
Fix such a $D$ and choose $i_D \geq i_0$ such that $\phi (\omega_F^{i_D} L) \subset U_D$
and \eqref{eq:2.18} holds.
Then all the $\phi (\omega_F^j L)$ with $j \geq i_D$ are normal subgroups of
$\phi (\omega_F^{i_D} L)$. We define
\[
K_n := \phi (\omega_F^{n + i_D} L) \text{ for } n \geq 0 .
\]
Now the first two bullets of the statement are satisfied.

For $\alpha \in \Phi (\mc G, \mc S_\es)_\red \cup \{0\}$, the above argument
also shows that $\phi (\omega_F^j L_\alpha)$ with $j \geq i_D$ is a normal subgroup
of $\phi (\omega_F^{i_D} L_\alpha)$. The bijectivity of $\phi : L \to L$ implies that
\[
\phi ( \omega_F^j L ) \cap \phi (\omega_F^{i_D} L_\alpha) = \phi (\omega_F^j L_\alpha) .
\]
For any ordering of $\Phi (\mc G, \mc S_\es)_\red \cup \{0\}$, not necessarily the same
as we fixed before, we can consider the multiplication map
\begin{equation}\label{eq:2.20}
\prod\nolimits_{\alpha \in \Phi (\mc G, \mc S_\es)_\red \cup \{0\}} 
\phi (\omega_F^j L_\alpha) \longrightarrow \phi (\omega_F^j L) .
\end{equation}
By the above, \eqref{eq:2.20} gives rise to a well-defined map on the lower line of the
following diagram:
\begin{equation}\label{eq:2.21}
\begin{array}{ccc}
\bigoplus_\alpha \omega_F^j L_\alpha / \omega_F^{j+1} L_\alpha & \longrightarrow &
\omega_F^j L / \omega_F^{j+1} L \\
\downarrow \phi & & \downarrow \phi \\
\prod_\alpha \phi (\omega_F^j L_\alpha) / \phi (\omega_F^{j+1} L_\alpha) & \longrightarrow &
\phi (\omega_F^j L) / \phi (\omega_F^{j+1} L)
\end{array}
\end{equation}
By \eqref{eq:2.19} the two lower terms in \eqref{eq:2.21} are abelian groups, and the
map between them is a group homomorphism. Similarly \eqref{eq:2.18} says that the two
upper terms are abelian Lie algebras, and that the map between them is a Lie homomorphism.
Consequently the diagram \eqref{eq:2.21} commutes.

Since $\phi$ was bijective to start with, the downward arrows induced by $\phi$ are still
bijective. The upper horizontal map is a bijection because $\omega_F^j L = \bigoplus_\alpha
\omega_F^j L_\alpha$, so the lower horizontal map is also bijective. Using this for
every $j \geq i_D$, in combination with the profiniteness of all the involved groups,
we deduce that \eqref{eq:2.20} is bijective. Since it is also an analytic map between
compact $\mc O_F$-varieties, it is in fact a homeomorphism.

For any pseudo-parabolic $F$-subgroup $\mc P_{\mc G}(\lambda)$ containing $Z_{\mc G}(\lambda)$,
we can focus on the roots $\alpha$ in $\Phi (\mc U_{\mc G}(\lambda), \mc S_\es)_\red$. 
The above argument shows that multiplication induces a homeomorphism
\[
\prod_{\alpha \in \Phi (\mc U_{\mc G}(\lambda), \mc S_\es)_\red} \hspace{-8mm} 
K_n \cap \mc U_\alpha (F) = \prod_{\alpha \in \Phi (\mc U_{\mc G}(\lambda), \mc S_\es)_\red} 
\hspace{-8mm} \phi (\omega_F^{n + i_D} L_\alpha) \longrightarrow \phi \Big( 
\prod_{\alpha \in \Phi (\mc U_{\mc G}(\lambda), \mc S_\es)_\red} \hspace{-8mm}
\omega_F^{n + i_D} L_\alpha \Big) .
\]
All this takes place in $\mc U_{\mc G} (\lambda)(F)$, and considering that as enveloping
group we see that the right hand side equals $K_n \cap \mc U_{\mc G} (\lambda)(F)$.
The groups $K_n \cap \mc Z_{\mc G} (\lambda)(F)$ and $K_n \cap \mc U_{\mc G} (\lambda^{-1})(F)$
admit analogous expressions.

Finally we use the freedom in the ordering of $\Phi (\mc G, \mc S_\es)_\red \cup \{0\}$. 
We may assume that the roots in 
$\Phi (\mc U_{\mc G}(\lambda),\mc S_\es)_\red$ come first, then those in 
$\Phi (\mc Z_{\mc G}(\lambda), \mc S_\es)_\red$ and finally the remaining set 
$\Phi (\mc U_{\mc G}(\lambda^{-1}), \mc S_\es)_\red$. Under those conditions \eqref{eq:2.40}
shows that \eqref{eq:2.20} becomes the Iwahori decomposition of $K_n$ with respect to 
$\mc P_{\mc G}(\lambda)$.
\end{proof}

\subsection{Actions on affine buildings} \

To understand the geometry of pseudo-reductive groups over non-archimedean local fields 
better, we first make such groups act on affine buildings. 

\begin{lem}\label{lem:1.1}
Let $\mc C$ be a commutative pseudo-reductive $F$-group. Let $\mc S$ be the maximal
$F$-split torus in $\mc C$, and let $X_* (\mc S)$ be its cocharacter lattice. Then
there is a natural cocompact, proper action of $\mc C (F)$ on $X_* (\mc S) \otimes_\Z \R$,
by translations.
\end{lem}
\begin{proof}
By \cite[Theorem 1.4.3]{Loi} $\mc C (F)$ contains a unique maximal compact subgroup,
say $K_{\mc C}$, which is open. Then $\mc C (F) / K_{\mc C}$ is a finitely generated free
abelian group and $\mc C (F)$ acts properly and cocompactly on $\mc C (F) / K_{\mc C} 
\otimes_\Z \R$, by multiplication. This can be regarded as a translation action of
$\mc C (F)$ on a vector space.

Furthermore $K_{\mc C} \cap \mc S (F)$ is the unique maximal compact subgroup of
$\mc S (F)$, and $\mc S (F) / (K_{\mc C} \cap \mc S (F))$ embeds naturally in
$\mc C (F) / K_{\mc C}$. The image is of finite index in $\mc C (F) / K_{\mc C}$
because $\mc C (F) / \mc S(F)$ is compact (Lemma \ref{lem:1.6}.a) and 
$K_{\mc C}$ is open. Hence there is a natural isomorphism
\begin{equation}\label{eq:1.2}
\mc S (F) / (K_{\mc C} \cap \mc S (F)) \otimes_\Z \R \to \mc C (F) / K_{\mc C} \otimes_\Z \R .
\end{equation}
Let $\omega_F$ be a uniformizer of $F$. The map
\begin{equation}\label{eq:1.3}
X_* (\mc S) \to \mc S (F) : \nu \mapsto \nu (\omega_F^{-1})
\end{equation}
and \eqref{eq:1.2} provide canonical isomorphisms
$X_* (\mc S) \to \mc S (F) / (K_{\mc C} \cap \mc S (F))$ and
\begin{equation}\label{eq:1.5}
X_* (\mc S) \otimes_\Z \R \to \mc C (F) / K_{\mc C} \otimes_\Z \R .
\end{equation}
We transfer the $\mc C (F)$-action from the right hand side to the left hand side.
Restricted to $\mc S (F)$, this recovers the action coming from \eqref{eq:1.3} and 
used in Bruhat--Tits theory, see \cite[\S 1.1]{ScSt}. In this way we get a canonical 
proper, cocompact $\mc C (F)$-action on $X_* (\mc S) \otimes_\Z \R$.
\end{proof}

Let $\mc T_{nr,i}$ be the standard maximal split torus in $Sp_{2 n_i}$, so that
$\mc T_{nr,i}(F)$ is a maximal $F$-split torus in $R_{F'_i/F}(\mc G'_{nr,i})(F)$.
We write
\[
\mc T = \prod\nolimits_{i \in I_{nr}} R_{F''_i / F}(\mc T_{nr,i}) \times
\Big( \prod\nolimits_{i \in I_s \cup I_e} R_{F_i / F}(\mc T_i) \rtimes \mc S \Big) / \alpha (\mc T') .
\]
In view \eqref{eq:1.45} we may assume that our maximal $F$-split torus of $\mc G$ is given as
\[
\mc S_\es := \prod\nolimits_{i \in I_{nr}} \mc T_{nr,i} \times \mc S .
\]

\begin{prop}\label{prop:2.3}
There exists an affine building $\mc B (\mc G,F)$ with an action 
of the group $\mc G (F)$ from \eqref{eq:2.6}, such that:
\enuma{
\item The action is isometric, proper and cocompact;
\item The set of apartments of $\mc B (\mc G,F)$ is naturally in bijection with the set of
maximal $F$-split tori in $\mc G$.
\item $\mc C (F)$ stabilizes an apartment.
}
\end{prop}
\begin{proof}
For notational ease we simplify the description of the totally non-reduced factors and of 
the exotic pseudo-reductive factors.
During this proof we will denote the factors $R_{F''_i/F}(Sp_{2 n_i})$ from \eqref{eq:2.2} and 
\eqref{eq:1.37} and the factors $R_{F_i/F} (\overline{\mc G_i})$ from \eqref{eq:2.11} also as 
$R_{F_i / F}(\mc G_i)$. That is, we treat them on the same footing as the standard $\mc G_i (F_i)$.

Let $\mc{BT} (\mc G_i, F_i)$ be the Bruhat--Tits building of $\mc G_i (F_i)$. We claim that
this building admits a natural isometric action of any group $H_i$ that acts on $\mc G_i (F_i)$ 
by automorphisms of topological groups. (See Proposition \ref{prop:3.19}.c for a generalization
and more detailed proof of this claim.) The stabilizer $\mc G_i (F_i)_{x_i}$ of any point
$x_i \in \mc{BT} (\mc G_i,F_i)$ is a compact open subgroup of $\mc G_i (F_i)$. As $\mc G_i$ is
semisimple and simply connected, \cite[Corollaire 3.3.3]{BrTi1} applies, and tells us that 
$\mc G_i (F_i)_{x_i}$ is a maximal compact subgroup if and only if $x_i$ is a vertex of 
$\mc{BT} (\mc G_i, F_i)$. In that case $x_i$ is the only fixed point of $\mc G_i (F_i)_{x_i}$,
for different facets of $\mc{BT}(\mc G_i,F_i)$ have different stabilizers.

For any $h \in H_i$, $h (\mc G_i (F_i)_{x_i})$ is another maximal compact subgroup of 
$\mc G_i (F_i)_{x_i}$, so of the form $\mc G_i (F_i)_{y_i}$ for a unique vertex of 
$\mc{BT}(\mc G_i,F_i)$. It is easily seen that the definition $h (x_i) := y_i$ extends uniquely
to an isometric action of $H_i$ on $\mc{BT}(\mc G_i,F_i)$. This applies in particular 
to the conjugation actions of $G' \rtimes C$ and $G$ on $\mc G_i (F_i)$.

Let $\mc Z = (\mc S \cap Z(\mc G) )^\circ$ be the maximal $F$-split torus in $Z(\mc G)$. 
Then $\mc S = \mc Z (\mc S \cap \phi_{\mc C}(\mc T'))^\circ$ and
\begin{equation}\label{eq:1.4}
X_* (\mc S) \otimes_\Z \R = X_* (\mc S \cap \phi_{\mc C}(\mc T')^\circ)
\otimes_\Z \R \; \oplus \; X_* (\mc Z) \otimes_\Z \R .
\end{equation}
In the decomposition \eqref{eq:1.4}, with the action from Lemma \ref{lem:1.1}, 
$\phi_{\mc C} (T')$ acts only the first summand, and the quotient $C / \phi_{\mc C} (T')$ 
acts on the second summand. 
Thus $G' \rtimes C$ and $G$ act on $X_* (\mc Z) \otimes_\Z \R$, via the quotient maps
\begin{equation}\label{eq:1.7}
G' \rtimes C \to G \to C / \phi_{\mc C} (T') .
\end{equation}
We define the affine buildings
\begin{equation}\label{eq:1.23}
\begin{aligned}
\mc B (\mc G' \rtimes \mc C ,F) & = \prod\nolimits_{i \in I_s \cup I_e \cup I_{nr}} 
\mc{BT} (\mc G_i ,F_i) \times X_* (\mc S) \otimes_\Z \R ,\\
\mc B (\mc G, F) & = \prod\nolimits_{i \in I_s \cup I_e \cup I_{nr}} \mc{BT} (\mc G_i ,F_i) 
\times X_* (\mc Z) \otimes_\Z \R .
\end{aligned}
\end{equation}
(a) The group $G' \rtimes C$ acts componentwise on these two buildings, and the action
on $\mc B (\mc G,F)$ factors through $G$. It is known from Bruhat--Tits theory 
\cite[\S 2.2]{Tit} that the action of $\mc G' (F)$ on
$\prod\nolimits_{i \in I_s \cup I_e \cup I_{nr}} \mc{BT} (\mc G_i ,F_i)$
is proper and cocompact. The action of
$C / \phi_{\mc C} (T')$ (resp. of $C$) on $X_* (\mc Z) \otimes_\Z \R$ (resp. on
$X_* (\mc S) \otimes_\Z \R$) is proper and cocompact by Lemma \ref{lem:1.1}. Hence
the action of $G' \rtimes C$ on $\mc B (\mc G' \rtimes \mc C,F)$ is proper and cocompact.
By Lemma \ref{lem:2.5}.a the action of $G$ on $\mc B (\mc G,F)$ is also proper and cocompact.

Recall \cite[Lemme 2.5.1]{BrTi1} that $\mc B (\mc G_i, F_i)$ is endowed with a 
$\mc G_i (F)$-invariant metric $d_i$, which comes from an inner product on an apartment of
this building. We endow $\mc B (\mc G ,F)$ with a metric of the form
\begin{equation}\label{eq:1.9}
d \big( (\prod\nolimits_i x_i ,x_{\mc Z}), (\prod\nolimits_i y_i, y_{\mc Z}) \big) =
\sum\nolimits_i d (x_i,y_i) + \inp{x_{\mc Z} - y_{\mc Z}}{x_{\mc Z} - y_{\mc Z}}_{\mc Z} ,
\end{equation}
where $\inp{}{}_{\mc Z}$ is an inner product on $X_* (\mc Z) \otimes_\Z \R$. Since 
$G' \rtimes C$ and $G$ act on $X_* (\mc Z) \otimes_\Z \R$ by translations, the actions of 
$G' \rtimes C$ and $G$ on $\mc B (\mc G',F)$ are isometric. In the same way we can define a 
$G' \rtimes C$-invariant metric on $\mc B (\mc G' \rtimes \mc C,F)$.\\
(b) By definition an apartment of $\mc B (\mc G,F)$ is a set of the form
$\prod_i \mh A_i \times X_* (\mc Z) \otimes_\Z \R$, where each $\mh A_i$ is an apartment
of $\mc{BT}(\mc G_i,F_i)$. According to \cite[\S 2.1]{Tit} there is a canonical bijection 
between the set of apartments of $\mc{BT}(\mc G_i,F_i)$ and the set of maximal $F_i$-split 
tori of $\mc G_i$. By \cite[Proposition A.5.15.ii]{CGP} the latter set is in natural bijection 
with the set of maximal $F$-split tori in $R_{F_i / F}(\mc G_i)$. Thus the apartments of 
$\mc{BT}(\mc G,F)$ correspond bijectively to the maximal $F$-split tori in $\mc G'$. 

Clearly the maps $\mc S' \leftrightarrow \mc S' \times \mc S$ provide a bijection between
these and the maximal $F$-split tori of $\mc G' \rtimes \mc C$. As $\mc G = (\mc G' \rtimes 
\mc C) / \alpha (\mc T')$, the natural map $\mc G' \rtimes \mc C \to \mc G$ sends maximal
$F$-split tori to maximal $F$-split tori. As $\alpha (\mc T')$ central, 
\cite[Proposition 2.2.12.i]{CGP} shows that this induces a bijection between the set of maximal 
$F$-split tori of $\mc G' \rtimes \mc C$ and the analogous set for $\mc G$. \\
(c) Since we assumed that every $\mc T_i$ contains a maximal $F_i$-split torus, we can define
$\mh A_{\mc T_i}$ to be the apartment of $\mc BT (\mc G_i, F_i)$
associated to that maximal $F_i$-split torus. The sets
\[
\mh A_{\mc T} = \prod\nolimits_i \mh A_{\mc T_i} \times X_* (\mc S) \otimes_\Z \R \quad
\text{and} \quad \mh A_\es = \prod\nolimits_i \mh A_{\mc T_i} \times X_* (\mc Z) \otimes_\Z \R
\]
are apartments in, respectively, $\mc B (\mc G' \rtimes \mc C,F)$ and $\mc B (\mc G,F)$.
The latter is associated to the maximal $F$-split torus $\mc S_\es$.
The action of $C$ on $\mc B (\mc G',F)$ factors via $\mc T' (F) / Z(\mc G' (F))$, so it
stabilizes $\prod_i \mh A_{\mc T_i}$. It follows that $C \subset G' \rtimes C$ stabilizes
$\mh A_{\mc T}$ and that $C \subset G$ stabilizes $\mh A_\es$.
\end{proof}

The building $\prod_i \mc{BT}(\mc G_i,F_i)$ is endowed with a canonical polysimplicial
structure, but $X_* (\mc Z) \otimes_\Z \R$ is not. A facet of $\mc B (\mc G,F)$ is by
definition a subset of the form $\mf f = \prod_i \mf f_i \times X_* (\mc Z) \otimes_\Z \R$,
where each $\mf f_i$ is a facet of $\mc{BT}(\mc G_i,F_i)$. 
By \cite[\S 3.1]{Tit} and Lemma \ref{lem:1.1}
\begin{equation}\label{eq:2.60}
\text{the pointwise stabilizer of any facet of } \mc B (\mc G,F) 
\text{ is an open subgroup of } \mc G(F).
\end{equation}
We will call a point $x = ((x_i)_i,x_{\mc Z}) \in \mc B (\mc G,F)$ a vertex if each 
$x_i$ is a vertex of $\mc{BT}(\mc G_i,F_i)$ and $x_{\mc Z} \in X_* (\mc Z) \otimes_\Z \R$.
Recall that $x_i$ is a special vertex of $\mc{BT}(\mc G_i,F_i)$ if there is an apartment
$\mh A_i$ such that for every wall in $\mh A_i$ there is a parallel wall in $\mh A_i$ 
which contains $x_i$. We say that $x$ is a special vertex if each $x_i$ is so. By 
\cite[\S 1.3.7]{BrTi1} every apartment of $\mc{BT}(\mc G_i,F_i)$, and hence every 
apartment of $\mc B (\mc G,F)$, contains special vertices.

As a first consequence of the existence of appropriate affine buildings for $\mc G(F)$,
we analyse the structure of the centralizer of a maximal $F$-split torus.

\begin{lem}\label{lem:2.7}
$Z_G (S_\es) = Z_{\mc G}(\mc S_\es)(F)$ has a unique maximal compact subgroup,
say $Z_G (S_\es)_\cpt$.
\end{lem}
\begin{proof}
The group $Z_G (S_\es)$ stabilizes the apartment $\mh A_\es$ of $\mc B (\mc G,F)$
and acts by translations on it (see the proof of Lemma \ref{lem:1.1} for $C$ and
\cite[\S 1.2]{Tit} for $G_{nr} \times G'_r$).
By the Bruhat--Tits fixed point theorem every compact subgroup of $Z_G (S_\es)$ fixes
a point of $\mh A_\es$. By Proposition \ref{prop:2.3} the action is proper, so every maximal
compact subgroup of $Z_G (S_\es)$ is the stabilizer of a point of $\mh A_\es$.
But the action is by translations, so every point of $\mh A_\es$ has the same isotropy
group. That is the unique maximal compact subgroup of $Z_G (S_\es)$.
\end{proof}

\subsection{The Iwasawa and Cartan decompositions} \
\label{subsec:dec}

The above actions on buildings can be used to construct maximal compact subgroups of 
pseudo-reductive $F$-groups. Recall that $K_{\mc C}$ is the unique maximal 
compact subgroup of $C$, and denote the unique maximal compact subgroup of $T'$ by $T'_\cpt$.

\begin{lem}\label{lem:1.3}
Let $x = ((x_i)_i,x_{\mc Z})$ be a vertex in the apartment $\mh A_\es$ of $\mc B (\mc G,F)$.
\enuma{
\item $G'_{(x_i)_i}$ is a maximal compact subgroup of $G'$ and $G'_{(x_i)_i} \rtimes K_{\mc C}$
is a maximal compact subgroup of $G' \rtimes C$.
\item $G_x$ is a maximal compact subgroup of $G$. It contains 
$(G'_{(x_i)_i} \rtimes K_{\mc C}) / \alpha (T'_\cpt)$ as an open subgroup of finite index.
}
\end{lem}
\begin{proof}
(a) As a consequence of the Bruhat--Tits fixed point theorem \cite[Corollaire 3.3.2]{BrTi1}, 
the group $G'_{(x_i)_i} = \prod_i \mc G_i (F_i)_{x_i}$ is a maximal compact subgroup of 
$G' = \prod_i \mc G_i (F_i)$. 
The action of $C$ on $\prod_i \mh A_{\mc T_i}$ factors via $\mc T' (F) / Z(\mc G' (F))$,
so it acts by translations and every compact subgroup of $\mc C$ acts trivially.
In particular the maximal compact subgroup $K_{\mc C}$ of $C$ fixes $\prod_i \mh A_{\mc T_i}$
pointwise. Hence $K_{\mc C}$ normalizes $G'_{(x_i)_i}$, and $G'_{(x_i)_i} \rtimes K_{\mc C}$ 
is a compact open subgroup of $G' \rtimes C$.

Since the image of $G'_{(x_i)_i} \rtimes K_{\mc C}$ in $C \cong G' \rtimes C / G'$  is the unique
maximal compact subgroup of $C$ and since $G'_{(x_i)_i}$ is maximal compact in $G'$,
$G'_{(x_i)_i} \rtimes K_{\mc C}$ is maximal compact in $G' \rtimes C$.\\
(b) $G_{x}$ is maximal compact in $G$ for the same reason as $K'$.
By construction the image of $K' \rtimes K_{\mc C}$ in $G$ is contained in $G_x$. By Proposition
\ref{prop:1.7} it is open and of finite index in $G_x$. By Lemma \ref{lem:2.5}.a this image is 
isomorphic to $(K' \rtimes K_{\mc C} ) / (\alpha (T') \cap K' \rtimes K_{\mc C})$. Since $K'$ is 
associated to a point in the apartment of $\mc B (\mc G',F)$ stabilized by $T'$, 
$\alpha (T') \cap (K' \rtimes K_{\mc C})$ is the maximal compact subgroup 
$\alpha (T')_\cpt = \alpha (T'_\cpt)$ of $\alpha (T')$. 
\end{proof}

The Iwasawa decomposition and the Cartan decomposition are two of the most useful
results abouts the structure of reductive groups over local fields. 
We thank Gopal Prasad for explaining to us how the next result can be proven using the
universal smooth $F$-tame central extension of $\mc G$ from \cite{CP}. Nevertheless, with
an eye on later results we find it useful to provide a different argument.

Let $x'$ be a special vertex of the apartment $\mh A_\es$ and let $K = G_{x'}$ be its
stabilizer in $G$. Let $K'$ be the stabilizer in $G' = G_{nr} \times G'_r$ of the image of $x'$ in 
$\mc B (\mc G' \rtimes \mc C,F)$. The maximal compact subgroups $K \subset G$ and $K' \subset G'$ 
contain representatives for all elements of $W(\mc G,\mc S_\es)$, because $x'$ and its image 
in $\mc B (\mc G' \rtimes \mc C,F)$ are special vertices \cite[Proposition 4.4.6]{BrTi1}.

\begin{thm}\label{thm:2.4}
The group $G$ admits an Iwasawa decomposition. More precisely, for every pseudo-parabolic
$F$-subgroup $\mc P$ of $\mc G$ we have
\[
\mc G (F) = \mc P (F) K = K \mc P (F).
\]
\end{thm}
\begin{proof}
First we prove the theorem for $(\mc G_{nr} \times \mc G'_r)(F)$ with the compact
subgroup $K'$. By construction
\begin{align*}
& \mc B (\mc G_{nr} \times \mc G'_r,F) = \mc B (\mc G_{nr},F) \times \mc B (\mc G'_r,F) ,\\
& K' = (K' \cap \mc G_{nr}(F)) \times (K' \cap \mc G'_r (F)) .
\end{align*}
Moreover both factors of $K'$ are associated to special vertices of the appropriate affine
buildings. Thus we may treat the two groups $\mc G_{nr}$ and $\mc G'_r$ separately.

We note that $\mc G_{nr}(F) = \prod_{i \in I_{nr}} Sp_{2n_i}(F''_i)$ and that each 
$Sp_{2n_i}(F''_i)$ satisfies an Iwasawa decomposition. As compact subgroup of 
$Sp_{2n_i}(F''_i)$ one can use the isotropy group of a special vertex in the Bruhat--Tits 
building $\mc B (Sp_{2n_i},F''_i)$. Now $K' \cap \mc G_{nr}(F)$ is a direct product of such 
isotropy groups, so $(\mc G_{nr}(F), K' \cap \mc G_{nr}(F))$ satisfies the Iwasawa 
decomposition, for every pseudo-parabolic $F$-subgroup.

Similarly, for $\mc G'_r (F)$ the problem reduces to the direct factors $\mc G_i (F_i)$,
which by \eqref{eq:2.11} are isomorphic to reductive groups over local fields. The
same argument as for $\mc G_{nr}(F)$ applies here. This concludes the case
$(\mc G_{nr} \times \mc G'_r)(F)$, now we consider $\mc G (F)$.

Every pseudo-parabolic $F$-subgroup arises from a $F$-rational
cocharacter $\lambda : GL_1 \to \mc G$. Since all maximal $F$-split tori in $\mc G (F)$
are conjugate \cite[Theorem C.2.3]{CGP}, $\mc P$ is conjugate to $\mc P_{\mc G}(\lambda)$
for some $\lambda : GL_1 \to \mc T$.

In particular $\mc P$ is $\mc G (F)$-conjugate to a pseudo-parabolic $F$-subgroup of
$\mc G$ containing $\mc C$. By \cite[Proposition 2.2.12]{CGP} the pre-image $\mc P_1$
of $\mc P$ in $\mc G_{nr} \times \mc G'_r \rtimes \mc C$ is again pseudo-parabolic.
It follows that there exists a pseudo-parabolic $F$-subgroup $\mc P_2$ of $\mc G_{nr} \times
\mc G'$ such that $\mc P_1$ is $(\mc G_{nr} \times \mc G'_r \rtimes \mc C)(F)$-conjugate to
$\mc P_2 \rtimes \mc C$, say by an element $g_2$. Then $\mc P_1$ contains $g_2 \mc P_2
g_2^{-1}$, which is a pseudo-parabolic $F$-subgroup of $\mc G_{nr} \times \mc G'_r$.

From the case $(\mc G_{nr} \rtimes \mc G'_r)(F)$ we know that $K' \mc P_1 (F)$ contains
\[
K' g_2 \mc P_2 (F) g_2^{-1} = (\mc G_{nr} \times \mc G'_r)(F).
\]
The image of $K' \mc P_1 (F)$ in
\[
\mc C (F) = (\mc G_{nr} \times \mc G'_r \rtimes \mc C)(F) \big/ (\mc G_{nr} \times \mc G'_r)(F)
\]
contains
\[
(\mc G_{nr} \times \mc G'_r)(F) g_2 (\mc P_2 \times \mc C)(F) g_2^{-1} \big/
(\mc G_{nr} \times \mc G'_r)(F) = \bar{g_2} \mc C (F) \bar{g_2}^{-1} = \mc C (F) ,
\]
where $\bar{g_2}$ the image of $g_2$ in the quotient. Hence
\[
(K' \rtimes K_{\mc C}) \mc P_1 (F) = (\mc G_{nr} \times \mc G'_r \rtimes \mc C)(F) .
\]
Since $\mc P_1$ contains the central subgroup $\alpha (\mc T')$ we conclude with Lemma
\ref{lem:1.3} that
\begin{multline}
K \mc P (F) \supset \big( (K' \rtimes K_{\mc C}) \mc P_1 (F) / \alpha (T') \big) \mc P (F)
= \big( (\mc G_{nr} \times \mc G'_r \rtimes \mc C)(F) / \alpha (T') \big) \mc P (F) .
\end{multline}
Since $Z_{\mc G_r}(\mc S) \subset Z_{\mc G}(\mc S_\es) \subset \mc P$, we can use 
Lemma \ref{lem:2.5}.b to conclude that the right hand side is none other than $\mc G (F)$.
\end{proof}

Related to the Iwasawa decomposition is the question whether $(\mc G / \mc P) (F)$ is compact,
for any pseudo-parabolic $F$-subgroup $\mc P$. We note that by \cite[Proposition 3.5.7]{CGP},
$\mc G (F) / \mc P (F)$ can be identified with the set of pseudo-parabolic $F$-subgroups of 
$\mc G$ that are $\mc G (F)$-conjugate to $\mc P$. It is known from \cite[Proposition C.1.6]{CGP} 
that the variety $\mc G / \mc P$ is pseudo-complete, but it is unclear (to the author) 
whether or not this  implies compactness of $(\mc G / \mc P) (F)$. 
Fortunately Theorem \ref{thm:2.4} solves this.

\begin{cor}\label{cor:2.8}
Let $\mc P$ be any pseudo-parabolic $F$-subgroup of the pseudo-reductive group $\mc G$.
The natural map $\mc G (F) / \mc P (F) \to (\mc G / \mc P)(F)$ is a homeomorphism and
this space is compact.
\end{cor}
\begin{proof}
By \cite[Corollary 15.1.4]{Spr} the map $\mc G (F) \to (\mc G / \mc P) (F)$ is surjective. The
induced map $\mc G (F) / \mc P (F) \to (\mc G / \mc P)(F)$ is a bijective submersion of analytic
$F$-varieties, in particular a homeomorphism. By Theorem \ref{thm:2.4} already $K \to \mc G (F) /
\mc P (F)$ is surjective. Since $K$ is compact, so is its image $\mc G (F) / \mc P (F)$.
\end{proof}

%\textbf{Remark.}
%Brian Conrad informed the author that Corollary \ref{cor:2.8} can also be proved in a
%completely different way, forgoing maximal compact open subgroups. Namely, since the 
%$\mc G / \mc P$ is insensitive to operations such as central quotients or taking the derived group, 
%we can reduce the issue to the case $\mc G = R_{F_i | F} \mc G_i$ (with $\mc G_i$ as in the proof 
%of Proposition \ref{prop:2.3}). Then $(\mc G / \mc P)(F) \cong (\mc G_i / \mc P_i)(F_i)$ for some 
%pseudo-parabolic $F_i$-subgroup $\mc P_i$ of $\mc G_i$. It is well-known that this space is compact 
%if $\mc G_i$ is reductive. If $\mc G_i$ is basic exotic or basic non-reduced, then one can apply
%results in \cite[\S 11.4]{CGP} to reduce to the reductive case.\\

Recall that $\mc S_\es$ is a maximal $F$-split torus in $\mc G$ and that $\mc P_\es$ is a
minimal pseudo-parabolic $F$-subgroup of $\mc G$ containing $\mc C$ and $\mc S_\es$. We
also recall $Z_G (S_\es)_\cpt$ from Lemma \ref{lem:2.7}.

\begin{thm}\label{thm:2.6}
The group $G$ admits a Cartan decomposition. More explicitly, there exists a finitely
generated semigroup $A \subset Z_G (S_\es)$ such that
\begin{itemize}
\item $G = K A K$ and the natural map $A \to K \backslash G / K$ is bijective;
\item $A$ represents the orbits of the Weyl group $W(\mc G,\mc S_\es)$ on
$Z_G (S_{\es}) / Z_G (S_{\es})_\cpt$.
\end{itemize}
\end{thm}
\begin{proof}
First we consider $G_{nr} \times G'_r \rtimes C$, with the maximal compact
subgroup $K' \rtimes K_{\mc C}$. 

By \eqref{eq:1.37} and \eqref{eq:2.11} $G_{nr} \times G'_r$ is a direct product of semisimple
groups (we call them just $\mc G_i (F_i)$ for convenience) over non-archimedean local fields.
By \cite[Proposition 4.4.3]{BrTi1} each of these $\mc G_i (F_i)$ satisfies a Cartan
decomposition, with respect to the good compact subgroup $\mc G_i (F_i) \cap K'$. It also says
that for $A \cap \mc G_i (F_i)$ we should take a set of representatives for the orbits of
the Weyl group $W(\mc G_i,\mc S_i)$ on
\[
Z_{\mc G_i}(\mc S_i)(F_i) / Z_{\mc G_i}(\mc S_i) (F_i)_\cpt,
\]
namely a cone (essentially the intersection with a positive Weyl chamber) in this lattice.
For the moment, it is more to convenient to let $A_i$ be a group in $Z_{\mc G_i}(\mc S_i)(F_i)$
representating the cosets with respect to $Z_{\mc G_i}(\mc S_i)(F_i)_\cpt$. 

We put $A' = \prod_i A_i$, a finitely generated abelian subgroup of $Z_{G'}(S_\es)$ which represents
$Z_{G'}(S_\es) / Z_{G'}(S_\es)_\cpt$. It intersects $\phi_{\mc C}^{-1}(Z_G (S_\es)_\cpt) \subset 
Z_{G'}(S_\es)_\cpt$ trivially, so dividing out $\alpha (T')$ and then $Z_G (S_\es)_\cpt$ 
sends $A'$ injectively to $Z_G (S_\es) / Z_G (S_\es)_\cpt$. We extend the image of $A'$ in $Z_G (S_\es)$ 
to a set of representatives $A_C$ for $Z_G (S_\es) / Z_G (S_\es)_\cpt$. Since $K'$ contains 
$Z_{G'}(S_\es )_\cpt$, we obtain and $K' A' = K' Z_{G'}(S_\es)$.  We calculate inside $G$, 
temporarily identifying subsets of $G'$ with their image in $G$:
\begin{equation}\label{eq:1.39}
K A_C K = K A' A_C K' K = \bigcup\nolimits_{a \in A_C} K K' Z_{G'}(S_\es) (a K' a^{-1}) a K .
\end{equation}
Here $a K' a^{-1}$ is another good maximal compact subgroup of $G'$, namely the stabilizer of the
special vertex $a x' \in \mh A_\es$. By a more general version of the Cartan decomposition for the
$\mc G_i (F_i)$\cite[Proposition 7.4.15]{BrTi1}, $K' Z_{G'}(S_\es) (a K' a^{-1}) = G'$.
As $K$ contains $Z_{G}(S_\es)_\cpt$, we have $A_C K = Z_{G}(S_\es) K$.
Hence the right hand side of \eqref{eq:1.39} equals
\[
\bigcup\nolimits_{a \in A_C} K G' a K = K G' Z_G (S_\es) K .
\]
By Lemma \ref{lem:2.5}.b the last expression equals $G$, which together with the above implies
\begin{equation}\label{eq:1.40}
K A_C K = G .
\end{equation} 
Now that we know this, we may replace $A_C \subset Z_G (S_\es)$ by any set of representatives for 
$Z_G (S_\es) / Z_G (S_\es)_\cpt$, and then \eqref{eq:1.40} remains true. In the proof of Lemma 
\ref{lem:2.7} we already observed that $Z_G (S_\es) / Z_G (S_\es)_\cpt$ is a lattice in 
$\mh A_\es$. Hence we can it represent by finitely generated subgroup of $Z_G (S_\es)$. 

Of course, some redundancy remains in \eqref{eq:1.40}. The Weyl group $W(\mc G,\mc S_\es)$ can be 
identified with $\prod_i W(\mc G_i,\mc S_i)$, and since $x'$ is special we can represent 
$W(\mc G,\mc S_\es)$ by elements of $K$. For such a $w \in W(\mc G,\mc S_\es)$ represented in $K$,
\[
K a K = K w a w^{-1} K. 
\]
Therefore we may replace $A_C$ by a set $A$ of representatives for the orbits
of $W(\mc G,\mc S_\es)$ on $Z_G (S_{\es}) / Z_G (S_{\es})_\cpt$, and then \eqref{eq:1.40}
implies $G = K A K$. We take $A$ so that $A x'$ is the intersection of $Z_G (S_{\es})  x' = A_C x'$ 
with a positive Weyl chamber for $W(\mc G,\mc S_\es)$ in $\mh A_\es$ (with $x'$ as origin). Since 
$W(\mc G,\mc S_\es)$ is a subgroup of the (extended) affine Weyl group 
\[
N_G (S_\es) / Z_G (S_\es)_\cpt \cong W(\mc G,\mc S_\es) \rtimes Z_G (S_{\es}) / Z_G (S_{\es})_\cpt,
\] 
$A$ is still finitely generated (but now as semigroup).

We claim that this positive Weyl chamber in $Z_G (S_{\es}) x'$ is a fundamental domain for the action 
of $K$ on $G x' \subset \mc B (\mc G,F)$. 
Notice that $W(\mc G,\mc S_\es)$ and $K$ fix the factor $X_* (\mc Z) \otimes_\Z \R$ of $\mc B (\mc G,F)$ 
pointwise. Hence the projection of $G x'$ on $X_* (\mc Z) \otimes_\Z \R$ equals the projection of 
$Z_G (S_{\es}) x'$, which is also the same as the projections of $A_C x'$ and of $A x'$
on $X_* (\mc Z) \otimes_\Z \R$.

Therefore it suffices to check our claim for the other factor $\prod_i \mc BT (\mc G_i,F_i)$ of 
$\mc B (\mc G,F)$. There it reduces to a claim about the action of $G$ on $\mc BT (G_i,F_i)$. There it 
follows from (in fact is equivalent to) the Cartan decomposition with respect to the stabilizer of the
special vertex $x_i \in A_{\mc T_i} \subset \mc BT (\mc G_i,F_i)$.

Finally, suppose that $a,a' \in A$ and $K a K = K a' K$. Then 
\[
K a x' = K a K x' = K a' K x' = K a' x',
\]
where both $a x'$ and $a x$ lie in the positive Weyl chamber in $A_\es$. 
Our claim entails that $a = a'$, which establishes the bijectivity of $A \to K \backslash G / K$.
\end{proof}

\newpage

\section{Quasi-reductive groups}
\label{sec:quasi}

Throughout this section we suppose that $F$ is a non-archimedean local field and that 
$\mc Q$ is a quasi-reductive $F$-group, as defined in \cite[C.2.11]{CGP}. As an $F$-group, 
$\mc Q (F)$ is certainly locally compact and totally disconnected.
By Lemma \ref{lem:1.4} and \eqref{eq:1.35}, it is unimodular.

Let us recall some important results about the $F$-unipotent radical of $\mc Q$, and prove two new ones.

\begin{thm}\label{thm:3.8}
\enuma{
\item $\mc R_{u,F}(\mc Q)$ is a $F$-wound unipotent group \cite[Corollary B.3.5]{CGP}.
\item $\mc R_{u,F}(\mc Q)(F)$ is compact \cite[Th\'eor\`eme VI.1]{Oes}. 
\item Every $F$-torus in $\mc Q$ centralizes $\mc R_{u,F}(\mc Q)$ \cite[Proposition B.4.4]{CGP}.
\item Let $\mc U \subset \mc Q$ be a root subgroup with respect to a nontrivial root of a 
$F$-torus in $\mc Q$. Then $\mc U$ commutes with $\mc R_{u,F}(\mc Q)$.
\item $\mc R_{u,F}(\mc Q)(F)$ has arbitrarily small compact open subgroups which are normalized
by $\mc Q (F)$.
}
\end{thm}
\begin{proof}
(d) Let $\mc T$ be the relevant $F$-torus in $\mc Q$. By part (c) $\mc R_{u,F}(\mc Q) \subset
Z_{\mc Q}(\mc T)$, so $\mc R_{u,F}(\mc Q)$ normalizes $\mc U$. But at the same time
$\mc R_{u,F}(\mc Q)$ is normal in $\mc Q$ and $\mc R_{u,F}(\mc Q) \cap \mc U = \{1\}$,
so $\mc R_{u,F}(\mc Q)$ commutes with $\mc U$.\\
(e) Since $\mc R_{u,F}(\mc Q)(F)$ is totally disconnected, compact and Hausdorff, it is a
profinite group and has small compact open subgroups, say $K$. 

Let $\mc P_\es$ be a minimal pseudo-parabolic subgroup of $\mc Q$ containing 
$Z_{\mc Q}(\mc S_\es)$. By \cite[Propositions C.2.4 and C.2.26]{CGP} 
\begin{equation}\label{eq:3.12}
\mc P_\es = Z_{\mc Q}(\mc S_\es) \ltimes \mc R_{us,F}(\mc P_\es) = Z_{\mc Q}(\mc S_\es) 
\ltimes \prod\nolimits_{\alpha \in \Phi (\mc Q, \mc S_\es)_\red^+} \mc U_\alpha ,
\end{equation}
where $ \Phi (\mc Q, \mc S_\es)^+$ is the positive subsystem of $\Phi (\mc Q, \mc S_\es)$
consisting of all roots that appear in Lie$(\mc R_{us,F}(\mc P_\es))$.
By part (d) $\mc R_{u,F}(\mc Q)$ commutes with $\mc R_{us,F}(\mc P_\es)$. Now the 
Bruhat decomposition (Theorem \ref{thm:2.2} for $\mc Q, \mc P_\es$) shows that
\begin{equation}\label{eq:3.20}
\{ q K q^{-1} : q \in Q \} = \{ n K n^{-1} : n \in N_Q (S_\es) \} . 
\end{equation}
By part (c) $\mc S_\es$ centralizes $\mc R_{u,F}(\mc Q)$. The argument with 
\eqref{eq:2.24}, \eqref{eq:2.19} and \eqref{eq:1.25} shows that there exists an open 
neighborhood $U_D$ of 1 in $Q$, all whose elements normalize $K$. By Proposition \ref{prop:2.1}
for $Z_{\mc Q}(\mc S_\es)$ we may assume that $U_D$ is a compact open subgroup of 
$Z_{Q}(S_\es)$. Then $N_{N_Q (S_\es)}(K)$ contains the open subgroup $S_\es U_D$.
By \eqref{eq:3.26} (whose proof does not rely on the current part e)
$S_\es U_D$ has finite index in $N_Q (S_\es)$. Hence
$[N_{N_Q (S_\es)}(K) : S_\es U_D]$ and the set \eqref{eq:3.20} are finite. Then
$K' := \bigcap_{q \in Q} q K q^{-1}$ is a finite intersection, so in particular another 
compact open subgroup of $\mc R_{u,F}(\mc Q)(F)$. Clearly it is normalized by $Q$. 
Starting with a very small $K \subset \mc R_{u,F}(\mc Q)(F)$, we can make $K' \subset K$ 
arbitrarily small.
\end{proof}

The definition of the unipotent radical implies that $\mc G := \mc Q / \mc R_{u,F}(\mc Q)$ 
is a pseudo-reductive $F$-group.
One could be inclined to deduce from this that the quasi-reductive group $\mc Q (F)$ is an
extension of the pseudo-reductive group $\mc G (F)$ by the compact group $\mc R_{u,F}(\mc Q)(F)$.
However, it seems that this need not be true in general, the correct analogue is more complicated.
Consider the short exact sequence of linear algebraic groups
\begin{equation}\label{eq:3.2}
1 \to \mc R_{u,F}(\mc Q) \to \mc Q \xrightarrow{\pi} \mc G = \mc Q / \mc R_{u,F}(Q) \to 1.
\end{equation}
It induces an exact sequence of $F$-rational points
\begin{equation}\label{eq:3.1}
1 \to \mc R_{u,F}(\mc Q)(F) \to \mc Q (F) \xrightarrow{\pi} \mc G (F) ,
\end{equation}
where the last map need not be surjective. We denote the image of $\pi : \mc Q (F) \to \mc G (F)$
by $\tilde{Q}$.

\begin{thm}\label{thm:3.2}
\enuma{
\item The map $\pi : \mc Q (F) \to \mc G (F)$ is smooth and open with respect to the locally
compact topology. The group $\tilde{Q}$ has finite index in $\mc G (F)$.
\item The group $\tilde{Q}$ contains all maximal $F$-tori of $\mc G (F)$ and all groups
$\mc R_{us,F}(\mc P)(F)$, where $\mc P$ is a pseudo-parabolic $F$-subgroup of $\mc G$.
\item Let $\mc S_\es$ be a maximal $F$-split torus in $\mc G$ and let $\mc P$ be a
pseudo-parabolic $F$-subgroup containing $Z_{\mc G}(\mc S_\es)$. Then $\tilde{Q}$
satisfies the Bruhat decomposition
\[
\tilde{Q} = (\mc P (F) \cap \tilde Q) N_{\tilde{Q}}(S_\es) (\mc P (F) \cap \tilde Q).
\]
}
\end{thm}
\begin{proof}
(a) This is an instance of Proposition \ref{prop:1.7}.\\
(b) Let $\mc S$ be a maximal $F$-torus of $\mc G$. By \cite[Proposition 2.2.10]{CGP}
$\pi^{-1}(\mc S)$ contains a maximal $F$-torus $\mc S'$ of $\mc Q$. By Theorem
\ref{thm:3.8}.c $\mc R_{u,F}(\mc Q)$ commutes with $\mc S'$, and $\mc R_{u,F}(\mc Q) \cap \mc S'
= 1$ because one group is unipotent and the other is semisimple. Then
$\mc S' \times \mc R_{u,F}(\mc Q)$ is a subgroup of $\pi^{-1}(\mc S)$, which by Proposition
\ref{prop:1.7}.b and \eqref{eq:3.2} in fact exhausts $\pi^{-1}(\mc S)$. Hence
\begin{equation}\label{eq:3.3}
\mc S \cong \big( \mc S' \times \mc R_{u,F}(\mc Q) ) / \mc R_{u,F}(\mc Q) \cong \mc S'
\end{equation}
and $\pi : \mc S' (F) \to \mc S (F)$ is an isomorphism.

By \cite[Proposition 2.2.10]{CGP} $\mc P' := \pi^{-1}(\mc P)$ is a pseudo-parabolic $F$-subgroup
of $\mc Q$. In view of \cite[Corollary 2.2.5]{CGP},
\[
\mc R_{u,F}(\mc P') = \pi^{-1}(\mc R_{u,F}(\mc P)) =  \pi^{-1}(\mc R_{us,F}(\mc P))
\]
is generated by two normal unipotent subgroups: $\ker \pi = \mc R_{u,F}(\mc Q)$ and
$\mc R_{us,F}(\mc P')$. These two subgroups commute by Theorem \ref{thm:3.8}.d. As in
\cite[\S 2.2]{CGP}, let $\lambda : GL_1 \to \mc Q$ be a cocharacter such that $\mc P' = 
\mc P_{\mc Q}(\lambda) \mc R_{u,F}(\mc Q)$. Then $\mc R_{u,F}(\mc Q) \subset \mc Z_{\mc Q}(\lambda)$
and $\mc R_{us,F}(\mc P') \subset \mc U_{\mc Q}(\lambda)$, so by \cite[Proposition 2.1.8.(2)]{CGP}
their intersection is trivial. We deduce that
\[
\mc R_{u,F}(\mc P') = \mc R_{u,F}(\mc Q) \times \mc R_{us,F}(\mc P') 
\]
and $\mc R_{us,F}(\mc P') \cong \mc R_{us,F}(\mc P)$. In particular
\begin{equation}\label{eq:3.4}
\pi : \mc R_{us,F}(\mc P')(F) \to \mc R_{us,F}(\mc P)(F) \text{ is an isomorphism.}
\end{equation}
This proves that $\pi (\mc Q (F))$ contains the required subgroups.\\
(c) It suffices to consider the case of a minimal pseudo-parabolic $F$-subgroup
$\mc P_\es$ of $\mc G$ containing $Z_{\mc G}(\mc S_\es)$. Let $N$ be the smallest
normal subgroup of $\mc G (F)$ which contains all the subgroups $\mc R_{us,F}(\mc P_\es)(F)$.
By \cite[Proposition C.2.2.4]{CGP} $N$ contains representatives for the Weyl group of
$(\mc G,\mc S_\es)$. So does $\tilde{\mc Q}(F)$, because $N \subset \tilde{Q}$
by part (b). Recall the Bruhat decomposition -- Theorem \ref{thm:2.2} with
representatives for $W(\mc G,\mc S_\es)$ in $\tilde{Q}$:
\[
\mc G (F) = \mc P_\es (F) N_{\mc G}(\mc S_\es)(F) \mc P_\es (F) =
\mc P_\es (F) Z_{\mc G}(\mc S_\es)(F) W(\mc G,\mc S_\es) \mc P_\es (F) .
\]
With \eqref{eq:3.12} we can rewrite this as
\begin{equation}\label{eq:3.5}
\mc G (F) = Z_{\mc G}(\mc S_\es)(F)
\mc R_{us,F}(\mc P_\es)(F) W(\mc G,\mc S_\es) \mc R_{us,F}(\mc P_\es) (F) .
\end{equation}
By part (b) the three rightmost terms are contained in $\tilde Q$, so we obtain
\begin{align*}
\tilde Q & = Z_{\tilde Q}(S_\es) \mc R_{us,F}(\mc P_\es)(F) W(\mc G,\mc S_\es) 
\mc R_{us,F}(\mc P_\es) (F) \\
& = \mc R_{us,F}(\mc P_\es)(F) Z_{\tilde Q}(S_\es) W(\mc G,\mc S_\es) Z_{\tilde Q}(S_\es) 
\mc R_{us,F}(\mc P_\es) (F) \\
& = (\mc P_\es (F) \cap \tilde Q) N_{\tilde{Q}}(S_\es) (\mc P_\es (F) \cap \tilde Q) . \qedhere
\end{align*}
\end{proof}

\subsection{Valuated root data} \
\label{subsec:BT}

We would like to generalize our results for pseudo-reductive groups to quasi-reductive groups.
In spite of Theorem \ref{thm:3.2} this is not so trivial, and we find it convenient to use 
more Bruhat--Tits theory. Let $\mc S_\es$ be a maximal $F$-split torus in $\mc Q$. By 
\eqref{eq:3.3} we may identify it with its image in $\mc G$, which we also denote by $\mc S_\es$.

Let $\Phi (\mc Q,\mc S_\es)$ be the root system of $\mc Q$ with respect to the maximal $F$-split 
torus $\mc S_\es$. Since $\mc R_{u,F}(\mc Q)$ commutes with $\mc S_\es$, $\Phi (\mc Q,\mc S_\es)$ 
can be identified with $\Phi (\mc G,\mc S_\es)$ \cite[Theorem C.2.15]{CGP}. In particular
they have the same Weyl group 
\begin{equation}\label{eq:3.51}
N_{\mc Q}(\mc S_\es) / Z_{\mc Q}(\mc S_\es) = W(\mc Q,\mc S_\es) \cong 
W(\mc G, \mc S_\es) = N_{\mc G}(\mc S_\es) / Z_{\mc G}(\mc S_\es) .
\end{equation}

\begin{lem}\label{lem:3.16}
\enuma{
\item $Z_{\mc Q}(\mc S_\es) / \mc S_\es$ is $F$-anisotropic and $Z_Q (S_\es) / S_\es$ is compact.
\item $Z_Q (S_\es)$ has a unique maximal compact subgroup $Z_Q (S_\es)_\cpt$, and the 
quotient $Z_Q (S_\es) / Z_Q (S_\es)_\cpt$ is a finitely generated free abelian group.
\item $Z_Q (S_\es)$ has arbitrarily small compact open subgroups which are normal in
$N_Q (S_\es)$.
} 
\end{lem}
\begin{proof}
(a) Since $\mc R_{u,F}(\mc Q) \cap \mc S_\es = \{1\}$, there is a short exact sequence
\[
1 \to \mc R_{u,F}(\mc Q) \to Z_{\mc Q}(\mc S_\es) / \mc S_\es \to 
Z_{\mc G}(\mc S_\es) / \mc S_\es \to 1 .
\]
By Lemma \ref{lem:1.6}.b the outer terms are $F$-anisotropic, so the middle term is also
$F$-anisotropic. By Proposition \ref{prop:1.8} $(Z_{\mc Q}(\mc S_\es) / \mc S_\es)(F)$ and
its closed subgroup $Z_Q (S_\es) / S_\es$ are compact.\\
(b) Clearly $\pi (Z_{\mc Q}(\mc S_\es)) \subset Z_{\mc G}(\mc S_\es)$.
By Theorem \ref{thm:3.2}, the Bruhat decompositions for $\mc Q$ and $\mc G$ with respect to 
$\mc S_\es$ and \eqref{eq:3.51}, $Z_{\mc Q}(\mc S_\es)$ 
must be as large as possible with this constraint. That is,
\begin{equation}\label{eq:3.52}
Z_{\mc Q}(\mc S_\es) = \pi^{-1} (Z_{\mc G}(\mc S_\es)).
\end{equation}
Now first the statement follows immediately from Lemma \ref{lem:2.7} and Theorem \ref{thm:3.8}.b.
Let $Z_Q (S_\es)$ act on $\mh A_\es$ via $\pi : Z_Q (S_\es) \to Z_G (S_\es)$.
In the proof of Lemma \ref{lem:2.7} we noted that $Z_G (S_\es)_\cpt$ equals the 
$Z_G (S_\es)$-stabilizer of any point of $\mh A_\es$, so the same goes for its preimage
$Z_Q (S_\es)_{cpt}$ in $Z_Q (S_\es)$. Consequently $Z_Q (S_\es) / Z_Q (S_\es)_\cpt$ acts
freely on $\mh A_\es$ by translations. From Lemma \ref{lem:1.1} we see that any 
$Z_G (S_\es)$-orbit in $\mh A_\es$ is discrete and cocompact, so $Z_Q (S_\es) / Z_Q (S_\es)_\cpt
\cong \Z^{\dim \mh A_\es}$.\\
(c) Proposition \ref{prop:2.1} for $Z_{\mc Q}(S_\es)$ (with the empty root system) yields
a decreasing family of compact open subgroups $K_n \subset Z_Q (S_\es)$ with 
$\bigcap_{n=1}^\infty K_n = \{1\}$. By part (a) and \eqref{eq:2.26} 
\begin{equation}\label{eq:3.26}
N_Q (S_\es) / S_\es \text{ is compact.} 
\end{equation}
The normalizer of $K_n$ in $N_Q (S_\es)$ contains the
cocompact open subgroup $S_\es K_n$, so it has finite index in $N_Q (S_\es)$. Therefore
\[
\bigcap\nolimits_{q \in N_Q (S_\es)} q K_n q^{-1}  
\]
is an intersection of only finitely many terms, and is open in $Z_Q (S_\es)$. By 
construction this intersection is also a compact normal subgroup of $N_Q (S_\es)$.
\end{proof}

Each $\alpha \in \Phi (\mc Q,\mc S_\es)$ gives rise to a root subgroup $\mc U_\alpha$ in 
$\mc Q$, see \cite[C.2.21]{CGP}. Let $m_\alpha = m_\alpha (u_\alpha) \in N_Q (S_\es)$ 
be as in \cite[Proposition C.2.24]{CGP} and define $M_\alpha = m_\alpha Z_Q (S_\es)$.

\begin{thm}\label{thm:3.3}
The group $Q = \mc Q (F)$ has a generating root datum
\[
\big( Z_Q (S_\es), (U_\alpha,M_\alpha)_{\alpha \in \Phi (\mc Q,\mc S_\es)} \big)
\]
in the sense of \cite[\S 6.1]{BrTi1}. It admits a valuation \cite[\S 6.2]{BrTi1}.
\end{thm}
\begin{proof}
It was noted in \cite[Remark C.2.28]{CGP} that these $F$-subvarieties satisfy the axioms
of a generating root datum. 

By \eqref{eq:3.4} every $U_\alpha$ maps isomorphically to its
image in $G = \mc G (F)$ and by Theorem \ref{thm:3.2}.c the $U_\alpha$ and
$Z_{\tilde Q}(S_\es)$ generate $\tilde Q = \pi (\mc Q (F))$. Hence we can transform the above
into a generating root datum for the group $\tilde Q$:
\begin{equation}\label{eq:3.9}
\big( Z_{\tilde Q} (S_\es), (U_\alpha,\pi (m_\alpha ) Z_{\tilde Q}(S_\es)
)_{\alpha \in \Phi (\mc Q,\mc S_\es)} \big) .
\end{equation}
A valuation of either of these two generating root data consists of a family of maps
$\phi_\alpha : U_\alpha \to \R \cup \{\infty\}$ satisfying several conditions
\cite[D\'efinition 6.2.1]{BrTi1}. These conditions involve only the groups $U_\alpha$ and the
sets $M_\alpha$. Since $\pi : Q \to G$ is bijective on the $U_\alpha$ and $\pi (M_\alpha) =
\pi (m_\alpha ) Z_{\tilde Q}(S_\es)$, it suffices to show that the root datum for
$\tilde Q$ admits a valuation. This in turn follows if we show that the generating root datum
\begin{equation}\label{eq:3.8}
\big( Z_{G} (S_\es), (U_\alpha,\pi (m_\alpha ) Z_{G}(S_\es) )_{\alpha \in \Phi (\mc G,\mc S_\es)} \big)
\end{equation}
for $G$ admits a valuation.

Recall the shape of the pseudo-reductive group $\mc G$, which we described in \eqref{eq:2.6}
and the lines before that. Equations \eqref{eq:2.3}, \eqref{eq:2.11} and \eqref{eq:1.18} show
how $\mc G (F)$ is constructed from some reductive $F_i$-groups $\mc G_i (F_i)$.
As before, let $\mc T_i$ be a maximal $F_i$-torus in $\mc G_i$, such that $\mc S_\es \cap
R_{F_i / F} (\mc G_i)$ is contained in $R_{F_i / F}(\mc T_i)$. By \cite[Proposition A.5.15]{CGP}
$\mc T_i$ contains a maximal $F_i$-split torus $\mc S_i$ of $\mc G_i$. Then
\[
\Phi (\mc G,\mc S_\es) = \bigsqcup\nolimits_i \Phi (R_{F_i / F} (\mc G_i) \mc S_\es,\mc S_\es)
\quad \text{can be identified with} \quad \bigsqcup\nolimits_i \Phi (\mc G_i, \mc S_i) .
\]
Suppose that $\alpha_i \in \Phi (\mc G_i, \mc S_i)$ corresponds to $\alpha \in \Phi (\mc G,\mc S_\es)$.
From \eqref{eq:2.6} we see that the root subgroup $\mc U_{\alpha_i}(F_i) \subset \mc G_i (F_i)$ is
mapped isomorphically to $\mc U_\alpha (F) \subset \mc G (F)$. The element $m_\alpha (u)$ can also
be constructed entirely in $\mc G_i (F_i)$, and then be mapped to $\mc G (F)$ via \eqref{eq:2.6}.

The main result of \cite{BrTi2} says that the generating root datum
\begin{equation}\label{eq:2.23}
\big( Z_{\mc G_i}(\mc S_i) (F_i), ( \mc U_{\alpha_i}(F_i), M_{\alpha_i} )_{\alpha_i \in
\Phi (\mc G_i, \mc S_i)} \big)
\end{equation}
for $\mc G_i (F_i)$ admits a valuation. We use the maps $\phi_{\alpha_i} : \mc U_{\alpha_i}(F_i) \to
\R \cup \{ \infty \}$ from such a valuation, and we retract them to
$\phi_{\alpha} : \mc U_{\alpha}(F) \to \R \cup \{ \infty \}$. Then the shape of $\mc G$, as in
Lemma \ref{lem:2.5}, immediately implies that the all conditions of \cite[D\'efinition 6.2.1]{BrTi1},
except possibly (V2), remain valid for the $\phi_\alpha$. The condition (V2) states that,
for all $\alpha \in \Phi (\mc G,\mc S_\es)$ and all $m \in M_\alpha$, the function 
\[
\mc U_{-\alpha} (F) \setminus \{1\} \to \R : u \mapsto \phi_{-\alpha}(u) - \phi_\alpha (m u m^{-1})
\]
is constant. Recall that $M_\alpha = \pi (m_\alpha) Z_G (S_\es)$. Condition (V2) holds for 
$\pi (m_\alpha)$ because it boils to a statement entirely in $\mc G_i (F_i)$, where we know that
we have a valuation. Therefore it suffices to show that, for every $z \in Z_G (S_\es)$, 
\begin{equation}\label{eq:3.31}
\mc U_\alpha (F) \setminus \{1\} \to \R : u \mapsto \phi_\alpha (u) - \phi_\alpha (z u z^{-1}) 
\qquad \text{is constant.}
\end{equation}
The construction of $\mc G$ entails that the conjugation action $c(z)$ of $z$ on $\mc U_\alpha$ 
is computed via elements of $Z_{\mc G'}(\mc S_\es) = \prod_i R_{F_i/F} Z_{\mc G_i}(\mc S_\es)$ and of 
$\prod_i R_{F_i/F} (\mc T_i / Z(\mc G_i))$. By the surjectivity of $\mc G' \rtimes \mc C \to \mc G$
and $\mc T_i \to \mc T_i / Z(\mc G_i)$ on rational points over a separably closed field, there exists 
a finite Galois extension $F'/F$ and elements
$z_1 \in Z_{\mc G'}(\mc S_\es)(F'), z_2 \in \prod_i R_{F_i/F} (\mc T_i)(F')$ such that
$c(z)$ is conjugation by $z_1$ composed with conjugation by $z_2$. The valuations of \eqref{eq:2.23}
constructed by Bruhat and Tits are schematic, in particular they are compatible with finite Galois
extensions, in the sense that the $F$-rational points can be obtained from the $F'$-rational points
by intersecting with $\mc U_\alpha (F)$. Hence conjugation with $z_1$ satisfies \eqref{eq:3.31}, 
and the same for $z_2$. Thus (V2) holds, which provides the desired valuations on the root data 
\eqref{eq:3.8} and \eqref{eq:3.9}.
\end{proof}

Bruhat and Tits also devised the notion of a prolonged valuation \cite[D\'efinition 6.4.38]{BrTi1}.
It presumes the existence of a valuated root datum for a group $G$, with in particular a subgroup 
$U_0$ (called $T$ in \cite{BrTi1}), a root system $\Phi$ and root subgroups $U_\alpha$ for 
$\alpha \in \Phi$. The valuation gives rise to subgroups
\begin{equation}\label{eq:3.54}
U_{\alpha,r} := \phi_\alpha^{-1} ([r,\infty]) \qquad \alpha \in \Phi, r \in \R .
\end{equation}
As observed in \cite[p. 111]{ScSt} any ``good filtration" of $U_0$ gives rise to a prolongation 
of the valuation of a generating root datum for a locally compact group $G$. 

\begin{defn}\label{defn:0}
A good filtration of $U_0$ is a family of subgroups 
$U_{0,k} \; (k \in \R)$ of $U_0$, such that:
\begin{itemize}
\item[(i)] for $k \leq 0$, $U_{0,k}$ is a maximal compact subgroup of $U_0$;
\item[(ii)] $U_{0,k} \supset U_{0,\ell}$ for $k \leq \ell$;
\item[(iii)] $U_{0,k}$ lies between certain subgroups $H_{[k]}$ and $H_{(k)}$ of $U_0$.
\item[(iv)] $[U_{0,k},U_{0,\ell}] \subset U_{0,k+\ell}$ for all $k,\ell \in \R_{\geq 0}$.
\end{itemize}
\end{defn}
We recall from \cite[6.4.13]{BrTi1} that $H_{(k)}$ consists of those $h \in U_0$ such that 
\begin{equation}\label{eq:3.53}
[h,U_{\alpha,r}] \subset U_{\alpha,r+k} U_{2 \alpha, 2r + k} 
\qquad \forall \alpha \in \Phi, \forall r \in \R .
\end{equation}
By \cite[6.4.14]{BrTi1} the group $H_{[k]}$ for $G$ is the intersection of $U_0$ with 
the group generated by certain compact subgroups of the $U_\alpha$ for $\alpha \in \Phi$.

\begin{prop}\label{prop:3.17}
There exist compact open normal subgroups of $N_Q (S_\es)$ which form a good filtration of
$Z_Q (S_\es)$ and provide a prolongation of the valuation of the generating root datum from 
Theorem \ref{thm:3.3}. 
\end{prop}
\begin{proof}
First we will settle this for the pseudo-reductive group $\mc G$, with $U_0 = Z_G (S_\es)$.
It was shown in \cite[Proposition I.2.6]{ScSt} that the valuations of the root data from 
\eqref{eq:2.23} used above can be prolonged. In fact there exists a natural prolongations 
\cite[\S 5]{Yu}, which give compact open subgroups of $Z_{\mc G_i}(\mc S_i)(F_i)$ that are stable
under all automorphisms of $Z_{\mc G_i}(\mc S_i)$ as an $\mc O_{F_i}$-scheme.

Via the same procedure as in part (b), these give rise to a prolongation of the valuation of the 
generating root datum for $\mc G' (F)$ 
(which is just like \eqref{eq:3.8}, only with $G$ replaced by $G'$). 

This yields such subgroups $U'_{0,k}$ for $\mc G' (F)$, that is, in $Z_{G'}(S_\es)$. 
In fact these come from $\mc O_F$-schemes \cite[Corollary 8.8]{Yu}, and they are characteristic 
subgroups of $Z_{\mc G'}(\mc S_\es)$ in the category of $\mc O_F$-schemes. The actions of
$(N_{\mc G'}(\mc S_\es) \rtimes \mc C)(F)$ and of $N_{\mc G}(\mc S_\es)(F)$ on $Z_{G'}(S_\es)$ 
are via automorphisms of $\mc O_F$-schemes, so $N_{G'}(S_\es) \rtimes C$ and $N_G (S_\es)$ 
normalize every $U'_{0,k}$.

At the same time we get analogous subgroups of $\mc G' (F')$ for every (finite Galois) 
extension $F'/F$. This yields a consistent definition of a subgroup 
$\psi_{\mc C}^{-1}(U'_{0,k}) \subset C$, for every $k \in \R_{\geq 0}$. 
The same argument as in the proof of part (b) shows that 
\begin{equation}\label{eq:3.49}
[\psi_{\mc C}^{-1}(U'_{0,k}), U'_{0,\ell}] \subset U'_{0,k+\ell} 
\text{ for all } k,\ell \in \R_{\geq 0}.
\end{equation}
Hence the $U'_{0,k} \rtimes \psi_{\mc C}^{-1}(U'_{0,k})$ with $k \in \R_{\geq 0}$ form a 
family of subgroups of $G' \rtimes C$ which satisfies (ii) and (iv). 
Notice that for $k = 0$ we obtain
\begin{equation}\label{eq:3.33}
Z_{G'} (S_\es)_\cpt \rtimes \psi_{\mc C}^{-1}(U'_{0,0}) = (Z_{G'} (S_\es) \rtimes C)_x , 
\end{equation}
for any point $x$ in the standard apartment of $\mc B (\mc G',F)$.

However, $\psi_{\mc C}^{-1}(U'_{0,k})$ is compact if and only if $Z(\mc G)(F)$ is compact. 
When $Z (\mc G)(F)$ is not compact, we need to do more work. Applying Proposition \ref{prop:2.1} 
to $\mc C$ with the trivial root system gives us a decreasing sequence of compact open subgroups 
$C_n \; (n \in \N)$ of $C$, with $\bigcap_{n=1}^\infty C_n = \{1\}$. We consider
\begin{equation}\label{eq:3.32}
U'_{0,k} \rtimes \big( \psi_{\mc C}^{-1}(U'_{0,k}) \cap C_{\lceil k \rceil} \big) 
\qquad k \in \R ,
\end{equation}
where $\lceil k \rceil \in \Z$ denotes the ceiling of $k$.

The group $(\mc S' \rtimes \mc S)(F)$, with $\mc S'$ as in \eqref{eq:1.46}, is contained in 
$Z_{G'} (S_\es) \rtimes C$ and centralizes $\mc C$. Thus it normalizes \eqref{eq:3.32}. 
By Lemma \ref{lem:1.6}.b $(\mc S' \rtimes \mc S)(F)$ is cocompact in $N_{G'} (S_\es) \rtimes C$, 
and \eqref{eq:3.32} is compact and open in $N_{G'} (S_\es) \rtimes C$. Therefore the normalizer
of \eqref{eq:3.32} has finite index in $N_{G'} (S_\es) \rtimes C$. In view of Proposition
\ref{prop:1.7}.c dividing out the central subgroup $\alpha (T')$ maps $N_{G'} (S_\es) \rtimes C$ 
to a finite index subgroup of $N_G (S_\es)$. Hence the normalizer of 
\eqref{eq:3.32} in $N_G (S_\es)$ has finite index in there. Then the intersection in
\begin{equation}\label{eq:3.50}
\bigcap\nolimits_{g \in N_{G} (S_\es)} g \big( U'_{0,k} \rtimes 
\big( \psi_{\mc C}^{-1}(U'_{0,k}) \cap C_{\lceil k \rceil} \big) \big) g^{-1}
\end{equation}
is essentially over finitely many terms, and the result is again open in $Z_{G'} (S_\es) \rtimes C$. 
Moreover \eqref{eq:3.50} is a compact subgroup of $N_{G'}(S_\es) \rtimes C$, normalized by
$N_G (S_\es)$. Since $N_G (S_\es) \rtimes C$ normalizes $U'_{0,k}$, the coordinates of 
the elements of \eqref{eq:3.50} in $Z_{G'}(S_\es)$ are precisely $U'_{0,k}$. Therefore 
\eqref{eq:3.50} has the form $U'_{0,k} C'_k$ for a compact open subgroup 
\[
C'_k \subset \psi_{\mc C}^{-1}(U'_{0,k}) \cap C_{\lceil k \rceil} .
\]
For $k \in \R_{\leq 0}$, we put $U_{0,k} = Z_G (S_\es)_\cpt$, which is compact, open and
normal in $N_G (S_\es)$ by Lemma \ref{lem:2.7}. For $k \in \R_{>0}$ we define $U_{0,k}$ to be 
the image of $U'_{0,k} C'_k$ in $G$. By the above construction and Proposition \ref{prop:1.7}, 
every $U_{0,k}$ is a compact open normal subgroup of $N_G (S_\es)$.

Now property (i) holds by definition, while the property (ii) for 
$\tilde U_{0,k} \rtimes \psi_{\mc C}^{-1}(U'_{0,k})$ immediately implies it also for $U_{0,k}$. 
Since $C$ is commutative, \eqref{eq:3.49} entails that $U'_{0,k} C'_k$ and $U_{0,k}$ satisfy (iv).
The description of $G$ as in \eqref{eq:2.6} implies that $H_{[k]}$ for $G$ is just the image of 
$H_{[k]}$ for $G'$, which is contained in $U'_{0,k}$ by (iii). So $H_{[k]} \subset U_{0,k}$, 
as required.
All elements of $U'_{0,k}$ have the property \eqref{eq:3.53} and 
\[
U'_{0,k} C'_k \subset U'_{0,k} \psi_{\mc C}^{-1}(U'_{0,k}) .
\]
Hence taking commutators with any element of $U_{0,k}$ also behaves as in \eqref{eq:3.53} and 
$U_{0,k} \subset H_{(k)}$. We conclude that (iii) holds for the groups $U_{0,k}$, and that 
they define a prolongation of the valuated root datum for $G$.

To prolong the valuation of the generating root datum for $Q$, we can take the groups
$\pi^{-1}(U_{0,k})$ with $k \in \R$. By \eqref{eq:3.52} these are subgroups of the group
$Z_Q (S_\es)$, which plays the role of $U_0$ for $Q$. The properties (ii)--(iv) of 
the groups $U_{0,k}$ immediately imply them also for the groups $\pi^{-1}(U_{0,k})$. By Theorem 
\ref{thm:3.8}.c $\pi^{-1}(N_G (S_\es)) = N_Q (S_\es)$, so the $\pi^{-1}(U_{0,k})$ are open normal 
subgroups of $N_Q (S_\es)$. Theorem \ref{thm:3.8}.b shows that they are compact, and that 
\[
\pi^{-1}(U_{0,k}) = \pi^{-1}(Z_G (S_\es)_\cpt) = Z_Q (S_\es)_\cpt 
\quad \text{for } k \leq 0 . \qedhere
\]
\end{proof}

Notice that, when $\mc Q$ is not pseudo-reductive, the intersection $\bigcap_{k \in \R} U_{0,k} =\\
\mc R_{u,F}(\mc Q)(F)$ is more than just the identity element. It would be useful to find a
prolonged valuation with $\bigcap_{k \in \R} U_{0,k} = \{1\}$, but we did not achieve this
in general.

\subsection{Affine buildings} \

Let $\mc Z$ be a maximal $F$-split torus in $Z(\mc Q)$. Its image in the 
commutative quasi-reductive $F$-group $\mc Q / \mc D (\mc Q) \mc R_{u,F}(\mc Q)$ is a maximal
$F$-split torus $\mc Z$' of the same rank. In particular we can identifty $X_* (\mc Z') 
\otimes_\Z \R$ with $X_* (\mc Z) \otimes_\Z \R$. We let $\mc Q (F)$ act on this space via 
$\mc Q \to \mc Q / \mc D (\mc Q) \mc R_{u,F}(\mc Q)$ and then as in Lemma \ref{lem:1.1}.
Notice that this action is by translations, and in particular by isometries.

For use in Paragraph \ref{par:adm} we introduce and investigate a subgroup of $Q$ which plays a 
role analogous to the kernel of all unramified characters of $Q$ (if $Q$ were reductive). 
The action of $Q$ on $X_* (\mc Z) \otimes_\Z \R$ yields a group homomorphism
\begin{equation}\label{eq:3.30}
Q \to X_* (\mc Z) \otimes_\Z \R : q \mapsto q \cdot 0 . 
\end{equation}
We define ${}^0 Q$ as the kernel of this homomorphism, or equivalently as the isotropy group
of $0 \in X_* (\mc Z) \otimes_\Z \R$ in $Q$.

\begin{lem}\label{lem:3.15}
\enuma{
\item The group ${}^0 Q$ is an open normal, unimodular subgroup of $Q$.
\item ${}^0 Q$ contains every compact subgroup of $Q$.
\item The intersection ${}^0 Q \cap Z(Q)$ is compact and $Q / {}^0 Q Z(Q)$ is finite.
} 
\end{lem}
\begin{proof}
(a) From the proof of Lemma \ref{lem:1.1} we see that the stabilizer in \\
$(\mc Q / \mc D (\mc Q) \mc R_{u,F}(\mc Q))(F)$ of any point of $X_* (\mc Z) \otimes_\Z \R$ 
is open and normal in that group. Hence ${}^0 Q$ is an open normal subgroup of $Q$. 
Since $Q$ is unimodular (Lemma \ref{lem:1.4}), so is ${}^0 Q$.\\
(b) The vector space $X_* (\mc Z) \otimes_\Z \R$ has no compact subgroups besides $\{0\}$. 
Hence the image of any compact subgroup $K$ of $Q$ under \eqref{eq:3.30} is $\{0\}$, which
means that $K \subset {}^0 Q$.\\
(c) The quotient $Q / {}^0 Q$ injects into $X_* (\mc Z) \otimes_\Z \R$. We note that by
\eqref{eq:1.5} and \eqref{eq:1.4} the image is a lattice of full rank in $X_* (\mc Z) 
\otimes_\Z \R$. The image of $\mc Z (F) \subset Z(Q)$ is the sublattice $X_* (\mc Z)$,
which has finite index in any larger lattice in $X_* (\mc Z) \otimes_\Z \R$. Hence 
$Q / {}^0 Q \mc Z (F)$ and $Q / {}^0 Q Z(Q)$ are finite.

The group $Z(Q)$ acts on $X_* (\mc Z) \otimes_\Z \R$ via $\pi : Z(Q) \to \mc C (F)$. 
The $\mc C (F)$-action is proper by Lemma \ref{lem:1.1}, so the isotropy group
$\mc C (F)_0$ is compact. By Theorem \ref{thm:3.8}.b so are $\pi^{-1} (\mc C (F)_0)$
and its closed subgroup ${}^0 Q \cap Z(Q)$.
\end{proof}

The big advantage of Theorem \ref{thm:3.3} is that now many results of Bruhat and Tits apply to
$Q = \mc Q(F)$, in particular it gives rise an affine building $\mc{BT}(\mc Q,F)$. Then 
$\mc{BT}(\mc Q,F) \times X_* (\mc Z) \otimes_\Z \R$ is an affine building -- maybe not literally 
according to \cite{BrTi1}, but in the sense of \cite[\S 2.1]{Tit}. The remarks in \cite{Tit}
entail that the results from \cite{BrTi1} are valid just as well for such an extended building.

\begin{thm}\label{thm:3.4}
\enuma{
\item $Q$ contains a double Tits system.
\item This double Tits system gives rise to a thick affine building $\mc{BT}(\mc Q,F)$ with 
an isometric, cocompact $Q$-action.
\item $\mc{BT}(\mc Q,F) \times X_* (\mc Z) \otimes_\Z \R$ is canonically homeomorphic to 
$\mc B (\mc G,F)$.
\item We endow $\mc{BT}(\mc Q,F) \times X_* (\mc Z) \otimes_\Z \R$ with the diagonal $Q$-action 
and we let $Q$ act on $\mc B (\mc G,F)$ via $\pi : \mc Q (F) \to \mc G (F)$. Then part (c) becomes 
an isomorphism of $Q$-spaces. 
\item The $Q$-action on $\mc{BT}(\mc Q,F) \times X_* (\mc Z) \otimes_\Z \R$ is isometric, proper and 
cocompact. Moreover $\mc{BT}(\mc Q,F) \times X_* (\mc Z) \otimes_\Z \R$ is a universal space for
proper $Q$-actions \cite[\S 1]{BCH}.
}
\end{thm}
\begin{proof}
(a) With Theorem \ref{thm:3.3} at hand, this is \cite[Th\'eor\`eme 6.5.i]{BrTi1}. \\
(b) This is worked out in \cite[\S 2 and \S 6]{BrTi1}. For later use we recall some parts of the
construction. Firstly, a standard apartment $\mh A_0$ with an action of $N_Q (S_\es)$ is exhibited in
\cite[Proposition 6.2.5]{BrTi1}. Next, a standard Iwahori subgroup $B$ of $Q$ is built, and a
set of chambers is defined as the collection of subgroups of $Q$ that are conjugate to $B$
(in other words, the set of Iwahori subgroups of $Q$). This yields an affine building
$\mc BT (\mc Q,F)$, on which $Q$ acts by \cite[Th\'eor\`eme 6.5.ii]{BrTi2}. It is naturally a direct
product of affine buildings whose facets are simplices and whose underlying affine Weyl group is 
irreducible \cite[Proposition 2.6.4]{BrTi1}. Every such affine building is thick, in the sense
that every simplex of codimension is contained in the closure of at least three chambers.\\
(c) The group $Z_Q (S_\es)$ acts by translations on $\mh A_0$, and Bruhat and Tits call the isotropy
subgroup (of any point of $\mh A_0$) $H$. Comparing with the proof of Lemma \ref{lem:2.7}, we see that
$H$ contains the unique maximal compact subgroup of $Z_Q (S_\es)$. In particular $H$ contains the
compact normal subgroup $\mc R_{u,F}(\mc Q)(F)$ of $Q$.

The standard Iwahori subgroup $B$ contains $H$, so in particular $\mc R_{u,F}(\mc Q)(F) \subset B$.
Since all Iwahori subgroups of $Q$ are conjugate, $\mc R_{u,F}(\mc Q)(F)$ is contained in every one
of them. But every Iwahori subgroup fixes the associated chamber of $\mc{BT}(\mc Q,F)$ pointwise,
so $\mc R_{u,F}(\mc Q)(F)$ fixes this entire building pointwise. In view of the construction of
the valuated root data in the proof of Theorem \ref{thm:3.3} (first for $\tilde Q$, from there
for $Q$), this entails that $\mc{BT}(\mc Q,F)$ is the same as the affine building
$\mc{BT}(\tilde Q)$ associated to the analogous double Tits system for $\tilde Q$.

The latter double Tits system is contained in the subgroup $\tilde Q'$ of $\tilde Q$ generated
by the root subgroups $\mc U_\alpha (F)$ and $\tilde H$ (the image of $H$ in $\tilde Q$).
In effect, restriction from $\tilde Q$ to $\tilde Q'$ means that $Z_{\tilde Q}(S_\es)$ is
replaced by $\tilde{H} X$, where $X$ is a finitely generated free abelian subgroup of
$Z_{\tilde Q}(S_\es)$ which acts on $\mh A_0$ as the translation part of the affine Weyl group
$\mh W$ from \cite[6.2.11]{BrTi1}. Here the $\mh W$-action is generated by the affine reflections
of $\mh A_0$ coming from elements $m_\alpha (u) \in N_{\tilde Q}(S_\es)$ as in the definition of
$M_\alpha$. Thus the entire building $\mc{BT}(\mc Q,F)$ depends only on $\tilde Q'$.

Similarly the affine building $\mc{BT}(\mc G,F)$ depends only on the subgroup $G''$ of $G$
generated by the $U_\alpha$ and $H''$ (the $H$ for $G$). From Theorem \ref{thm:3.2} we know that
the difference between $\tilde Q'$ and $G''$ is small, in a precise sense. Together with
\cite[6.2.11]{BrTi1} we find first that $\tilde Q'$ and $G''$ have the same image in the group of
affine transformations of $\mh A_0$, and then that their difference comes entirely from the
finite index inclusion $\tilde H \subset H''$. As in the proof of Theorem \ref{thm:3.2}.c, one
can find Bruhat decompositions for $\tilde Q'$ and $G''$:
\begin{equation}\label{eq:3.10}
\begin{aligned}
G'' = H'' X & \mc R_{u,F}(\mc P_\es)(F) W(\mc G,\mc S_\es) \mc R_{u,F}(\mc P_\es)(F) , \\
\tilde Q' = \tilde H X & \mc R_{u,F}(\mc P_\es)(F) W(\mc G,\mc S_\es) \mc R_{u,F}(\mc P_\es)(F) .
\end{aligned}
\end{equation}
Since $H''$ and $\tilde H$ are contained in the respective standard Iwahori subgroups, we
deduce from \eqref{eq:3.10} that there is a natural bijection between the chambers of
$\mc{BT}(\mc G,F)$ and those of $\mc{BT}(\tilde Q)$. The standard apartments in these
buildings are the same, because they depend only on the root subgroups $U_\alpha$ and on the
valuation. We conclude that $\mc{BT}(\tilde Q)$ can be identified canonically with
$\mc{BT}(\mc G,F)$.

From the above argument we also see that we obtain the same building $\mc{BT}(\mc G,F)$ if 
we replace $Z_G (S_\es)$ by any subgroup containing the intersection of $Z_G (S_\es)$ with
the subgroup of $G$ generated by the root subgroups. In particular $\mc{BT}(\mc G',F)$ is
just the same as $\mc{BT}(\mc G,F)$.

Recall that the building $\mc B (\mc G',F)$ is constructed in \eqref{eq:1.23} and Proposition
\ref{prop:2.3}, as the direct product of Bruhat--Tits buildings for the semisimple groups
$\mc G_i (F_i)$. Hence there are canonical homeomorphisms
\begin{equation}\label{eq:3.11}
\! \prod\nolimits_i \mc{BT}(\mc G_i ,F_i) = \mc B (\mc G',F) = \mc{BT}(\mc G',F) = 
\mc{BT}(\mc G,F) = \mc{BT}(\tilde Q) = \mc{BT}(\mc Q,F) . \hspace{-3mm}
\end{equation}
Recall from \eqref{eq:1.23} that 
$\mc B (\mc G,F) = \mc B(\mc G',F) \times X_* (\mc Z) \otimes_\Z \R$, where $\mc Z$ denotes
a maximal $F$-split torus in $Z(\mc G)$. From the proof of Theorem \ref{thm:3.2}.b, in particular 
\eqref{eq:3.3}, we see that the maximal $F$-split torus in $Z(\mc Q)$ is mapped isomorphically
to $\mc Z$ by $\pi : \mc Q \to \mc G$. Therefore we may identify this $\mc Z$ with the $\mc Z$
in the statement of the proposition. Together with \eqref{eq:3.11} this gives the desired
homeomorphism.\\
(d) The $Q$-action on $X_* (\mc Z) \otimes_\Z \R$ as subset of $\mc B (\mc G,F)$ goes via
$\mc Q \to \mc G \to \mc C / \phi (\mc T')$ and then as in \eqref{eq:1.4}.
As noted after \eqref{eq:2.6}, the image of $\mc G'$ in $\mc G$ is $\mc D (\mc G)$. Hence
\[
\mc C / \phi (\mc T') \cong \mc G / \mc D (\mc G) \cong \mc Q / \mc D (\mc Q) \mc R_{u,F}(\mc Q) ,
\]
showing that the action of $Q$ on $X_* (\mc Z) \otimes_\Z \R$ as defined just before the current 
proposition agrees with the action via $G$.

We already observed that the action of $Q$ on $\mc{BT}(\mc Q,F)$ can be identified with the
action of $\tilde Q = Q / \mc R_{u,F}(\mc Q)(F)$ on $\mc{BT}(\mc G,F)$. In \eqref{eq:3.11}
the action of $\tilde Q$ on $\mc{BT}(\mc G',F)$ is defined as the canonical extension of the
$\tilde Q'$-action (coming from the valuated root datum for that group) to $\tilde Q$. 
By \cite[Th\'eor\`eme 6.5.ii]{BrTi2} the embedding $\tilde Q' \to \tilde Q$ has the correct
type for such an extension. 

Similarly the $G$-action on $\mc{BT}(\mc G',F)$ is defined by first considering the action of
$G''$ derived from the valuated root datum, and then extending to $G$ via the embedding
$G'' \to G$. The two actions of $\tilde Q'$ on these buildings can be identified via 
\eqref{eq:3.11}, and they are extended in the same way to $\tilde Q$. Therefore the actions of
$\tilde Q$ on $\mc{BT}(\mc Q,F) \times X_* (\mc Z) \otimes_\Z \R$ and on $\mc B (\mc G,F)$
can be identified as well.\\
(e) The claim about the properties of the $Q$-action follows immediately from Theorem 
\ref{thm:3.2} and the compactness of $\mc R_{u,F}(\mc Q)(F)$. Bruhat and Tits showed that\\
$\mc{BT}(\mc Q,F) \times X_* (\mc Z) \otimes_\Z \R$ is a CAT(0)-space \cite[3.2.1]{BrTi1}, 
that it has unique geodesics \cite[2.5.13]{BrTi1} and that every compact subgroup of $Q$ 
fixes a point of this affine building \cite[Proposition 3.2.4]{BrTi1}. 
By \cite[Proposition 1.8]{BCH} these properties guarantee that 
$\mc{BT}(\mc Q,F) \times X_* (\mc Z) \otimes_\Z \R$ is a universal space for proper $Q$-actions.
\end{proof}

Theorem \ref{thm:3.4}.c justifies the definition
\begin{equation}\label{eq:3.17}
\mc B (\mc Q,F) := \mc{BT}(\mc Q,F) \times X_* (\mc Z) \otimes_\Z \R , 
\end{equation}
with the $Q$-action from Theorem \ref{thm:3.4}.d. Notice that with these conventions
\[
X_* (\mc Z) \otimes_\Z \R = \mc B (\mc Q / \mc D (\mc Q) \mc R_{u,F}(\mc Q),F) =
\mc B (\mc Q / \mc D (\mc Q),F).
\]
An apartment in the affine building $\mc B (\mc Q,F)$ is 
a subset of the form $\mh A \times X_* (\mc Z) \otimes_\Z \R$, where $\mh A$ is an apartment in
$\mc{BT}(\mc Q,F)$. By Proposition \ref{prop:2.3}.c, Theorem \ref{thm:3.4}.d and \eqref{eq:3.3}
every apartment of $\mc B (\mc Q,F)$ is naturally associated to a maximal $F$-split
torus in $\mc Q$, and conversely. As all maximal $F$-split tori in $\mc Q$ are $Q$-conjugate,
all apartments of $\mc B (\mc G,F)$ are $Q$-associate. We note also that $N_Q (S_\es)$ 
stabilizes the apartment associated to $\mc S_\es$.

\subsection{Compact open subgroups} \ 

The existence of maximal compact (open) subgroups of $\mc Q (F)$ was already established in
\cite{Loi}. We can use affine building of $\mc Q (F)$ to construct many compact open subgroups 
with very specific nice properties.

\begin{prop}\label{prop:3.18}
Let $x = (x_s,x_{\mc Z}) \in \mc{BT}(\mc Q,F) \times X_* (\mc Z) \otimes_\Z \R$.
\enuma{
\item $Q_x$ is a compact open subgroup of $Q$.
\item The stabilizer $Q_{x_s}$ of $x_s \in \mc{BT}(\mc Q,F)$ equals $Z_Q (S_\es)_{x_s} Q_x$.
It contains $Z(Q) Q_x$ as a cocompact subgroup.
\item Suppose that $x_s$ is a vertex of $\mc{BT}(\mc Q,F)$, so that $x$ is a vertex of 
$\mc B (\mc Q,F)$. Then $Q_x$ is a maximal compact subgroup of $Q$ and $Q_{x_s}$ is a maximal
compact modulo center subgroup of $Q$.
}
\end{prop}
\begin{proof}
(a) The group $Q_x$ is compact because the action of $Q$ on $\mc B (\mc Q,F)$ is proper
(Theorem \ref{thm:3.4}.e). It is open in $Q$ because this action factors via $\pi : Q \to G$ 
and $G_x$ is open in $G$ -- see \eqref{eq:2.60}.\\
(b) By definition (see Theorem \ref{thm:3.4}.d) 
\[
Q_x = Q_{x_s} \cap Q_{x_{\mc Z}} ,\text{ so } Z_Q (S_\es)_{x_s} Q_{x} \subset Q_{x_s} .
\]
The $Q$-action on $X_* (\mc Z) \otimes_\Z \R$ factors via 
$(\mc Q / \mc D (\mc Q))(F)$, so for any $q \in Q_{x_s}$ we can find $q' \in Z_Q (S_\es)$ with 
$q' q x_{\mc Z} = x_{\mc Z}$. In view of Theorem \ref{thm:3.4}.d and the decomposition 
\eqref{eq:1.4}, we can even achieve this with $q' \in Z_Q (S_\es)_{x_s}$.
Then $q = (q')^{-1} (q' q)$ with $q' q \in Q_{x_s} \cap Q_{x_{\mc Z}} = Q_x$.

The center $Z(Q)$ acts trivially on $\mc{BT}(\mc Q,F)$ because the $Q$-action is defined in
terms of conjugation of Iwahori (or parahoric) subgroups of $Q$ -- see the proof of Theorem
\ref{thm:3.4}.b. Thus
\[
Z(Q) Q_x \subset Q_{x_s} = Z_Q (S_\es)_{x_s} Q_x .
\]
By Lemma \ref{lem:3.16}.a $(S_\es )_{x_s}$ is cocompact in $Z_Q (S_\es)_{x_s}$ and from
\eqref{eq:1.4} we see that $Z(Q) \cap S_\es$ is cocompact in $(S_\es )_{x_s}$. Hence
$Q_{x_s} / Z(Q) Q_x \cong Z_Q (S_\es)_{x_s} / Z(Q)$ is compact.\\
(c) By the Bruhat--Tits fixed point theorem \cite[Proposition 3.2.4]{BrTi1} every compact
subgroup of $Q$ fixes a point of $\mc B (\mc Q,F)$. From part (a) we know that the 
$Q$-stabilizer of any point is compact. The construction of isotropy groups for affine
buildings associated to double Tits systems \cite[\S 2.1.2]{BrTi1} entails that the stabilizers
of vertices are maximal among the stabilizers of points. Therefore $Q_x$ is maximal compact.

As observed above, $Z(Q)$ acts trivially on $\mc{BT}(\mc Q,F)$. Hence we may replace $Q$ by
$Q / Z(Q)$ when it comes to the isotropy group of $x_s$. Using part (b) as input, the same
argument as for $Q_x$ applies.
\end{proof}

We can describe maximal compact (modulo center) subgroups of $Q$ more precisely, in terms
of its building. Their characterization enables one to lift automorphisms of $Q$ to 
isometries of $\mc{BT}(\mc Q,F)$.

\begin{prop}\label{prop:3.19}
\enuma{
\item Every maximal compact modulo center subgroup of $Q$ is of the form $Q_{y_s}$ 
for a unique $y_s \in \mc{BT}(\mc Q,F)$, which is the barycenter of a polysimplex.
\item Every maximal compact subgroup of $Q$ is of the form $Q_{(y_s,0)}$ 
for a unique $y_s \in \mc{BT}(\mc Q,F)$, which is the barycenter of a polysimplex.
\item Every group which acts on $Q$ by automorphisms of topological groups has a natural
isometric action on $\mc{BT}(\mc Q,F)$.
}
\end{prop}
\begin{proof}
(a) Consider a maximal compact modulo center subgroup of $Q$. By the Bruhat--Tits fixed point
theorem and Proposition \ref{prop:3.18}.b, it equals the isotropy group $G_{y_s}$, for a point 
$y_s \in \mc{BT}(\mc Q,F)$. Let $\mf f$ be the polysimplex of $\mc{BT}(\mc Q,F)$ containing $y_s$.

Suppose that $Q_{y_s}$ also fixes another point $y'_s$. By the uniqueness of 
geodesics \cite[Remarque 3.2.5]{BrTi1} $Q_{y_s}$ fixes the line segment $[y_s,y'_s]$ pointwise.
Then it would also fix the intersection of $\bar{\mf f}$ with a geodesic containing $[y_s,y'_s]$.
In particular $Q_{y_s}$ would fix a point $x_s \in \bar{\mf f} \setminus \mf f$. By 
\cite[\S 2.1]{BrTi1}, $Q_{x_s}$ contains the parahoric subgroup $P_{\mf f'}$ associated to
the facet $\mf f'$ of $\mc{BT}(\mc Q,F)$ in which $x_s$ lies. Since $\mf f'$ is a boundary
facet of $\mf f$, $P_{\mf f'}$ is strictly larger than $P_{\mf f}$. The construction of
the $Q$-action on its building \cite[Th\'eor\`eme 6.5]{BrTi1} entails that there exists 
a subgroup $Q' \subset Q$ such that $Q' \cap Q_{y_s} = P_{\mf f}$ and $Q' \cap Q_{x_s} = 
P_{\mf f'}$. Thus $Q_{x_s} \supsetneq Q_{y_s}$, contradicting the maximality (modulo center)
of $Q_{y_s}$.

Therefore $y_s$ is the unique fixed point of $Q_{y_s}$. Since $Q_{y_s}$ stabilizes $\mf f$ 
and acts on it by polysimplicial automorphisms, this is only possible if $y_s$ is the
barycenter of $\mf f$.\\
(b) Consider any maximal compact subgroup of $G$. By the Bruhat--Tits fixed point theorem
it is of the form $Q_y$ for some $y \in \mc B (\mc Q,F)$. Since $Q$ acts by translations on
$X_* (\mc Z) \otimes_\Z \R$, $Q_y$ fixes $X_* (\mc Z) \otimes_\Z \R$ pointwise. Hence
$Q_y = Q_{(y_s,0)}$ for some $y_s \in \mc{BT} (\mc Q,F)$.

If $Q_y$ would fix another $y'_s \in \mc{BT}(\mc Q,F)$, then we proceed as in the proof of 
part (a) to find an $x_s$ fixed by $Q_y$. We obtain $Q' \cap Q_y = P_{\mf f} \cap {}^0 Q$ 
and $Q' \cap Q_{(x_s,0)} = P_{\mf f'} \cap {}^0 Q$. The difference between $P_{\mf f}$ and
$P_{\mf f'}$ can already be detected in some root subgroup (see \cite[\S 3.1]{Tit} and
\cite[Proposition 7.4.4]{BrTi1}) and those are contained in ${}^0 Q$ (Lemma \ref{lem:3.15}).
Therefore $Q' \cap Q_{(x_s,0)} \supsetneq Q' \cap Q_y$, which would contradict the maximality
of $G_y$. This remainder of the argument is exactly as in part (a).\\
(c) Let $X_M \subset \mc{BT}(\mc Q,F)$ be the set of points whose isotropy group is a maximal 
compact modulo center subgroup of $Q$. Let $\Gamma$ be a group acting on $Q$ by automorphisms 
of topological groups. For $\gamma \in \Gamma$ and $y_s \in X_M$, $\gamma (Q_{y_s})$ is another 
maximal compact modulo center subgroup of $Q$. By part (a) it is of the form $Q_{z_s}$, for a 
unique point $z_s \in \mc{BT}(\mc Q,F)$. The definition $\tilde \gamma (y_s) := z_s$ provides a
permutation of $X_M$.

Every $y_s \in X_M$ determines a set of ``neighbors"
\[
Nb (y_s) := \big\{ x_s \in X_M \setminus \{y_s\} : [x_s,y_s] \cap X_M = \{x_s,y_s\} \big\} 
\]  
In view of part (a), $Nb (y_s)$ is finite and equals the set of $x_s \in X_M \setminus \{y_s\}$
such that $Q_{x_s}$ and $Q_{y_s}$ are the only maximal compact modulo center subgroups of $Q$
containing $Q_{x_s} \cap Q_{y_s}$. As $Q_{\tilde \gamma (y_s)} = \gamma (Q_{y_s})$, we have
$\tilde \gamma (Nb (y_s)) = Nb (\tilde \gamma (y_s))$. Using part (a), this implies that 
$\tilde \gamma$ is an isometry of $X_M$, and that it maps the collection of vertices of any
single polysimplex to the collection of vertices of a single (other) polysimplex. Thus 
$\tilde \gamma$ extends uniquely to an isometry of $\mc{BT}(\mc Q,F)$ which is affine 
on every polysimplex. 

For $\gamma' \in \Gamma$, the uniqueness of $\tilde \gamma$ implies that 
$\widetilde{\gamma \gamma'} = \tilde \gamma \tilde \gamma'$. Hence $\gamma \mapsto \gamma'$ 
defines a group homomorphism $\Gamma \to \mr{Isom}(\mc{BT}(\mc Q,F))$.
\end{proof}

Now we can formulate and prove a more precise version of Proposition \ref{prop:2.1}.
The main advantage over Proposition \ref{prop:2.1} is that now we can find arbitrarily small
compact open subgroups that are normal in the isotropy group of a point of $\mc B (\mc Q,F)$.
This property will be used in Theorem \ref{thm:1.1}, via \cite{Ber}.

\begin{thm}\label{thm:3.10}
Let $\mc S$ be a maximal $F$-split torus of $\mc Q$ and let $\mh A_{\mc S} = X_* (\mc S) 
\otimes_\Z \R$ be the associated apartment of $\mc B (\mc Q,F)$. Let $x \in \mh A_{\mc S}$ 
and let $x_s$ be its image in $\mc{BT}(\mc Q,F)$ via \eqref{eq:3.17}.
There exists a sequence of compact open subgroups $(K_n )_{n=1}^\infty$ of $Q$ such that:
\enuma{
\item $K_n \supset K_{n+1}$ and the $K_n$ form a neighborhood basis of 1 in $Q$.
\item $K_n \subset Q_x \subset Q_{x_s}$ and $K_n$ is normal in $Q_{x_s}$.
\item $K_n$ satisfies the Iwahori decomposition: 
Suppose that $\mc P_{\mc Q}(\lambda)$ is a pseudo-parabolic $F$-subgroup of $\mc Q$ 
containing $Z_{\mc Q}(\mc S)$. Then the multiplication map
\[
\big( K_n \cap U_{\mc Q}(\lambda)(F) \big) \times \big( K_n \cap Z_{\mc Q}(\lambda)(F) \big) 
\times \big( K_n \cap U_{\mc Q}(\lambda^{-1})(F) \big) \to K_n
\]
is a homeomorphism.
\item Suppose that $a \in Z_Q (S)$ is positive with respect to $\mc P_{\mc Q}(\lambda)$:
\begin{align*}
& \inp{\alpha}{a \cdot x - x} \geq 0 \quad \forall \alpha \in \Phi (\mc P_{\mc Q}(\lambda),\mc S) , \\ 
& \inp{\alpha}{a \cdot x - x} = 0 \quad \forall \alpha \in \Phi (\mc Z_{\mc Q}(\lambda),\mc S) .
\end{align*}
Then $a$ normalizes $K_n \cap Z_{\mc Q}(\lambda)(F)$ and
\begin{align*}
& a \big( K_n \cap U_{\mc Q}(\lambda)(F) \big) a^{-1} \supset 
\big( K_n \cap U_{\mc Q}(\lambda)(F) \big) , \\
& a \big( K_n \cap U_{\mc Q}(\lambda^{-1})(F) \big) 
a^{-1} \subset \big( K_n \cap U_{\mc Q}(\lambda^{-1})(F) \big) .
\end{align*}
\item Suppose that $a \in Z_Q (S)$ is positive with respect to $\mc P_{\mc Q}(\lambda)$ and
\[
\inp{\alpha}{a \cdot x - x} > 0 \quad \forall \alpha \in \Phi (\mc P_{\mc Q}(\lambda),\mc S) .  
\]
Then $\{  a^m \big( K_n \cap U_{\mc Q}(\lambda)(F) \big) a^{-m} : m \in \N \}$
forms a neighborhood basis of 1 in $U_{\mc Q}(\lambda)(F)$ and
\[
\bigcup\nolimits_{m \in \N} a^{-m} \big( K_n \cap U_{\mc Q}(\lambda)(F) \big) a^m = 
U_{\mc Q}(\lambda)(F) .
\]
Similarly $\{  a^{-m} \big( K_n \cap U_{\mc Q}(\lambda^{-1})(F) \big) 
a^m : m \in \N \}$ forms a neighborhood basis of 1 in $U_{\mc Q}(\lambda^{-1})(F)$ and
\[
\bigcup\nolimits_{m \in \N} a^m \big( K_n \cap U_{\mc Q}(\lambda^{-1})(F) \big) a^{-m} =  
U_{\mc Q}(\lambda^{-1})(F) . 
\]
}
\end{thm}
\begin{proof}
Recall that in Proposition \ref{prop:3.17} we exhibited normal subgroups of 
$N_G (S_\es)$ to prolong the valuation for the generating root datum of $G$ from Theorem 
\ref{thm:3.3}. Given this, $n \in \Z_{\geq 0}$ and a facet $\mf f \subset \mc{BT}(\mc G,F) = 
\mc{BT}(\mc Q,F)$ containing $\tilde x$, Schneider and Stuhler \cite[p.114]{ScSt} construct 
a subgroup $U_{\mf f}^{(n)} \subset G$. It is checked in \cite[\S I.2]{ScSt} that it has  
the properties (a), (b) and (c) required in the theorem, for $\mc G, F$ and $\tilde x$. 
Let us note here that
\begin{equation}\label{eq:3.56}
Z_G (S_\es) \cap U_{\mf f}^{(n)} = U_{0,n+} := \bigcap\nolimits_{k \in \R_{>n}} U_{0,k} .
\end{equation}
For an arbitrary quasi-reductive group $Q$ we could use the same construction, that would 
produce a family of subgroups $\pi^{-1} \big( U_{\mf f}^{(n)}) \subset Q$ which almost satisfies 
the theorem. The only problem would be part (a): the intersection of those groups would be 
$\mc R_{u,F}(\mc Q)(F)$, which is nontrivial if $\mc Q$ is not pseudo-reductive. Below we will
use the properties (b) and (c) of the groups $\pi^{-1} \big( U_{\mf f}^{(n)})$.
By adjusting these groups, we construct the $K_n$.

In view of the conjugacy of maximal $F$-split tori in $\mc Q$ \cite[Theorem C.2.3]{CGP},
we may assume that $\mc S = \mc S_\es ,\; \mh A_{\mc S} = \mh A_\es$ and $\mf f \subset \mh A_\es$.
With Lemma \ref{lem:3.16}.c we can find a decreasing sequence of compact open subgroups 
$Z_r \; (r \in \R)$ of $Z_Q (S_\es)$, such that every $Z_r$ is normal in $N_Q (S_\es)$ and 
$\bigcap_{r \in \R} Z_r = \{1\}$.

For every $\alpha \in \Phi (\mc Q,\mc S_\es)$ and every $n \in \N$ we have a set of commutators 
\begin{equation}\label{eq:3.55}
\big[ U_\alpha \cap Q_x, U_{-\alpha} \cap \pi^{-1} \big( U_{\mf f}^{(n)} \big)\big] \subset 
\pi^{-1} \big( U_{\mf f}^{(n)} \big) \subset Q .
\end{equation}
Define the positive roots $\Phi (\mc Q,\mc S_\es)^+$ as in \eqref{eq:3.14}. In view of the Iwahori 
decomposition of $\pi^{-1} \big( U_{\mf f}^{(n)} \big)$, for sufficiently large $n$ we can find 
$k (n+) \in \R$ such that 
\[
\Big( \prod_{\alpha \in - \Phi (\mc Q,\mc S_\es)^+_\red} U_\alpha \cap \pi^{-1} \big( U_{\mf f}^{(n)} 
\big) \Big) Z_{k(n+)} \Big( \prod_{\alpha \in \Phi (\mc Q,\mc S_\es)^+_\red} U_\alpha \cap 
\pi^{-1} \big( U_{\mf f}^{(n)} \big) \Big) 
\]
contains the union, over $\alpha \in \Phi (\mc Q,\mc S_\es)$, of the sets \eqref{eq:3.55}.  
In the notation of \eqref{eq:3.56}, this says precisely that $Z_{k(n+)}$ is contains the group 
$H_{[n+]}$ from Definition \ref{defn:0}. Since $\pi^{-1}(U_{0,n+})$ comes from a
prolonged valuation, it also contains $H_{[n+]}$. 

For each $\alpha$ the set \eqref{eq:3.55} decreases to $\{1\}$ when $n \to \infty$, so we can 
achieve that $k((n+1)+) > k(n+)$ and $\lim_{n \to \infty} k(n+) = \infty$. For sufficiently 
large $n$ we consider the group
\begin{equation}\label{eq:3.57}
\tilde U_{0,n+} := \pi^{-1}(U_{0,n+}) \cap Z_{k(n+)} .
\end{equation}
Both terms on the right hand side are open and normal in $N_Q (S_\es)$ (the former by Proposition 
\ref{prop:3.17}) and $Z_{k(n+)}$ is compact, $\tilde U_{0,n+}$ is compact open and normal in 
$N_Q (S_\es)$. The $\tilde U_{0,n+}$ have intersection $\{1\}$ and by Proposition \ref{prop:3.17} 
they fulfill (ii) and (iii) of Definition \ref{defn:0}. (However, it is not clear whether they 
satisfy (iv), we have too little control over $[\tilde U_{0,n+},\tilde U_{0,m+}]$.) Slightly 
varying on the construction of $U_{\mf f}^{(n)}$ in \cite{ScSt}, we define
\[
K_n := \Big( \prod_{\alpha \in - \Phi (\mc Q,\mc S_\es)^+_\red} U_\alpha \cap \pi^{-1} 
\big( U_{\mf f}^{(n)} \big) \Big) \tilde U_{0,n+} \Big( \prod_{\alpha \in \Phi 
(\mc Q,\mc S_\es)^+_\red} U_\alpha \cap \pi^{-1} \big( U_{\mf f}^{(n)} \big) \Big) .
\]
The arguments in \cite[\S I.2]{ScSt} show that $K_n$ is a normal subgroup of $Q_x$ and that
it admits an Iwahori decomposition. Recall from Proposition \ref{prop:3.17}.b that 
$Q_{x_s} = Z_Q (S_\es)_{x_s} Q_x$.
Since $\pi : Q \to G$ becomes an isomorphism on root subgroups \eqref{eq:3.4}, the properties
of $U_{\mf f}^{(n)}$ entail that $Z_Q (S_\es)_{\tilde x}$ normalizes 
\[
\prod\nolimits_{\alpha \in \Phi (\mc Q,\mc S_\es)^+_\red} U_\alpha \cap 
\pi^{-1} \big( U_{\mf f}^{(n)} \big),
\]
and similarly with negative roots. As $\tilde U_{0,n+}$ is a normal subgroup of $N_Q (S_\es)$, 
$Z_Q (S_\es)_{\tilde x}$ also normalizes $\tilde U_{0,n+}$. Now we see from the Iwahori 
decomposition of $K_n$ that $Z_Q (S_\es)_{x_s}$ and $Q_{x_s}$ normalize it.

(d) In view of the Iwahori decomposition of $K_n$, we only have to check the corresponding
statements for each $K_n \cap U_\alpha$ with $\alpha \in \Phi (\mc Q,\mc S_\es)_\red \cup \{0\}$.
As explained after \eqref{eq:3.57}, $a \in Z_Q (S_\es)$ normalizes $K_n \cap Z_Q (S_\es) = 
\tilde U_{0,n+}$. 

For $\alpha \in \Phi (\mc G,\mc S_\es)$, $\pi : \mc Q \to \mc G$ restricts
to an isomorphism on $\mc U_\alpha$ and $a$ normalizes $\mc U_\alpha$. Therefore it suffices
to consider the effect of conjugation by $\pi (a) \in Z_G (S_\es)$ on $U_\alpha \cap \pi (K_n)  
= U_\alpha \cap U_{\mf f}^{(n)}$. By \cite[p. 114]{ScSt}
\[
\pi (a) U_{\mf f}^{(n)} \pi (a)^{-1} = U_{\pi (a) \mf f}^{(n)} 
\]
Furthermore \cite[Corollary I.2.8]{ScSt} implies that 
\begin{equation}\label{eq:2.27}
U_\alpha \cap U_{\mf f}^{(n)} \subset U_\alpha \cap U_{\pi (a) \mf f}^{(n)} 
\; \Longleftrightarrow \; \inp{\alpha}{x} \leq \inp{\alpha}{\pi (a) x} .
\end{equation}
The action of $Z_Q (S_\es)$ on $\mh A_\es$ factors through $\pi$, so the right hand side
is equivalent to $\inp{\alpha}{a \cdot x - x} \geq 0$.\\
(e) The equivalence \eqref{eq:2.27} can be expressed in terms of the compact open subgroups
$U_{\alpha,r} \subset U_\alpha$ defined in \eqref{eq:3.54}. They satisfy
\[
U_{\alpha,r} \subset U_{\alpha,s} \text{ if } r \geq s ,\; \bigcup\nolimits_{r \in \R} 
U_{\alpha,r} = U_\alpha \text{ and } \bigcap\nolimits_{r \in \R} U_{\alpha,r} = \{1\} .
\]
As explained above, we can interpret \eqref{eq:2.27} also in $Q$. Then it implies
\begin{equation}\label{eq:2.28}
\lim_{\inp{\alpha}{a x} \to \infty} U_\alpha \cap U_{a \cdot \mf f}^{(n)} = U_\alpha \quad 
\text{and} \quad \lim_{\inp{\alpha}{a x} \to -\infty} U_\alpha \cap U_{a \cdot \mf f}^{(n)} = \{1\}. 
\end{equation}
Now apply the Iwahori decomposition of $K_n$ to decompose 
$K_n \cap \mc U_{\mc Q}(\lambda^{\pm 1})(F)$ as products over root subgroups,
then \eqref{eq:2.28} becomes the required statement.
\end{proof}

\subsection{Decompositions} \

We recall some results of Bruhat and Tits about $\mc{BT}(\mc Q,F)$ and lift them
to $\mc B (\mc Q,F)$. We use this to decompose $Q$ in various ways, generalizing the
decompositions named after Iwasawa, Cartan and Iwahori--Bruhat.

\begin{thm}\label{thm:3.5}
Let $x$ be any point of the standard apartment $\mh A_0$ of $\mc{BT}(\mc Q,F)$. 
\enuma{
\item $Q = \mc R_{us,F}(\mc P_\es)(F) N_Q (S_\es) Q_x = Q_x N_Q (S_\es) \mc R_{us,F}(\mc P_\es)(F)$.
\item $Q \cdot x \cap \mh A_0 = N_Q (S_\es) \cdot x$.
\item $\mh A_0$ is a fundamental domain for the action of $\mc R_{us,F}(\mc P_\es)(F)$ 
on $\mc{BT}(\mc Q,F)$.
\item $Q_x$ acts transitively on the collection of apartments of $\mc{BT}(\mc Q,F)$ containing $x$.
\item For every $y \in \mc{BT}(\mc Q,F)$ there exists an apartment containing $\{x,y\}$.
}
\end{thm}
\begin{proof}
(a) The first equality is the Iwasawa decomposition for $\mc{BT}(\mc Q, F)$ 
\cite[Proposition 7.3.1.i]{BrTi1}. The second equality follows by considering the inverses
of elements of $Q$.\\
(b) This is a direct consequence of the definition of $\mc{BT}(\mc Q,F)$ in \cite[7.4.2]{BrTi1}
(which is equivalent to the earlier definition).\\
(c) By \cite[Proposition 7.3.1.ii]{BrTi1} the map
\[
\begin{array}{cccl}
\mc R_{us,F}(\mc P_\es)(F) \backslash Q / Q_x & \to & \mh A_0 \\
\mc R_{us,F}(\mc P_\es)(F) q Q_x & \mapsto & q x & \text{ for } q \in N_Q (S_\es)  
\end{array}
\]
is well-defined and injective. Hence $\mc R_{us,F}(\mc P_\es)(F) \cdot x$ contains exactly one
element of $\mh A_0$, namely $x$ itself. 

Let $y \in \mc{BT}(\mc Q,F)$. By definition \cite[7.4.2]{BrTi1} $Q \cdot \mh A_0 = 
\mc{BT}(\mc Q,F)$, so $y \in Q \cdot x$ for some $x \in \mh A_0$. With part (a) we can find 
$u \in \mc R_{us,F}(\mc P_\es)(F), n \in N_Q (S_\es)$ and $k \in Q_x$ such that 
\[
y = u n k x = u n x \; \in \; u N_Q (S_\es) \cdot x \subset u \mh A_0 
\subset \mc R_{us,F}(\mc P_\es)(F) \mh A_0.
\]
Thus $\mc R_{us,F}(\mc P_\es)(F) \mh A_0$ contains every element of $\mc{BT}(\mc Q,F)$.\\
(d) Let $\mh A_x$ be an apartment of $\mc{BT}(\mc Q,F)$ containing $x$. By the definition of 
apartments \cite[2.2.2]{BrTi1}, there exists a $q \in Q$ with $\mh A_x = q \mh A_0$. Using 
part (a) we write $q = k n u$ with $k \in Q_x , n \in N_Q (S_\es)$ and 
$u \in \mc R_{us,F}(\mc P_\es)(F)$. Then $x = k^{-1} x \in k^{-1} \mh A_x$ and we can rewrite
this apartment as
\[
k^{-1} \mh A_x = k^{-1} q \mh A_0 = n u \mh A_0 = (n u n^{-1}) n \mh A_0 = (n u n^{-1}) \mh A_0 .
\]
Notice that $n \mc P_\es n^{-1}$ is another minimal 
pseudo-parabolic $F$-subgroup of $Q$ containing $Z_Q (S_\es)$. As $x \in \mh A_0 \cap 
(n u n^{-1}) \mh A_0$ and $n u n^{-1} \in \mc R_{us,F}(n \mc P_\es n^{-1})(F)$, part (c) implies 
that $(n u n^{-1}) x = x$. Thus $\mh A_x  = k (n u n^{-1}) \mh A_0$ where 
$k (n u n^{-1}) \in Q_x$.\\
(e) See \cite[Proposition 2.3.1]{BrTi1}.
\end{proof}

For notational convenience we formulated Theorem \ref{thm:3.5} only in the ``standard" situation.
After \eqref{eq:3.17} we remarked that all apartments of $\mc{BT}(\mc Q,F)$ are $Q$-associate.
Hence the above also holds for other apartments.
Since the collection of apartments of $\mc B (\mc Q,F)$ is canonically in bijection with that for
$\mc{BT}(\mc Q,F)$ via \eqref{eq:3.17}, the transfer of Theorem \ref{thm:3.5} to $\mc B (\mc Q,F)$ 
is straightforward. We refrain from writing it down here, since the statements are exactly 
the same as for $\mc{BT}(\mc Q,F)$. 

The standard apartment of $\mc B (\mc Q,F)$ is 
\[
\mh A_\es := \mh A_0  \times X_* (\mc Z) \otimes_\Z \R .
\]
The unique maximal compact subgroup $Z_Q (S_\es)_\cpt$  of $Z_Q (S_\es)$ (from Lemma 
\ref{lem:3.16}.b) is the isotropy group of any point of $\mh A_\es$. The quotient 
$Z_Q (S_\es) / Z_Q (S_\es)_\cpt$ acts cocompactly on $\mh A_\es$ by translations, so it is a
finitely generated free abelian group. The group $N_Q (S_\es) / Z_Q (S_\es)_\cpt$ is an 
extension of the Weyl group $W(\mc Q,\mc S_\es)$ by the lattice $Z_Q (S_\es) / Z_Q (S_\es)_\cpt$,
it can be regarded as an (extended) affine Weyl group for $(\mc Q,\mc S_\es)$.

Recall that a maximal compact subgroup $K$ of a reductive group $H$ over a local field is
called good if the Iwasawa decomposition $H = PK$ holds for any parabolic subgroup $P$ of $H$. 
Equivalently \cite[Proposition 4.4.2]{BrTi1}, it contains representatives for all elements of 
the Weyl group of $H$ (with respect to a maximal split torus). 

By the Bruhat--Tits fixed point theorem \cite[Corollaire 3.3.2]{BrTi1}, every maximal compact 
subgroup of $Q$ is the full stabilizer of a point of $\mc B (\mc Q,F)$. Recall also that 
every $Q$-orbit in $\mc B (\mc Q,F)$ intersects every apartment.  
With this in mind, we define a maximal compact subgroup of $Q$ to be \emph{good} if it is 
conjugate to a group $Q_{x'}$ with 
\begin{equation}\label{eq:3.58}
x' \in \mh A_\es \quad \text{and} \quad N_Q (S_\es)_{x'} Z_Q (S_\es) = N_Q (S_\es). 
\end{equation}
Every quasi-reductive $F$-group $Q$ has maximal compact subgroups, for instance the 
isotropy groups of special vertices \cite[\S 4.4.6]{BrTi1}. 

\begin{thm}\label{thm:3.6}
Let $K$ be any good maximal compact subgroup of $Q$ and let $\mc P$ be any pseudo-parabolic 
$F$-subgroup of $\mc Q$. Then $Q$ admits the Iwasawa decomposition $Q = K \mc P (F) = \mc P (F) K$.
\end{thm}
\begin{proof}
Conjugating $K$ and $\mc P$ by an element of $Q$, we may assume that $K = Q_{x'}$ for some
$x' \in \mh A_\es$. We may and will also assume that $\mc P$ is a minimal pseudo-parabolic 
$F$-subgroup. Let $\mc S_1$ be a maximal $F$-split torus in $\mc P$ and let $\mh A_1$ be the 
associated apartment of $\mc B (\mc Q,F)$, that is, the apartment of $\mc{BT}(\mc Q,F)$ 
associated to $Z_{\mc Q}(\mc S_1)(F)$ times $X_* (\mc Z) \otimes_\Z \R$. 

We abbreviate $U := \mc R_{us,F}(\mc P)(F)$. 
With Theorem \ref{thm:3.5}.c (for $\mh A_1$) we can find a $u_1 \in U$ such that 
$x' \in u_1 \mh A_1$. Then $\mc S_2 := u_1 \mc S_1 u_1^{-1}$ is another maximal $F$-split 
torus in $\mc P$, and $\mh A_2 = u_1 \mh A_1$. By \eqref{eq:3.12} 
$\mc P (F) = Z_Q (S_2) \ltimes U$.

Now we consider any $q \in Q$. By Theorem \ref{thm:3.5}.c (for $\mh A_2$) we can find 
$u_2 \in U$ such that $u_2 q x' \in \mh A_2$. Since $K = Q_{x'}$ is good,
$N_Q (S_2) \subset Z_Q (S_2) K$. With Theorem \ref{thm:3.5}.b for $x' \in \mh A_2$ we get
\[
Q \cdot x' \cap \mh A_2 = N_Q (S_2) \cdot x' \cap \mh A_2 = Z_Q (S_2) K \cdot x' \cap \mh A_2
= Z_Q (S_2) \cdot x'.  
\]
Pick $z \in Z_Q (S_2)$ such that $u_2 q x' = z^{-1} x'$. Then $k := z u_2 q$ lies in $K$. 
We conclude that 
\[
q = u_2^{-1} z^{-1} k \in U Z_Q (S_2) K = \mc P (F) K . \qedhere
\]
\end{proof}

Notice that the algebraic variety $\mc Q / \mc P$ is naturally isomorphic to 
$\mc G / \pi (\mc P)$, because $\mc R_{u,F}(\mc Q) \subset \mc P$. With the 
surjectivity of $\mc Q (F) \to (\mc Q / \mc P)(F)$ \cite[Lemma C.2.1]{CGP} and Theorem 
\ref{thm:3.6}, the proof of Corollary \ref{cor:2.8} can be applied. It shows that
\begin{equation}\label{eq:3.25}
(\mc Q / \mc P)(F) \cong \mc Q (F) / \mc P (F) \text{ is compact.}
\end{equation}
From Theorem \ref{thm:1.30} one sees that $\mc Q / \mc P$ is the variety of
pseudo-parabolic subgroups of $\mc Q$ conjugate to $\mc P$.

\begin{thm}\label{thm:3.7}
Let $x$ and $y$ be any points of $\mh A_\es$. 
\enuma{
\item $Q = Q_x N_Q (S_\es) Q_y$.
\item The natural map $N_Q (S_\es)_x \backslash N_Q (S_\es) / N_Q (S_\es)_y \to
Q_x \backslash Q / Q_y$ is bijective.
\item Suppose that $x \in \mh A_\es$ is generic, in the sense that $N_Q (S_\es)_x =
Z_Q (S_\es)_\cpt$. (Thus $Q_x$ can be regarded as an Iwahori subgroup of $Q$, namely
the pointwise stabilizer of the unique chamber of $\mh A_\es$ containing $x$.) Choosing
representatives $w$ for $N_Q (S_\es) / Z_Q (S_\es)_\cpt$ in $N_Q (S_\es)$, we obtain the
Iwahori--Bruhat decomposition
\[
Q = \bigsqcup\nolimits_{w \in N_Q (S_\es) / Z_Q (S_\es)_\cpt} Q_x w Q_x .
\]
\item Let $x', y' \in \mh A_\es$ be such that $Q_{x'}$ and $Q_{y'}$ are good maximal 
compact subgroups of $Q$, for example special vertices. Then 
$N_Q (S_\es)_{x'} \backslash N_Q (S_\es) / N_Q (S_\es)_{y'}$ is canonically in bijection 
with the set of $W(\mc Q,\mc S_\es)$-orbits in $Z_Q (S_\es) / Z_Q (S_\es)_\cpt$. 

In case $x' = y'$, it can be represented by a finitely generated semigroup 
$A \subset Z_Q (S_\es)$ and we obtain the Cartan decomposition
\[
Q = \bigsqcup\nolimits_{a \in A} Q_{x'} a Q_{x'} . 
\]
} 
\end{thm}
\begin{proof}
(a) Consider any $q \in Q$, and let $x_s$ and $q y_s = (q y)_s$ be the components of,
respectively, $x$ and $q y$ in $\mc{BT}(\mc Q,F)$. By Theorem \ref{thm:3.5}.e there 
exists an apartment of $\mc{BT}(\mc Q,F)$ containing $\{ x_s , q y_s \}$. Taking the
Cartesian product with $X_* (\mc Z) \otimes_\Z \R$, we obtain an apartment $\mh A_q$ of
$\mc B (\mc Q,F)$ containing $\{ x, q y \}$. By Theorem \ref{thm:3.5}.d there is a 
$k \in Q_x$ such that $\mh A_q = \mh A_\es$. Then 
\[
k^{-1} q y \in \mh A_\es \cap Q \cdot y,
\]
so by Theorem \ref{thm:3.5}.b there exists a $n \in N_Q (S_\es)$ with $k^{-1} q y = n y$.
Then $k' := n^{-1} k^{-1} q$ lies in $Q_y$ and
\[
q = k n k' \in Q_x N_Q (S_\es) Q_y . 
\]
(b) Part (a) says that this natural map is surjective. Suppose that $n, n' \in N_Q (S_\es)$
determine the same double coset $Q_x n Q_y = Q_x n' Q_y$. As $Q_x \subset Q_{x_s}$ and 
$Q_y \subset Q_{y_s}$ also, $Q_{x_s} n Q_{y_s} = Q_{x_s} n' Q_{y_s}$. Then 
\cite[Propostion 4.2.1]{BrTi1} implies that
\[
N_Q (S_\es)_x n N_Q (S_\es)_y = N_Q (S_\es)_x n' N_Q (S_\es)_y .
\]
(c) This is a direct consequence of part (b) and the normality of $Z_Q (S_\es)_\cpt$ in
$N_Q (S_\es)$.\\
(d) Since $Q_{y'}$ is good, the canonical map 
\[
N_Q (S_\es)_{y'} \to W(\mc Q,\mc S_\es) = N_Q (S_\es) / Z_Q (S_\es) 
\]
is surjective, with kernel $Z_Q (S_\es)_\cpt$. Choose a set of representatives $W_{y'}
\subset N_Q (S_\es)_{y'}$ for $W(\mc Q,\mc S_\es)$, then
\begin{equation}\label{eq:3.14}
Z_Q (S_\es) W_{y'} = W_{y'} Z_Q (S_\es) = N_Q (S_\es) ,
\end{equation}
and similarly for $x'$. Consider the injection
\[
N_Q (S_\es) / N_Q (S_\es)_{y'} \to \mh A_\es : n N_Q (S_\es)_{y'} \mapsto n y' . 
\]
It is left $N_Q (S_\es)_{x'}$-equivariant and by \eqref{eq:3.14} its image is 
$Z_Q (S_\es) W_{y'} \cdot y' = Z_Q (S_\es) \cdot y'$. Dividing out by that left action 
and using $N_Q (S_\es)_{x'} / Z_Q (S_\es)_{x'} \cong W_{x'}$, we obtain a bijection
\begin{equation}\label{eq:3.15}
N_Q (S_\es)_{x'} \backslash N_Q (S_\es) / N_Q (S_\es)_{y'} \to 
W_{x'} \backslash Z_Q (S_\es) \cdot y' .
\end{equation}
Similarly the injection 
\[
Z_Q (S_\es) / Z_Q (S_\es )_\cpt \to \mh A_\es : z Z_Q (S_\es)_\cpt \mapsto z \cdot y' 
\]
induces a bijection
\begin{equation}\label{eq:3.16}
W(\mc Q ,\mc S_\es) \backslash Z_Q (S_\es) / Z_Q (S_\es )_\cpt \to 
W_{x'} \backslash Z_W (S_\es) \cdot y' .
\end{equation}
We regard $y'$ as the origin of the affine space $\mh A_\es$. Then 
\[
Z_Q (S_\es) / Z_Q (S_\es )_\cpt \cong Z_Q (S_\es) \cdot y' 
\]
becomes a lattice in $\mh A_\es$, so we can represent it by a finitely generated
free abelian subgroup $A'$ of $Z_Q (S_\es)$. Since $W (\mc Q,\mc S)$ is the Weyl group of a 
root system in $\mh A_\es$, a fundamental domain for the $W_{x'}$-action is formed
by a Weyl chamber with respect to $x'$ as base point. Choose a Weyl chamber $\mh A^+_\es$
which contains $y'$. Then 
\begin{equation}\label{eq:3.60}
Z_Q (S_\es) \cdot y' \cap \mh A_\es^+ \quad \text{represents} \quad 
W_{x'} \backslash Z_Q (S_\es) \cdot y' . 
\end{equation}
Let $A$ be the subset of $A'$ corresponding to $\mh A_\es^+ \cap Z_Q (S_\es ) \cdot y'$.
With \eqref{eq:3.15}, \eqref{eq:3.16} and part (a) we obtain bijections
\[
A \to W(\mc Q ,\mc S_\es) \backslash Z_Q (S_\es) / Z_Q (S_\es )_\cpt \to N_Q (S_\es)_{x'} 
\backslash N_Q (S_\es) / N_Q (S_\es)_{y'} \to Q_{x'} \backslash Q / Q_{y'} . 
\]
The second bijection is canonical if we regard $W(\mc Q,\mc S_\es)$ as  
$N_Q (S_\es)_{x'} / Z_Q (S_\es)_\cpt$.

By \cite[1.3.8 and Proposition 6.2.10]{BrTi1} the root system for $W_{x'}$ is contained in
$Z_Q (S_\es ) \cdot x'$. If $x' = y'$, then this implies that $A \cong Z_Q (S_\es) \cdot y' 
\cap \mh A_\es^+$ is a finitely generated semigroup in $Z_Q (S_\es)$. 
\end{proof}

\newpage

\section{Representation theory} 
\label{sec:rep}

Let $\mc G$ be a quasi-reductive group defined over a non-archimedean local field $F$.
We will investigate the representations of locally compact groups like $G = \mc G (F)$. 
We work in the category Rep$(G)$ of smooth $G$-representations on 
complex vector spaces. Thanks to Proposition \ref{prop:2.1}, this category is quite large.

\subsection{Parabolic induction and restriction} \
\label{par:ind}

We will show that certain results involving parabolic induction and Jacquet restriction hold
for all quasi-reductive $F$-groups. We note that for pseudo-reductive groups these results 
(except Theorem \ref{thm:1.14}) only use Section \ref{sec:pseudo}, and in particular do not 
rely on the Bruhat--Tits theory developed in Paragraph \ref{subsec:BT}.

Suppose that $H \subset G$ is a closed subgroup and that $(\pi,V) \in \mathrm{Rep}(H)$.
The vector space $\mathrm{ind}_H^G (V)$ consists of all locally constant functions
$G \to V$ such that
\[
f(g h^{-1}) = \pi (h) f(g) \; \forall h \in H, g \in G \quad
\text{and supp}(f) / H \text{ is compact.}
\]
This is a $G$-representation for the action by left translations.
Obviously the condition on the support is empty if $G / H$ is compact, which by Corollary
\ref{cor:2.8} and \eqref{eq:3.25} happens if $H$ is a pseudo-parabolic $F$-subgroup of $G$.

Let $\mc P_{\mc G}(\lambda)$ be a pseudo-parabolic $F$-subgroup of $\mc G$, and recall that
\[
\mc P_{\mc G}(\lambda) = \mc U_{\mc G}(\lambda) \rtimes \mc Z_{\mc G}(\lambda) \text{ with }
\mc U_{\mc G}(\lambda) = \mc R_{us,F} (\mc P_{\mc G}(\lambda)) .
\]
In this paragraph we will often abbreviate
\[
\mc P = \mc P_{\mc G}(\lambda) ,\; \overline{\mc P} = \mc P_{\mc G}(\lambda^{-1}) ,
\; M = \mc Z_{\mc G}(\lambda)(F) ,\; U = \mc U_{\mc G}(\lambda)(F)
\text{ and } \overline{U} = \mc U_{\mc G}(\lambda^{-1})(F).
\]
Notice that $M$ and $\overline{U}$ are not canonically determined by $\mc P$, this really
involves $\lambda : GL_1 \to \mc G$. By \cite[Corollary 2.2.5]{CGP} the group
\[
Z_{\mc G}(\lambda) / \mc R_{u,F}(\mc G) = \big( \mc P_{\mc G}(\lambda) / 
\mc R_{u,F}(\mc G) \big) \cap \big( \mc P_{\mc G}(\lambda) / \mc R_{u,F}(\mc G) \big)
\]
inherits the pseudo-reductivity of $\mc G / \mc R_{us,F}(\mc G)$. As $\mc R_{u,F}(\mc G)$ is
$F$-wound, this implies that $Z_{\mc G}(\lambda)$ is quasi-reductive. 
Then Lemma \ref{lem:1.4} and \eqref{eq:1.35} say that $M$ is unimodular. 
Let $\delta_P : P \to \R_{>0}$ be the modular function. Since the unipotent group $U$ is a 
union of compact subgroups, $\delta_P$ is trivial on $U$ and factors through $P / U \cong M$.

We define the normalized parabolic induction of $(\pi,V) \in \mathrm{Rep}(P/U)$ as
\[
i_P^G (\pi) = \mathrm{ind}_P^G (\delta_P^{-1/2} \otimes \pi) \in \mathrm{Rep}(G) .
\]
The space of $U$-coinvariants of $(\rho,W) \in \text{Rep}(G)$ is defined as
\[
W_U = W / \text{span} \{ \rho(u) w - w : w \in W, u \in U \} .
\]
This space carries a natural representation $\rho_U$ of $P / U$. We define the normalized Jacquet
restriction of $(\rho,W)$ as $r_P^G (\rho) = \delta_P^{1/2} \otimes \rho_U$. The functors
\[
i_P^G : \text{Rep}(P/U) \to \text{Rep}(G) \quad \text{and} \quad
r_P^G : \text{Rep}(G) \to \text{Rep}(P/U)
\]
for quasi-reductive $F$-groups have many of the properties that are known in the setting of 
reductive groups. Indeed, thanks to Proposition \ref{prop:2.1}, Lemma \ref{lem:2.5}.b and
Corollary \ref{cor:2.8} (as well as their analogues for quasi-reductive $F$-groups), 
everything in \cite[\S VI.1]{Ren} and \cite[Chapters 1--4]{Cas} remains
valid in our generality. Let us list some of these results:
\begin{align}
& i_P^G \text{ and } r_P^G \text{ are exact functors} , \\
& \label{eq:2.30} i_P^G \text{ is left adjoint to } r_P^G 
\text{ (for both see \cite[Proposition VI.1.2]{Ren})} , \\
& \label{eq:2.13} i_P^G \text{ preserves unitarity
(see \cite[IV.2.3]{Ren} or \cite[Proposition 3.1.4]{Cas})} ,\\
& \label{eq:2.29} r_P^G \text{ preserves finite type \cite[Proposition VI.1.3]{Ren}.}
\end{align}
(Recall that a $G$-representation has finite type if it is generated by a finite subset.)

It is easy to see in examples (e.g. the Steinberg representation of $SL_2 (F)$) that Jacquet
restriction does not preserve unitarity. Nevertheless, there is an analogue.
Let $(\pi,V) \in \text{Rep}(V)$ and let $(\check \pi,\check V)$ be its smooth contragredient
representation. Let $j_U : V \to V_U$ and $j_{\overline U} : \check V \to \check V_{\overline U}$
be the quotient maps.

Let $S_M$ be the maximal $F$-split torus in $Z(M)$. The negative chamber of $S_M$ consists
of all elements $m$ such that $\norm{\alpha (m)}_F < 1$ for all roots $\alpha$ of $U$ with respect
to $S_M$. We say that $m \in S_M$ lies deep in the negative chamber if there exists an
$\epsilon < 1$ such that $\norm{\alpha (m)}_F < \epsilon$ for all such $\alpha$. The smaller
$\epsilon$, the deeper.

\begin{thm}\label{thm:2.11} \textup{(Casselman)} \\
There exists a canonical $M$-invariant, non-degenerate pairing
\[
\inp{\cdot}{\cdot}_P : r_P^G (V) \times r^G_{\overline P} (\check V) \to \C .
\]
It is characterized by
\[
\inp{\pi (m) v}{\check v} = \inp{r_P^G (\pi) (m) j_U (v)}{j_{\overline{U}}(\check v)}_P
\]
for all $v \in V, \check v \in \check V$ and all $m$
sufficiently deep in the negative chamber of $S_M$.
\end{thm}
\begin{proof}
The statement is a combination of Proposition 4.2.3, Theorem 4.2.4 and Theorem 4.3.3 of \cite{Cas}.
The proof goes exactly the same, using Proposition \ref{prop:2.1}.
\end{proof}

\begin{lem}\label{lem:2.12}
Let $P' = M' U' \supset P$ be another pseudo-parabolic $F$-subgroup of $G$, with opposite 
$\overline{P'} = M' \overline{U'} \supset \overline{P}$. Then $P \cap M' = P \cap \overline{P'}$ 
is a pseudo-parabolic subgroup of $M' = P' \cap \overline{P'}$ and the above functors are 
transitive in the following sense:
\[
i_P^G = i_{P'}^G \circ i_{P \cap  M'}^{M'} ,\qquad
r_P^G = r_{P \cap M'}^{M'} \circ r_{P'}^G .
\]
\end{lem}
\begin{proof}
For reductive groups this is shown in \cite[Lemme VI.I.4]{Ren}. The proof relies on the structure 
of parabolic subgroups and of root subgroups.
By \eqref{eq:3.4} the latter are the same for quasi-reductive groups as for pseudo-reductive
groups. Using the structure of pseudo-parabolic subgroups, as described in \cite[\S 2.2]{CGP}, 
the argument from \cite[Lemme VI.I.4]{Ren} goes through.
\end{proof}

Another result that remains valid for quasi-reductive $F$-groups is Bernstein's geometric lemma,
which we now formulate. Let $\mc S_\es$ and $\mc P_\es$ be as in \eqref{eq:3.12}.
First we assume that $\mc P$ and $\mc P'$ are pseudo-parabolic $F$-subgroups of $\mc G$ 
containing $\mc P_\es$. Let $W(\mc G,\mc S_\es)^{\mc P',\mc P} \subset W (\mc G,\mc S_\es)$ 
be a set of representatives for
\[
W(\mc P',\mc S_\es) \backslash W(\mc G,\mc S_\es) / W(\mc P,\mc S_\es) .
\]
For each $w \in W(\mc G,\mc S_\es)^{\mc P',\mc P}$ we choose a lift $\bar w \in N_G (S_\es)$, 
and we let $\overline{W(\mc G,\mc S_\es)}^{\mc P',\mc P}$ be the collection of all these 
$\bar w$. The Bruhat decomposition (see Theorem \ref{thm:2.2}) says that
\begin{equation}\label{eq:2.17}
G = \bigsqcup\nolimits_{\bar w \in \overline{W(\mc G,\mc S_\es)}^{\mc P',\mc P}} P' \bar w P .
\end{equation}
When $\mc P$ and $\mc P'$ are general pseudo-parabolic $F$-subgroups of $\mc G$, we can adjust
the construction of the $\bar w$ as in \cite[\S V.4.7]{Ren}, so that \eqref{eq:2.17} remains
true with the new interpretation of $\overline{W(\mc G,\mc S_\es)}^{\mc P',\mc P}$.

\begin{thm}\label{thm:2.9}
Let $P,P'$ be pseudo-parabolic $F$-subgroups of $G$ and let $(\pi,V) \in \mathrm{Rep}(P / U)$.
Then $r^G_{P'} i_P^G (\pi ,V)$ admits a filtration with terms $F (V)_i$, such that the 
successive subquotients $F (V)_i / F (V)_{i-1}$ are
\[
i^{M'}_{M' \cap \bar w P \bar{w}^{-1}} \circ \bar w \circ
r^M_{M \cap \bar w P' \bar{w}^{-1}} (\pi,V) .
\]
Here $\bar w$ runs through $\overline{W(\mc G,\mc S_\es)}^{\mc P',\mc P}\!\!$ and
the functor $\bar w$ stands for $(\bar w \cdot \rho)(g) = \rho (\bar{w}^{-1} g \bar w)$.
\end{thm}
\begin{proof}
For reductive groups see \cite[\S III.1.2]{BeRu} or \cite[\S VI.5.1]{Ren}. With the
small compact open subgroups from Proposition \ref{prop:2.1}, Theorem \ref{thm:3.2} and the 
structure of pseudo-parabolic subgroups from \cite[\S 2.2]{CGP}, these two proofs remain 
valid for all quasi-reductive $F$-groups.
\end{proof}

We note one special case of Theorem \ref{thm:2.9}, where $P$ and $P'$ both equal the minimal
pseudo-parabolic $F$-subgroup $P_\es$ and $P_\es \cap \overline{P_\es} = Z_G (S_\es)$. 
It says that $r^G_{P_\es} i_{P_\es}^G (\pi,V)$ has a filtration with successive subquotients
\begin{equation}\label{eq:2.12}
(\bar w \cdot \pi, V) \qquad w \in W(\mc G,\mc S_\es) .
\end{equation}
In particular, when $\pi$ is irreducible, the $Z_G (S_\es)$-representation 
$r^G_{P_\es} i_{P_\es}^G (\pi)$ has finite length and the $\bar w \cdot \pi$ with 
$w \in W(\mc G,\mc S_\es)$ form precisely its Jordan--H\"older content.\\

Following the usual terminology for reductive groups, we say that $(\pi,V) \in \Rep (G)$ is 
supercuspidal if $r_P^G (\pi,V) = 0$ for all proper pseudo-parabolic $F$-subgroups $\mc P$ 
of $\mc G$. The following result and its proof are analogues of \cite[Corollaire VI.2.1]{Ren}.

\begin{lem}\label{lem:2.13}
Every irreducible smooth $G$-representation embeds in the parabolic induction of an
irreducible supercuspidal representation.
\end{lem}
\begin{proof} 
For any irreducible $(\rho,W) \in \Rep (G)$, Lemma \ref{lem:2.12} implies that there exists
a pseudo-parabolic $F$-subgroup $\mc P$ of $\mc G$ such that $r_P^G (\rho,W)$ is a nonzero
supercuspidal representation of $P/U$. By \eqref{eq:2.29} $r_P^G (\rho,W)$ has finite type,
and hence admits an irreducible quotient, say $(\pi,V) \in \Rep (P/U)$. By the exactness of 
$r_P^G$ that is again a supercuspidal $P/U$-representation. By Frobenius reciprocity 
\eqref{eq:2.30}
\[
0 \neq \Hom_{P/U} ( r_P^G (\rho), \pi) \cong \Hom_G (\rho, i_P^G (\pi)) . 
\]
In view of the irreducibility of $\rho$, any nonzero $G$-homomorphism 
$(\rho,W) \to i_P^G (\pi,V)$ gives the required embedding.
\end{proof}

Using the group ${}^0 G$ from \eqref{eq:3.30} we can formulate different,
equivalent characterizations of supercuspidality. 

\begin{thm}\label{thm:1.14} \textup{(Bernstein)} \ \\
For $(\pi,V) \in \Rep (G)$ the following are equivalent:
\begin{itemize}
\item $(\pi,V)$ is supercuspidal.
\item The restriction of $(\pi,V)$ to ${}^\circ G$ is a compact representation.
\item Every matrix coefficient of $(\pi,V)$ has compact support modulo $Z(G)$.
\end{itemize}
\end{thm}
\begin{proof}
The proof in the case of reductive groups is given in \cite[Th\'eor\`eme VI.2.1]{Ren}.
It uses:
\begin{itemize}
\item the properties of ${}^0 G$ from Lemma \ref{lem:3.15},
\item the small subgroups with an Iwahori decomposition from Theorem \ref{thm:3.10},
\item the Cartan decomposition from Theorem \ref{thm:3.7}.d.
\end{itemize}
With this input, the argument from \cite{Ren} is also valid for quasi-reductive groups.
\end{proof}

\subsection{Uniform admissibility} \
\label{par:adm}

Recall that a $G$-representation $(\pi,V)$ is admissible if $V^K$ has finite dimension for 
every compact open subgroup $K \subset G$. This property is important in harmonic analysis.
To establish the upcoming results, we make essential use of valuated root data and 
Bruhat--Tits theory, even in the case of pseudo-reductive groups.

\begin{thm}\label{thm:1.2}
Every irreducible smooth $G$-representation is admissible.
\end{thm}
\begin{proof}
By Lemma \ref{lem:2.13} such a representation embeds in the parabolic induction of an 
irreducible supercuspidal representation, say in $i_P^G (\pi,V)$. The compactness of $G/P$ 
\eqref{eq:3.25} implies that the functor $i_P^G$ preserves admissibility 
\cite[Lemme III.2.3]{Ren}. Therefore it suffices to prove that every irreducible 
supercuspidal representation is admissible. The argument for that given in
\cite[Th\'eor\`eme VI.2.2]{Ren} uses only Theorem \ref{thm:1.14}, Lemma \ref{lem:3.15}
and general properties of totally disconnected groups, so it also works for the
quasi-reductive group $G$.
\end{proof}

We fix a Haar measure on $G$, so that a convolution product is defined on $C_c (G)$.
For any compact open subgroup $K$, let $\mc H (G,K)$ be the convolution algebra of 
compactly supported $K$-biinvariant functions $G \to \C$. It is an involutive algebra 
with $f^* (g) = \overline{f (g^{-1})}$.

The Hecke algebra $\mc H (G)$ is defined as the union of the algebras $\mc H (G,K)$, 
over all compact open subgroups $K$ of $G$. It is well-known that Rep$(G)$ is equivalent 
with the category of nondegenerate $\mc H (G)$-modules, see \cite[Th\'eor\`eme III.1.4]{Ren}. 

\begin{lem}\label{lem:1.5}
The algebra $\mc H (G)$ is dense in $C_c (G)$ and in $L^1 (G)$.
\end{lem}
\begin{proof}
By definition $\mc H (G)$ is a subalgebra of $C_c (G) \subset L^1 (G)$.
By Proposition \ref{prop:2.1} the compact open subgroups $K$ form a neighborhood basis 
of 1 in $G$, so $\mc H (G)$ separates the points of $G$. The Stone--Weierstrass Theorem
says that $\mc H (G)$ is dense in $C_c (G)$. Since $L^1 (G)$ is a completion of 
$C_c (G)$, $\mc H (G)$ also lies dense in there.
\end{proof}

Let $K$ be any compact open subgroup of $G$ and write 
\begin{equation}\label{eq:1.21}
\mr{Irr}(G,K) = \{ V \in \Rep (G) : V \text{ is irreducible and } V^K \neq 0 \} . 
\end{equation}

\begin{prop}\label{prop:1.4}
\enuma{
\item There are functors
\[
\begin{array}{ccc} 
\Rep (G) & \longleftrightarrow & \mathrm{Mod}(\mc H (G,K)) \\
V & \mapsto & V^K \\
\mc H (G) \otimes_{\mc H (G,K)} W & \leftarrow & W
\end{array}
\]
\item These functors provide a bijection between (isomorphism classes in) Irr$(G,K)$ 
and (isomorphism classes of) irreducible representations of $\mc H (G,K)$.
\item Suppose that $\pi$ is a topologically irreducible unitary representation of $G$
on a Hilbert space $V$. Then $V^K$ is a topologically irreducible *-representation 
of $\mc H (G,K)$, or $V^K = 0$.
} 
\end{prop}
Recall that a representation of a group or an algebra on a topological vector space is
called topologically irreducible if every nonzero invariant subspace is dense. We 
also note that such a representation is usually not smooth.
\begin{proof}
For parts (a) and (b) see \cite[\S I.3.2 and Th\'eor\`eme III.1.5]{Ren}. \\
(c) Since $(\pi,V)$ is unitary, $V^K$ is a *-representation of $\mc H (G,K)$. 
Suppose that $V^K \neq 0$ and take $v \in V^K \setminus \{0\}$. 
Let $\langle K \rangle \in \mc H (G)$ be the idempotent corresponding to $K$, it is 
the unit element of $\mc H (G,K) = \langle K \rangle \mc H (G) \langle K \rangle$. 
Then inside $V^K$:
\[
\overline{\mc H (G,K) v} = \overline{\langle K \rangle \mc H (G) \langle K \rangle v} =
\langle K \rangle \overline{\mc H (G) v} .
\]
By Lemma \ref{lem:1.5} and the topological irreducibility of $V$, the right hand side 
equals $\langle K \rangle V = V^K$. Since $v \in  V^K \setminus \{0\}$ was arbitrary, 
this shows that $V^K$ is topologically irreducible.
\end{proof}

The quasi-reductive $F$-group $G$ satisfies the following uniform admissibility property:

\begin{thm}\label{thm:1.1}
For every compact open subgroup $K$ of $G$ there exists a bound $N(G,K) \in \N$ such that:
\enuma{
\item every irreducible $\mc H (G,K)$-module has dimension at most $N(G,K)$,
\item $\dim_\C (V^K) \leq N(G,K)$ for every irreducible smooth $G$-representation $V$,
\item $\dim_\C (V^K) \leq N(G,K)$ for every topologically irreducible unitary 
$G$-representation on a Hilbert space $V$.
}
\end{thm}
\begin{proof}
Let $W \in \mr{Irr}(\mc H (G,K))$ be any irreducible module. By Proposition \ref{prop:1.4}.b 
$\mc H (G) \otimes_{\mc H (G,K)} W$ is an irreducible smooth $G$-representation. By Theorem 
\ref{thm:1.2} it is admissible, so again using Proposition \ref{prop:1.4}.b we find that 
\begin{equation}\label{eq:1.29}
W = \big( \mc H (G) \otimes_{\mc H (G,K)} W \big)^K \text{ has finite dimension.}
\end{equation}
Let $K_n \subset G$ be a compact open subgroup as in Theorem \ref{thm:3.10}, such that
$K_n \subset K$. Then Irr$(G,K) \subset \mr{Irr}(G,K_n)$ and $\mc H (G,K) \subset 
\mc H (G,K_n)$. From Proposition \ref{prop:1.4}.b we get an injection
\[
\mr{Irr} (\mc H (G,K)) \to \mr{Irr} (\mc H (G,K_n)) : 
W \mapsto \mc H (G,K_n) \otimes_{\mc H (G,K)} W ,
\]
which increases the dimensions. Hence we may replace $K$ by $K_n$.

Now we follow the argument of \cite{Ber}. Let $x \in \mh A_\es$ be such that $G_x$ is a good 
maximal compact subgroup of $G$. Then $G = G_x A G_x$ 
by Theorem \ref{thm:3.7}.d. By Lemma \ref{lem:3.16}.a we can write $A = A_{S_\es} \Omega_1$,
where $\Omega_1 \subset Z_G (S_\es)$ is finite and $A_{S_\es} \subset S_\es$ is a finitely
generated semigroup representing $W(\mc G,\mc S_\es) \backslash S_\es / (S_\es)_\cpt$.

Write $\mc S_\es = \mc S \mc Z$ with $\mc S$ a maximal $F$-split torus in the subgroup of
$\mc G$ generated by the root subgroups (with respect to $\mc S_\es$). Accordingly
\[
A_{S_\es} = A^+ (A_{S_\es} \cap \mc Z (F)) \Omega_2 , 
\]
where $\Omega_2$ is finite, $A_{S_\es} \cap \mc Z (F)$ is a subgroup of $\mc Z (F)$ 
representing $\mc Z (F) / \mc Z (F)_\cpt$ and $A^+$ represents 
$W(\mc G,\mc S_\es) \backslash S / S_\cpt$. We obtain
\begin{equation}\label{eq:1.16}
G = G_x A^+ (A_{S_\es} \cap \mc Z (F)) (\Omega_2 \Omega_1) G_x
\end{equation}
where $A^+ \subset S_\es$ is a finitely generated semigroup, $A_{S_\es} \cap \mc Z (F)$ 
central and $\Omega_2 \Omega_1$ finite. This and Theorem \ref{thm:3.10} (parts a--d) take
care of all the assumptions on $G$ and the $K_n$ made in \cite{Ber}. 

Then we may use \cite[Assertion A]{Ber}, which together with \eqref{eq:1.29} gives us 
precisely part (a). (See also \cite[\S 5]{Rud} for more details about Bernstein's arguments.) 
Applying Proposition \ref{prop:1.4}.b, we get part (b). By part (a) and 
\cite[Theorem 6.27]{Rud}, every topologically irreducible *-representation of $\mc H (G,K)$ 
has dimension at most $N(G,K)$. We combine this with Proposition \ref{prop:1.4}.c to
obtain part (c).
\end{proof}

We remark that in Theorem \ref{thm:1.1}.c the finite dimensional module $V^K$ is also 
irreducible as a representation of $\mc H (G,K)$ in the purely algebraic sense -- 
in finite dimensional vector spaces the topology barely makes any difference. From 
Proposition \ref{prop:1.4}.b we then see that $\mc H (G) V^K$ is an irreducible smooth
$G$-subrepresentation of $V$. It is the same for every compact open subgroup $K$, 
namely the space of smooth vectors 
\[
V_{sm} := \{ v \in V : \{ g \in G : \pi (g) v = v \} \text{ is open in } G \} . 
\]
With Lemma \ref{lem:1.5} we find that, for every topologically irreducible unitary \\
$G$-representation $\pi$ on a Hilbert space $V$,
\begin{equation}
V_{sm} \text{ is a dense irreducible smooth $G$-subrepresentation of } V. 
\end{equation}
Finally, we turn to analytic properties of group $C^*$-algebras for quasi-reductive groups.
We refer to \cite[\S 13.9]{Dix} for the terminology. 

\begin{cor}\label{cor:1.3}
Let $C^* (G)$ be the maximal $C^*$-algebra of $G$.
\enuma{
\item Let $\pi$ be a topologically irreducible unitary $G$-representation on a Hilbert 
space $V$. For every $f \in C^* (G) ,\; \pi (f)$ is a compact operator on $V$. 
In other words, $C^* (G)$ is a liminal $C^*$-algebra.
\item The group $G$ has type I.
}
\end{cor}
\begin{proof}
(a) Consider any $f \in \mc H (G,K)$. Then $\pi (f) V \subset V^K$, so by Theorem 
\ref{thm:1.1}.c $\pi (f)$ has finite rank. Thus $\pi (\mc H (G))$ consists of finite 
rank operators on $V$. By Lemma \ref{lem:1.5} $\mc H (G)$ is dense in $C^* (G)$, so 
$\pi (C^* (G))$ is contained in the closure of the finite rank operators in the algebra 
of bounded linear operators on $V$. By definition that is the algebra of compact operators
on $V$.\\
(b) By part (a) and \cite[Th\'eor\`eme 5.5.2]{Dix}, $C^* (G)$ is of type I, which
by definition means that $G$ has type I.
\end{proof}

\newpage

\section{Disconnected groups} 
\label{sec:disc}

For the applications to the Baum--Connes conjecture that we have in mind, we need 
to consider disconnected algebraic groups. Let $F$ be an arbitrary field.
The Bruhat decomposition was announced for any connected linear algebraic $F$-group $\mc G$ 
by Borel and Tits \cite{BoTi}. Based on \cite{CGP} and for use in Paragraph \ref{par:Xi}, 
we will show that the connectedness assumption is superfluous.

Let $\mc G^+$ be a linear algebraic $F$-group (not necessarily smooth), with connected
component $\mc G$ (which is smooth, see \cite[Theorem 4.3.7.i]{Spr}). Let $\mc S$ be a maximal 
$F$-split torus of $\mc G$, let $Z_{\mc G^+ (F)}(\mc S)$ (resp. $N_{\mc G^+}(\mc S)$) be
its centralizer (resp. normalizer) in $\mc G^+ (F)$. The group $\mc G^+ (F)$ acts by 
conjugation on the set of maximal $F$-split tori of $\mc G$. Since all maximal $F$-split 
tori of $\mc G$ are already conjugate by elements of $\mc G (F)$ \cite[Theorem C.2.3]{CGP}, 
\begin{equation}\label{eq:1.11}
\mc G^+ (F) = \mc G (F) N_{\mc G^+ (F)}(\mc S) = N_{\mc G^+ (F)}(\mc S) \mc G (F) . 
\end{equation}
The Weyl group of $(\mc G,\mc S)$ is \cite[Proposition C.2.10]{CGP}
\begin{equation}\label{eq:2.26}
W(\mc G, \mc S) := N_{\mc G}(\mc S) / Z_{\mc G}(\mc S) \cong 
N_{\mc G (F)}(\mc S) / Z_{\mc G (F)}(\mc S) .
\end{equation}
The finite group 
\[
W(\mc G^+ (F),\mc S) := N_{\mc G^+ (F)}(\mc S) / Z_{\mc G^+ (F)}(\mc S)
\]
contains $W(\mc G,\mc S)$ as a normal subgroup.

\subsection{Pseudo-parabolic subgroups} \

For possibly disconnected groups, \cite[Definition 2.2.1]{CGP} says that every pseudo-parabolic
$F$-subgroup of $\mc G^+$ is of the form $\mc P^+ = \mc P_{\mc G^+}(\lambda) \mc R_{u,F}(\mc G)$.
Here $\lambda : GL_1 \to \mc G$ is a $F$-rational cocharacter and $\mc P_{\mc G^+}(\lambda)$
is defined by 
\begin{equation}\label{eq:1.55}
\mc P_{\mc G^+}(\lambda) (F) = \big\{ g \in \mc G^+ (F) : 
\lim_{t \to 0} \lambda (t) g \lambda (t)^{-1} \text{ exists in } \mc G^+ (F) \big\} .
\end{equation}
We note that $\mc P^+$ need not be connected if $\mc G^+$ is not, for $Z_{\mc G^+}(\lambda) := 
\mc P_{\mc G^+}(\lambda) \cap \mc P_{\mc G^+}(-\lambda)$ can contain elements of 
$\mc G^+ (F) \setminus \mc G (F)$. In particular $\mc P^+ (F)$ always contains the full 
$\mc G^+ (F)$-centralizer of the smallest $F$-torus of $\mc G$ through which $\lambda$ factors. 

Let $\mc P_\es$ be a minimal pseudo-parabolic $F$-subgroup of $\mc G$. We may assume that it 
contains $\mc S$. From \eqref{eq:1.55} we see that $\mc P_\es \cap N_{\mc G}(\mc S) = 
Z_{\mc G}(\mc S)$. By \cite[Proposition C.2.4]{CGP} 
\begin{equation}\label{eq:2.22}
\mc P_\es = Z_{\mc G}(\mc S) \ltimes \mc R_{us,F}(\mc P_\es) .
\end{equation}

\begin{thm}\label{thm:1.30}
Let $\mc G$ be any connected linear algebraic group defined over $F$.
\enuma{ 
\item All minimal pseudo-parabolic $F$-subgroups of $\mc G$ are conjugate under $\mc G (F)$.
\item Every pseudo-parabolic $F$-subgroup of $\mc G$ equals its own normalizer. 
}
\end{thm}
\begin{proof}
(a) This was announced in \cite{BoTi}, see \cite[Theorem C.2.5]{CGP} for a proof.\\
(b) As every pseudo-parabolic $F$-subgroup of $\mc G$ contains $\mc R_{u,F}(\mc G)$,
the statement reduces to the pseudo-reductive $F$-group $\mc G / \mc R_{u,F}(\mc G)$.
For that group it is given in \cite[Proposition 3.5.7]{CGP}.
\end{proof}

Let $N_{\mc G^+ (F)}(\mc P_\es, \mc S)$ be the simultaneous normalizer of $\mc P_\es$ and
$\mc S$ in $\mc G^+ (F)$, and put
\[
W (\mc G^+ (F),\mc P_\es,\mc S) = N_{\mc G^+ (F)}(\mc P_\es, \mc S) / Z_{\mc G^+ (F)}(\mc S) .
\]
The group $Z_{\mc G^+ (F)}(\mc S)$ normalizes every pseudo-parabolic $F$-subgroup
$\mc P$ of $\mc G$ containing $Z_{\mc G}(\mc S)$, for it fixes any cocharacter 
$GL_1 \to \mc S$ determining $\mc P$. Hence the subgroup $W (\mc G^+ (F),\mc P_\es,\mc S)$
of $W (\mc G^+ (F),\mc S)$ is precisely the stabilizer of $\mc P_\es$, and it acts
naturally on the set of pseudo-parabolic subgroups of $\mc G$ containing $Z_{\mc G}(\mc S)$.

\begin{lem}\label{lem:1.18}
$W(\mc G^+ (F),\mc S) = W (\mc G^+ (F),\mc P_\es,\mc S) \ltimes W(\mc G,\mc S)$.
\end{lem}
\begin{proof}
The $F$-group $\overline{\mc G} = \mc G / \mc R_{u,F}(\mc G)$ is pseudo-reductive and 
the image of $\mc S$ is a maximal $F$-split torus in there, say $\overline{\mc S}$.
These groups give rise to a root system $\Phi (\overline{\mc G},\overline{\mc S})$, 
whose Weyl group can be identified with $W(\mc G,\mc S)$ \cite[\S C.2.13]{CGP}. 
By \cite[Theorem C.2.15]{CGP} the set of minimal pseudo-parabolic $F$-subgroups of $\mc G$ 
containing $Z_{\mc G}(\mc S)$ is naturally in bijection with the collection of positive 
subsystems of $\Phi (\overline{\mc G},\overline{\mc S})$. 

Hence the Weyl group $W(\mc G,\mc S)$ acts simply transitively on both these sets. This
implies that $W(\mc G^+ (F),\mc S)$ is the semidirect product of its normal subgroup
$W(\mc G,\mc S)$ and the stabilizer of any minimal pseudo-parabolic $F$-subgroup.
\end{proof}

In contrast to Theorem \ref{thm:1.30}, a (minimal) pseudo-parabolic $F$-subgroup of
$\mc G^+$ can be properly contained in its own normalizer. We say that a pseudo-parabolic 
$F$-subgroup $\mc P^+$ of $\mc G^+$ is good, or more precisely $F$-good, if
\[
\mc P^+ (F) \mc G (F) = \mc G^+ (F).
\] 
It follows from Theorem \ref{thm:1.30} 
that every good pseudo-parabolic $F$-subgroup of $\mc G^+$ does equal its own normalizer.

By definition every pseudo-parabolic $F$-subgroup of the connected algebraic group $\mc G$ 
is good. It is easy to see that good pseudo-parabolic $F$-subgroups always exist in 
$\mc G^+$. Namely, let $\lambda : GL_1 \to \mc S$ be a cocharacter with 
$\mc P_\es = \mc P_{\mc G}(\lambda) \mc R_{u,F}(\mc G)$. For any 
$w \in W (\mc G^+ (F),\mc P_\es,\mc S)$, $w(\lambda)$ is another such cocharacter. 
Consider the $W (\mc G^+ (F),\mc P_\es,\mc S)$-invariant cocharacter
\[
\lambda^+ = \prod\nolimits_{w \in W (\mc G^+ (F),\mc P_\es,\mc S)} w(\lambda).
\]
With \cite[Proposition 2.2.4]{CGP}, one sees that $\mc P_{\mc G}(\lambda^+) = 
\mc P_{\mc G}(\lambda)$. By the \\ $W (\mc G^+ (F),\mc P_\es,\mc S)$-invariance, 
$\mc P_{\mc G^+}(\lambda^+)$ contains $N_{\mc G^+ (F)}(\mc P_\es, \mc S)$. By \eqref{eq:1.11} 
and Lemma \ref{lem:1.18} this means that 
\begin{equation}\label{eq:1.61}
\mc P_\es^+ := \mc P_{\mc G^+}(\lambda^+) \mc R_{us,F}(\mc G)
\end{equation}
is a good pseudo-parabolic $F$-subgroup of $\mc G^+$.

Now we can prove a Bruhat decomposition for possibly disconnected linear algebraic groups
over arbitrary fields.

\begin{thm}\label{thm:2.2}
Let $\mc P^+ \!$ and $\mc P'$ be pseudo-parabolic $F$-subgroups of 
$\mc G^+ \!$ containing $\mc P_\es$. 
\enuma{
\item $\mc G^+ (F) = \mc P' (F) N_{\mc G^+ (F)}(\mc S) \mc P^+ (F)$.
\item Write $W(\mc P, \mc S) = N_{\mc P}(\mc S) / Z_{\mc P}(\mc S)$. Suppose that $\mc P^+$ is 
$F$-good. The map
\[
W(\mc P' \cap \mc G,\mc S) \backslash W(\mc G,\mc S) / W(\mc P^+ \cap \mc G,\mc S)
\to (\mc P' \cap \mc G)(F) \backslash \mc G^+ (F) / \mc P^+ (F)
\]
is bijective.
}
\end{thm}
\begin{proof}
By \cite[Theorem C.2.20]{CGP} there is a Tits system $(\mc G (F), \mc P_\es (F), \mc S (F), R)$ 
with Weyl group $W(\mc G, \mc S)$. Here $R$ denotes the set of 
simple reflections in $W(\mc G,\mc S)$ determined by $\mc P_\es$. The theorem
for $\mc G (F)$ then follows from the properties of Tits systems \cite[\S IV.2.5]{Bou}. 

Choose a set of representatives $W^+ \subset N_{\mc G^+ (F)}(\mc P_\es,\mc S)$ for 
$W (\mc G^+ (F),\mc P_\es,\mc S)$. By Lemma \ref{lem:1.18} and \eqref{eq:1.11}
\begin{equation}\label{eq:1.12}
\mc G^+ (F) = \sqcup_{w \in W^+} \mc G (F) Z_{\mc G^+ (F)}(\mc S) w. 
\end{equation}
With the results for $\mc G (F)$ and 
$Z_{\mc G^+ (F)}(\mc S) \subset N_{\mc G^+ (F)}(\mc P_\es)$ we obtain
\begin{align*}
\mc G^+ (F) & = \mc P_\es (F) N_{\mc G (F)}(\mc S) \mc P_\es (F) Z_{\mc G^+ (F)}(\mc S) W^+ \\
& \subset \bigcup\nolimits_{w \in W^+} \mc P_\es (F) N_{\mc G (F)}(\mc S)
Z_{\mc G^+ (F)}(\mc S) w (w^{-1} \mc P_\es (F) w) \\
& = \bigcup\nolimits_{w \in W^+} \mc P_\es (F) Z_{\mc G^+ (F)}(\mc S) N_{\mc G (F)}(\mc S)
w \mc P_\es (F) \\
& = \mc P_\es (F) N_{\mc G^+ (F)}(\mc S) \mc P_\es (F) .
\end{align*}
Clearly this implies part (a). Moreover $Z_{\mc G^+ (F)}(\mc S) \subset \mc P^+$ and 
$\mc P^+$ is $F$-good, so it contains $\mc P_\es^+ \supset N_{\mc G^+ (F)}(\mc P_\es,\mc S)$.
Therefore map in part (b) is surjective. Furthermore the goodnesss of $\mc P^+$ implies that
$\mc G^+ (F) / \mc P^+ (F)$ can be identified with $\mc G (F) / (\mc P^+ \cap \mc G)(F)$.
Now the injectivity in (b) follows from part (b) for $\mc G$.
\end{proof}

\subsection{Good maximal compact subgroups} \
\label{par:Xi}

In this paragraph we assume that $F$ is a non-archimdean local field and that $\mc G^+$ is a 
possibly disconnected linear algebraic $F$-group (not necessarily smooth), whose identity
component $\mc G$ is quasi-reductive over $F$. We continue the use of the notations of
Section \ref{sec:rep} for $\mc G$ and $G$.
The locally compact group $G^+ = \mc G^+ (F)$ contains $G$ as an open normal subgroup of 
finite index. By Lemma \ref{lem:1.4} and \eqref{eq:1.35} 
\begin{equation}
G \text{ is unimodular, hence so is } G^+ .
\end{equation}
We would like $G^+$ to act naturally on the affine building $\mc B (\mc G,F)$. 
Unfortunately, it is not even clear that $N_{G^+}(S_\es)$ acts naturally on
\[
X_* (\mc S_\es) \otimes_\Z \R \quad \text{or} \quad
\mc B (\mc G / \mc D (\mc G) \mc R_{u,F}(\mc G),F) = X_* (\mc Z) \otimes_\Z \R.
\]
We can avoid this problem by assuming
\begin{equation}\label{eq:1.56}
\begin{split}
& \text{the action of } G \text{ on } 
\mc B (\mc G / \mc D (\mc G) \mc R_{u,F}(\mc G),F) \text{ by translations}\\
& \text{extends to an isometric action of } G^+ .
\end{split}
\end{equation}
Let us discuss which parts of Section \ref{sec:quasi} are valid for $\mc G^+$ under this
assumption. The conjugation action of $G^+$ on $G$ induces, via Proposition \ref{prop:3.19}.c,
a canonical isometric action of $G^+$ on $\mc{BT} (\mc G,F)$. Hence the action of $G$ on
\[
\mc B (\mc G,F) = \mc{BT}(\mc G,F) \times \mc B (\mc G / \mc D (\mc G) \mc R_{u,F}(\mc G),F) 
\]
extends to an isometric action of $G^+$. This $G^+$-action is proper because $G$ has finite
index in $G^+$ and the $G$-action is proper (Theorem \ref{thm:3.4}.e).
The proof of Proposition \ref{prop:3.18} also works for $G^+$. In particular it shows that
the $G^+$-stabilizer of any vertex of $\mc B (\mc G,F)$ is a maximal compact subgroup of $G^+$.
Theorem \ref{thm:3.10} is valid for $G^+$ with the same subgroups $K_n$. We note that to make
them normal in $G^+_{x_s}$ we have to take care that in Lemma \ref{lem:3.16}.c (which holds
just as well for $G^+$) we select subgroups of $Z_Q (S_\es)$ that are normal in $N_{G^+}(S_\es)$.

Using \eqref{eq:1.12}, one sees easily that Theorems \ref{thm:3.5} and \ref{thm:3.7} (except
part (d) of the latter) for $\mc G$ easily imply the same statements for $\mc G^+$. The Iwasawa
and Cartan decompositions do not generalize automatically to $\mc G^+$, for isotropy groups of
special vertices are not always good as maximal compact subgroups of $G^+$. \\

Almost everything in Paragraph \ref{par:ind} generalizes to $G^+$, even without 
\eqref{eq:1.56}. Indeed, these are statements about representations of totally 
disconnected locally compact groups, for which connectedness as algebraic groups plays only 
a minor role. Since $\mc R_{us,F}(\mc P^+) = \mc R_{us,F}(\mc P^+ \cap \mc G)$, the functors 
$r_P^G$ and $r_{P^+}^{G^+}$ do exactly the same on the vector space underlying a $G^+$-representation.

We note that for Theorem \ref{thm:2.9} we need the Bruhat decomposition for possibly 
disconnected linear algebraic groups, as established in Theorem \ref{thm:2.2}.
This entails that we must restrict Theorem \ref{thm:2.9} to the cases 
$r_{P'}^{G^+} i_{P^+}^{G^+}(\pi,V)$ where $\mc P^+$ is a good pseudo-parabolic $F$-subgroup 
of $\mc G^+$. In particular the important case \eqref{eq:2.12} is valid for $G^+$.

For Theorem \ref{thm:1.14} and most of Paragraph \ref{par:adm} we need the Cartan decomposition.
To ensure that our arguments for these results hold for $G^+$, we need to assume \eqref{eq:1.56} 
and that $G^+$ possesses a good maximal compact subgroup in the sense of \eqref{eq:3.58}.
Unfortunately, some disconnected groups do not possess any good maximal compact subgroup.
Surprisingly, that can fail already for groups with a simple neutral component.

\begin{ex}
Consider the group $SL_4 (F)$. It is simply connected, so its maximal compact subgroups are
precisely the stabilizers of special vertices in its Bruhat--Tits building \cite[\S 4.4.6]{BrTi1}.
It is known that all vertices of $\mc{BT}(SL_4,F)$ are special and that they form 4 orbits
under $SL_4 (F)$. 

Let $\tau$ be the inverse transpose mapping, an outer automorphism of $SL_4$ of order 2.
Write $w^+ = \mr{Ad}\left( \begin{smallmatrix}
 & & & 1 \\
 & & 1 & \\
 & 1 & & \\
\omega_F & & & 
\end{smallmatrix} \right) \circ \tau$, an automorphism of $SL_4 (F)$ whose square 
is inner. Then $SL_4^+ := SL_4 \cup w^+ SL_4$ is a disconnected 
reductive $F$-group with $W(SL_4^+ (F),\mc S) \cong \{ \mr{id},\tau \} \ltimes S_4$
(where $\mc S$ denotes the group of diagonal matrices in $SL_4$).

One checks that the natural action of $w^+$ on $\mc{BT}(SL_4,F)$ from Proposition 
\ref{prop:3.19}.c sends every orbit of vertices to another $SL_4 (F)$-orbit. Hence, for any 
$g \in SL_4 (F)$, $w^+ g$ does not fix any vertex of $\mc{BT} (SL_4,F)$. In other words, 
there does not exist any vertex of $\mc{BT} (SL_4,F)$ whose $SL_4^+ (F)$-stabilizer contains a 
representative for $w^+ \in W(SL_4^+ (F),\mc S)$, and $SL_4^+ (F)$ does not possess any 
good maximal compact subgroup.
\end{ex}

To alleviate this inconvenience we formulate two variations on the Iwasawa and Cartan 
decompositions for $G^+$. Unfortunately, part (c) still does not seem sufficient as a 
substitute for the Cartan decomposition in the proofs of Theorems \ref{thm:1.14} and 
\ref{thm:1.1}.

\begin{thm}\label{thm:1.9}
Let $\mc G^+$ be a linear algebraic $F$-group whose connected component $\mc G$ is
quasi-reductive over $F$.
\enuma{
\item Suppose that \eqref{eq:1.56} holds and that $G^+$ possesses a good maximal compact
subgroup $K^+$, in the sense of \eqref{eq:3.58}. Then $G^+$ admits Iwasawa and Cartan
decompositions with respect to $K^+$, just as in Theorems \ref{thm:3.6} and \ref{thm:3.7}.d.
\item Let $K$ be a good maximal compact subgroup of $G$. Then 
\[
G^+ = K \mc P^+ (F) = \mc P^+ (F) K
\]
for every good pseudo-parabolic $F$-subgroup $\mc P^+$ of $\mc G^+$.
\item Let $x,y \in \mh A_\es$ be such that $G_x$ and $G_x$ are good maximal compact subgroups
of $\mc G (F)$. Using Lemma \ref{lem:1.18} we choose a set of representatives 
$N^+ \subset N_{G^+}(\mc P_\es,S_\es)$ for $N_{G^+}(S_\es) / N_G (S_\es)$. 

Then there exist cones $A_n \subset Z_G (S_\es) / Z_G (S_\es)_\cpt$ for $n \in N^+$, such that
\[
G_x \backslash G^+ / G_y = 
\bigsqcup\nolimits_{n \in N^+} \bigsqcup\nolimits_{a \in A_n} G_x a n G_y . 
\]
}
\end{thm}
\begin{proof}
(a) Under these conditions the proofs of Theorems \ref{thm:3.6} and \ref{thm:3.7}.d are 
valid for $\mc G^+ (F)$.\\
(b) By the Iwasawa decomposition for $G$:
\[
K \mc P^+ (F) = K (\mc P^+ \cap \mc G)(F) \mc P^+ (F) = \mc G (F) \mc P^+ (F) . 
\]
The definition of ``good" for pseudo-parabolic $F$-subgroups of $\mc G^+$ says that the
right hand side equals $\mc G^+ (F)$.\\
(c) First we note that by \eqref{eq:1.11}
\begin{equation}\label{eq:1.57}
G_x \backslash G^+ / G_y = \bigsqcup\nolimits_{n \in N^+} G_x \backslash G n / G_y \cong
\bigsqcup\nolimits_{n \in N^+} G_x \backslash G / (n G_y n^{-1}) \cdot n . 
\end{equation}
Conjugation with $n \in N^+$ is an automorphism of $\mc G (F)$, so $n G_y n^{-1}$ is again
a good maximal compact subgroup of $G$. By Proposition \ref{prop:3.19}.a it fixes a unique
point of $\mc{BT}(\mc G,F)$, which lies in the apartment associated to the maximal
$F$-split torus $n \mc S_\es n^{-1} = \mc S_\es$. By the Bruhat--Tits fixed point theorem
$n G_y n^{-1}$ equals the $G$-stabilizer of a point of $\mh A_\es$, say $n(y)$. Choose
$A_n$ as in \eqref{eq:3.60}, for $G,x$ and $n(y)$. By Theorem \ref{thm:3.7}.d 
$G = \bigsqcup_{a \in A_n} G_x a G_{n(y)}$. Now the result follows from \eqref{eq:1.57}.
\end{proof}

To apply some work of Harish--Chandra, Waldspurger \cite{Wal} and Lafforgue \cite{Laf} 
to quasi-reductive $F$-groups, we want to exhibit a well-behaved generalization of 
the Harish-Chandra's $\Xi$-function \cite[\S II.1]{Wal}. Some results in Paragraph \ref{par:ind} 
were selected to facilitate its construction and to show its properties.
Let $K \subset G$ be a good maximal compact subgroup which fixes a point $x' \in \mh A_\es$.
(It exists since every $G$-orbit in $\mc B (\mc G,F)$ intersects every apartment.) 

Recall the maximal $F$-split torus $\mc S_\es$ and the minimal good pseudo-parabolic $F$-subgroup
$\mc P_\es^+$ of $\mc G^+$ from \eqref{eq:1.61}. With Theorem \ref{thm:1.9}.b we extend 
$\delta_{P_\es^+}$ to a function $G \to \R_{>0}$ which is right-$K$-invariant. Let $\mu_K$ 
be the Haar measure on $K$ with $\mu_K (K) = 1$. We define
\begin{equation}
\Xi (g) = \int_K \delta_{P_\es^+} (k g)^{1/2} \textup{d}\mu_K (k) \qquad g \in G^+ .
\end{equation}
By the compactness of $K$, this is a continuous function $G^+ \to \R_{>0}$.

Let $(\pi,V) = i_{P_\es^+}^{G^+} (\mr{triv},\C)$ be the normalized parabolic induction of the 
trivial repre\-sen\-tation of $N_{G^+}(\mc P_\es,S_\es) \cong \mc P_\es^+ (F) / 
\mc R_{us,F}(\mc P_\es^+)(F)$.
We saw in \eqref{eq:2.13} that it is a unitary smooth $G^+$-representation. Let $f_K \in V$ be
the unique $K$-invariant function with $f_K (1) = 1$. (The uniqueness follows from
the Iwasawa decomposition $G^+ = P_\es^+ K$.) The definition of the inner product on $V$ 
\cite[Proposition 3.1.3]{Cas} works out to
\begin{equation}\label{eq:2.14}
\Xi (g) = \inp{\pi (g) f_K}{f_K} \qquad g \in G^+ .
\end{equation}
The unitarity of $\pi$ implies that, for all $g \in G^+$:
\begin{equation}\label{eq:2.25}
\Xi (g) \leq \Xi (1) = \inp{f_K}{f_K} = \int_K f_K^2 (k) \textup{d}\mu_K (k) = 
\int_K \textup{d} \mu_K (k) = 1.
\end{equation}
Let us recall the construction of a length function on $G^+$ from \cite[p. 242]{Wal}.
Choose an embedding of $F$-groups $\tau : \mc G^+ \to GL_n$. Since every maximal compact 
subgroup of $GL_n (F)$ is conjugate to $GL_n (\mc O_F)$, we may assume that $\tau (K) 
\subset GL_n (\mc O_F)$. For $g \in G^+$, let $\mr{mc}(g,g^{-1})$ be the collection of all
matrix coefficients of $\tau (g)$ and $\tau (g)^{-1}$. Let $\nu_F$ be the discrete 
valuation of $F$, and define
\[
\ell (g) = \sup \{ -\nu_F (x) : x \in \mr{mc}(g,g^{-1}) \} .  
\]
This is a $K$-biinvariant length function $G^+ \to \R_{\geq 0}$.

\begin{lem}\label{lem:2.10}
The following integral converges for sufficiently large $t \in \R_{>0}$:
\[
\int_{G^+} \Xi (g)^2 (1 + \ell (g))^{-t} \textup{d}\mu_{G^+} (g) .
\]
\end{lem}
\begin{proof}
We can use the proof of \cite[Lemme II.1.5]{Wal} (which is the desired statement, but 
only for connected reductive $F$-groups), when we can show that all the ingredients are
also valid for $G^+ = \mc G^+ (F)$.

The first ingredient is an estimate on the volume of double cosets $KaK$ \cite[p. 241]{Wal}
with $a \in N_{G^+}(P_\es ,S_\es)$. This holds for $G^+$ because of the Cartan 
decomposition (Theorem \ref{thm:1.9}.c) and the existence of compact open subgroups with an 
Iwahori decomposition (Proposition \ref{prop:2.1} and Theorem \ref{thm:3.10}). Notice that
elements of $N_{G^+}(P_\es ,S_\es)$ normalize $U$ and hence also $\overline{U}$.

The most involved step in Waldspurger's argument is an estimate on $\Xi (g)$ for $g \in
N_{G^+} (P_\es,S_\es)$, namely \cite[Lemme II.1.1]{Wal}. An essential ingredient of the proof
of the latter is the normalized Jacquet restriction $r^{G^+}_{P_\es^+}(\pi) = r^{G^+}_{P_\es^+} 
i^{G^+}_{P_\es^+} (\mathrm{triv})$. By the special case \eqref{eq:2.12} of Theorem \ref{thm:2.9}, 
the trivial representation of $N_{G^+}(P_\es,S_\es)$ is the only irreducible constituent of 
$r^{G^+}_{P_\es^+} i^{G^+}_{P_\es^+} (\mathrm{triv})$, and it appears with multiplicity
$|W(\mc G,\mc S_\es)|$. Furthermore the argument for \cite[Lemme II.1.1]{Wal} uses
Theorem \ref{thm:2.11} and that $N_{G^+}(P_\es,S_\es) = S_\es C$ for some compact subset $C$, 
which follows from Lemma \ref{lem:1.6}.b (and its analogue Lemma \ref{lem:3.16}.a 
for quasi-reductive $F$-groups). With all this at hand, Waldspurger's proof goes through, 
and provides bounds on $\Xi (g)$ in terms of the length function $\ell$ on $G^+$.

Waldpurger uses the above estimates to reduce the integral \eqref{eq:2.6} to a sum over the set 
\[
W(\mc G,\mc S_\es) \backslash N_{G^+} (P_\es,S_\es) / Z_{G} (S_\es)_\cpt \cong
\bigsqcup\nolimits_{n \in N^+} A_n n
\]
appearing in the Cartan decomposition. In Lemmas \ref{lem:2.7} and \ref{lem:3.16}.b we showed 
that $Z_G (S_\es) / Z_G (S_\es)_\cpt$ is a finitely generated free abelian group, just 
as in the Cartan decomposition for connected reductive $p$-adic groups. It is normal 
and of finite index in $N_{G^+} (P_\es,S_\es) / Z_{G} (S_\es)_\cpt$. Consequently the same 
estimates as used in \cite[Lemme II.1.5]{Wal} work in our case, as long as we take $t$ 
sufficiently large compared to dim$(\mh A_\es)$.
\end{proof}


\newpage

\begin{thebibliography}{99}

\bibitem[BCH]{BCH} P.F. Baum, A. Connes, N. Higson, 
``Classifying space for proper actions and K-theory of group $C^*$-algebras",
pp. 240--291 in: \emph{$C^*$-algebras: 1943--1993. A fifty year celebration}, 
Contemp. Math. {\bf 167}, American Mathematical Society, Providence RI, 1994

\bibitem[Ber]{Ber} J. Bernstein, ``All reductive $p$-adic groups are tame",
Functional Analysis and Its Applications {\bf 8.2} (1974), 91--93

\bibitem[BeRu]{BeRu} J. Bernstein, K.E. Rumelhart, ``Representations of $p$-adic groups",
preprint, 1993

\bibitem[Bor]{Bor} A. Borel, \emph{Linear algebraic groups},
Mathematics lecture note series, W.A. Benjamin, 1969

\bibitem[BoTi]{BoTi} A. Borel, J. Tits, ``Th\'eor\`emes de structure et de conjugaison pour
les groupes alg\'ebriques lin\'eaires", C.R. Acad. Sci. Paris {\bf 287} (1978), 55--57

\bibitem[Bou]{Bou} N. Bourbaki, \emph{Groupes et  alg\`ebres de Lie. Chapitres IV, V et VI},
\'El\'ements de math\'ematique {\bf XXXIV}, Hermann, Paris, 1968

\bibitem[BrTi1]{BrTi1} F. Bruhat, J. Tits, ``Groupes r\'eductifs sur un corps local: I.
Donn\'ees radicielles valu\'ees",
Publ. Math. Inst. Hautes \'Etudes Sci. {\bf 41} (1972), 5--251

\bibitem[BrTi2]{BrTi2} F. Bruhat, J. Tits, ``Groupes r\'eductifs sur un corps local: II.
Sch\'emas en groupes. Existence d'une donn\'ee radicielle valu\'ee",
Publ. Math. Inst. Hautes \'Etudes Sci. {\bf 60} (1984), 5--184

\bibitem[Cas]{Cas} W. Casselman, ``Introduction to the theory of admissible representations
of $p$-adic reductive groups", preprint, 1995

\bibitem[CEN]{CEN} J. Chabert, S. Echterhoff, R. Nest,
``The Connes--Kasparov conjecture for almost connected groups and for linear $p$-adic groups",
Publ. Math. Inst. Hautes \'Etudes Sci. {\bf 97} (2003), 239--278

\bibitem[Con]{Con} B. Conrad, ``Finiteness theorems for algebraic groups over
function fields", Compositio Math. {\bf 148} (2012), 555--639

\bibitem[CGP]{CGP} B. Conrad, O. Gabber, G. Prasad,
\emph{Pseudo-reductive groups. Second edition},
New Mathematical Monographs {\bf 26}, Cambridge University Press, 2015

\bibitem[CoPr1]{CP} B. Conrad, G. Prasad, \emph{Classification of pseudo-reductive groups},
Annals of Mathematics Studies, Princeton University Press, 2015

\bibitem[CoPr2]{CP2} B. Conrad, G. Prasad, ``Structure and classification of pseudo-reductive groups", 
Proc. Symp. Pure Math. {\bf 14} (2017)

\bibitem[Del]{Del} P. Deligne, ``Le support du caract\`ere d'une representation supercuspidale'',
C.R. Acad. Sci. Paris Ser. A {\bf 283} (1976), 155--157

\bibitem[Dix]{Dix} J. Dixmier, \emph{Les C*-alg\`ebres et leurs representations},
Cahiers Scientifiques {\bf 29}, Gauthier-Villars \'Editeur, Paris, 1969

\bibitem[Laf]{Laf} V. Lafforgue, ``K-th\'eorie bivariante pour les alg\`ebres de Banach
et conjecture de Baum--Connes", Invent. Math. {\bf 149.1} (2002), 1--95

\bibitem[Loi]{Loi} B. Loisel, ``On profinite subgroups of an algebraic group over
a local field", arXiv:1607.05550, 2016

\bibitem[Oes]{Oes} J. Oesterl\'e, ``Nombres de Tamagawa et groupes unipotents en
caract\'eristique p", Inv. Math. {\bf 78} (1984), 13--88

\bibitem[PlRa]{PlRa} V.P. Platonov, A. Rapinchuk,
\emph{Algebraic groups and number theory}, Pure and Appliead Mathematics {\bf 139},
Academic Press, 1994

\bibitem[Ren]{Ren} D. Renard, \emph{Repr\'esentations des groupes r\'eductifs p-adiques},
Cours sp\'ecialis\'es {\bf 17}, Soci\'et\'e Math\'ematique de France, 2010

\bibitem[R\"ud]{Rud} T. R\"ud, ``Admissibility of representations of
totally disconnected locally compact groups", Master Thesis, EPFL and UBC, 2016

\bibitem[ScSt]{ScSt} P. Schneider, U. Stuhler,
``Representation theory and sheaves on the Bruhat--Tits building",
Publ. Math. Inst. Hautes \'Etudes Sci. {\bf 85} (1997), 97--191

%\bibitem[Sil]{Sil} A.J. Silberger,
%\emph{Introduction to harmonic analysis on reductive $p$-adic groups},
%Mathematical Notes {\bf 23}, Princeton University Press, Princeton NJ, 1979

\bibitem[Spr]{Spr} T. Springer, \emph{Linear algebraic groups. Second edition},
Progress in Mathematics {\bf 9}, Birkh\"auser, Boston MA, 1998

\bibitem[Tit]{Tit} J. Tits, ``Reductive groups over local fields",
pp. 29--69 in: \emph{Automorphic forms, representations and L-functions Part I},
Proc. Sympos. Pure Math. {\bf 33}, American Mathematical Society, Providence RI, 1979

\bibitem[Wal]{Wal} J.-L. Waldspurger,
``La formule de Plancherel pour les groupes $p$-adiques (d'apr\`es Harish-Chandra)",
J. Inst. Math. Jussieu {\bf 2.2} (2003), 235--333

\bibitem[Yu]{Yu} J.-K. Yu, 
``Smooth models associated to concave functions in Bruhat--Tits theory", 
Panoramas et synth\`eses {\bf 47} (2015), 227--258

\end{thebibliography}
\end{document}